\documentclass{amsart}

\usepackage{thesis} 

\theoremstyle{plain}
\newtheorem{slogan}[thm]{Slogan}
\newtheorem*{slogan*}{Slogan}

\newcommand{\tmf}{tm\!f}
\newcommand{\mmf}{mm\!f}
\renewcommand{\ell}{{\textup{ell}}}
\newcommand{\HR}{{H \! R}}
\newcommand{\bbA}{\mathbb{A}}
\newcommand{\bbC}{\mathbb{C}}
\newcommand{\ttSp}{{\texttt{Sp}}}
\newcommand{\mot}{{\textup{mot}}}

\newcommand{\Derp}{{\mathscr{D}_{\geq 0}}}

\newcommand{\oO}{\mc{O}}
\newcommand{\der}{{\textup{der}}}

\newcommand{\Enplustwo}{{\textup{E}^{n+2}}}
\newcommand{\nd}{^\textup{nd}}

\begin{document}

\title{Goerss--Hopkins obstruction theory for $\infty$-categories}
\date{\today}
\author{\textsc{Aaron Mazel-Gee}}

\begin{abstract}
Goerss--Hopkins obstruction theory is a powerful tool for constructing structured ring spectra from purely algebraic data.
Using the formalism of model $\infty$-categories, we provide a generalization that applies in an arbitrary presentably symmetric monoidal stable $\infty$-category (such as that of equivariant spectra or of motivic spectra).
\end{abstract}

\maketitle


\setcounter{tocdepth}{1}
\tableofcontents


\setcounter{section}{-1}

\section{Introduction}

\subsection{Summary}

It has long been recognized that \bit{structured ring spectra} -- in particular, $\bbA_\infty$ and $\bbE_\infty$ ring spectra -- are of central importance in stable homotopy theory.  Indeed, such operadic structure allows for a good theory of modules, and is useful for importing such fundamental algebraic constructions as Hochschild homology and algebraic K-theory to the world of ring spectra.  In a different direction, it also induces rich additional structure on the corresponding cohomology theory, namely that of power operations.

Some spectra admit enhancements to structured ring spectra for transparent reasons.  For instance, the Eilenberg--Mac Lane spectrum $\HR$ of an associative (resp.\! commutative) ring $R$ carries an $\bbA_\infty$ (resp.\! $\bbE_\infty$) structure essentially by construction, as the functor $H : \Ab \ra \Sp$ is lax symmetric monoidal.  And the real and complex K-theory spectra carry $\bbE_\infty$ structures as a result of the fact that the tensor product of vector bundles 
is associative and commutative up to natural isomorphism.

\bit{Goerss--Hopkins obstruction theory} is a tool for constructing a much broader class of structured ring spectra.  This has found many crucial uses in the study of structured ring spectra: its first application \cite{GH-moduli-spaces} was to show that the Morava $E$-theory spectra admit essentially unique $\bbE_\infty$ structures and to compute their automorphisms; perhaps its most spectacular application to date is the construction of the $\bbE_\infty$ ring spectrum $\tmf$ of topological modular forms \cite{tmfbook}; it is a key ingredient in Galois theory for $\bbE_\infty$ ring spectra \cite{RogGal}; and it plays an important role in a number of other works such as \cite{Szymik-K3,Westerland-J,TopEqEinftyDGAs,LawNau-BP,GL-Brauer,LawsonCalc,Roth-HopfGal}.

It would be highly desirable to have a more general version of Goerss--Hopkins obstruction theory.  In particular, this should apply in the settings of equivariant and motivic stable homotopy theory, as well as in the setting of diagrams of spectra (e.g.\! sheaves of spectra (e.g.\! over the moduli stack of elliptic curves)).  The recent work \cite{HL-tmf-level} indicates the expected utility of an obstruction theory for logarithmic ring spectra.

The purpose of the present paper is the construction of just such a generalized obstruction theory.

\begin{slogan}\label{main-slogan}
There is a Goerss--Hopkins obstruction theory for any presentably symmetric monoidal stable $\infty$-category.
\end{slogan}

\noindent We will explain \Cref{main-slogan} in precise detail in \Cref{section.informal.overview}.

\begin{rem}
In forthcoming work \cite{MG-Enmot}, we use this obstruction theory to produce $\bbE_\infty$ structures on the \textit{motivic} Morava $E$-theory spectra and compute their automorphisms.  This is a first step in constructing a motivic spectrum $\mmf$ of \bit{motivic modular forms}, in analogy with $\tmf$.\footnote{The works \cite{Ricka-mmf,GIKR-mmf} take a different approach, producing motivic spectra over $\bbR$ and $\bbC$ whose cohomologies coincide with that expected of $\mmf$ (in analogy with $\tmf$).  These constructions are indirect, and relatively specific to the chosen base fields; in particular, the resulting motivic spectra are not manifestly related to any theory of elliptic motivic spectra.}  As the construction of $\tmf$ has been highly influential in chromatic homotopy theory, so would the construction of $\mmf$ significantly advance the chromatic approach to motivic homotopy theory, which is a highly active area of research \cite{Voe-ICM,HuKriz-MGL,Vezz-BP,Borgh-K,Horn-chrommot,LevMor-MGL,PPR-univ,NSO,NSO-nonregular,BalmerSpectra,Isaksen-wantmmf,Isaksen-htpyofCmmf,Andrews-families,Hoyois-MGL,Horn-nilp,Joachimi-thick,HO-PrimesMot,Gheorge-exotic}.
\end{rem}

\begin{rem}
There has been much recent interest in ``genuine'' operadic structures, e.g.\! genuine $G$-spectra with multiplications indexed by maps of finite $G$-sets (instead of just finite sets) \cite{BH-Ninfty,HH-Eqmult,BH-Gsm,BH-Tamb,Rubin-Ninfty,BP-eqop,GW-eqop}, as well as analogous structures in motivic homotopy theory \cite{BachHoy}.  We do not contend with such structures here, but we are optimistic that the obstruction theory we construct admits a fairly direct enhancement to one that would handle them in a formally analogous way.
\end{rem}

\subsection{Model categories and $\infty$-categories}
\label{subsection.model.cats.and.infty.cats}

Despite the demand and evident utility, Goerss--Hopkins obstruction theory has thus far resisted generalization.  This is not without cause, however.  Its construction is based in a carefully chosen model category of spectra -- let us denote it here by $\ttSp$ --, and rests on a plethora of delicate foundational assumptions surrounding that choice (see \cite[Theorems 1.2.1 and 1.2.3]{GH}).  These assumptions greatly simplify the arguments; for instance, they guarantee (among many other things) that the homotopy theory $\Alg_{\bbE_\infty}(\Spectra)$ of $\bbE_\infty$ ring spectra is presented by the model category $\Alg_{\Comm}(\ttSp)$ of strict commutative algebra objects in that model category.  Thus, a direct generalization of Goerss--Hopkins obstruction theory e.g.\! to the motivic setting would involve obtaining a model category $\ttSp^\mot$ of motivic spectra sharing these same point-set features and then proceeding from there.\footnote{Indeed, \cite{Horn-Pre} provides such a model category of motivic spectra, but this first step towards a motivic Goerss--Hopkins obstruction theory has not been carried further.}

On the other hand, further reflection reveals that such a direct approach is actually less than ideal.  After all, this would require a new argument for each distinct homotopy theory $\C$ in which one wants to obtain a version of Goerss--Hopkins obstruction theory -- or at least, it would require the establishment of a suitable model category $\ttC$ presenting the homotopy theory $\C$.  On the other hand, the obstruction theory itself is completely independent of the ambient choice of model category: it only depends on the \textit{underlying $\infty$-category}.  Thus, the more robust approach to obtaining a generalized Goerss--Hopkins obstruction theory is to dispense with such irrelevant point-set technicalities and work at the level of $\infty$-categories.  This is the approach that we pursue here.\footnote{In private communication, Goerss has explained that there were two reasons that the culminating Goerss--Hopkins paper \cite{GH} was never published.  Firstly, they envisioned a more comprehensive version of the obstruction theory that would apply not just to spectra but to diagrams of spectra (in particular sheaves of spectra over the moduli stack of elliptic curves, towards constructing $\tmf$), but they never managed to work this out.  And secondly, they came to realize that the then-nascent theory of $\infty$-categories would be able to elegantly handle the various technical problems with which they had wrestled.}

\subsection{Model $\infty$-categories}
\label{subsection.model.infty.cats}

As it turns out, however, dispensing with point-set technicalities is not the same thing as dispensing with model structures.  Indeed, the original construction of Goerss--Hopkins obstruction theory takes place in the \bit{resolution model structure} on the category $s\ttSp$ of simplicial spectra.  This presents the \bit{nonabelian derived $\infty$-category} of spectra, which we denote by $\Derp(\Spectra)$.  Correspondingly, Goerss--Hopkins obstruction theory for a more general $\infty$-category $\C$ takes place in its nonabelian derived $\infty$-category $\Derp(\C)$.

On the other hand, the arguments necessary for setting up Goerss--Hopkins obstruction theory do \textit{not} just take place in the nonabelian derived $\infty$-category $\Derp(\Spectra)$.  Rather, they make essential use of the resolution model structure itself, which is necessary for computing hom-spaces therein.  The reason for this is that the nonabelian derived $\infty$-category $\Derp(\Spectra)$ enjoys a universal property \textit{as an $\infty$-category}.  This is one category-level removed from the hom-spaces themselves, and as a result the latter are a priori quite difficult to describe in explicit terms.\footnote{Perhaps the simplest example of this phenomenon arises in the groupoid completion of a one-object category, which corresponds to the group completion of the corresponding monoid.  This groupoid is easy to characterize in terms of its universal property, but it is hopelessly difficult to describe in concrete terms: this is an intractable (in fact, computationally undecidable) task, closely related to the so-called ``word problem'' for generators and relations in abstract algebra.}

Thus, in order to make computations within the nonabelian derived $\infty$-category $\Derp(\C)$, we apply the theory of \bit{model $\infty$-categories}, which we developed in previous work \cite{MIC-sspaces,MIC-rnerves,MIC-gr,MIC-hammocks,MIC-qadjns,MIC-fundthm} for this purpose.  Namely, in this paper we construct a resolution model structure on the \textit{$\infty$-category} $s\C$ of simplicial objects in $\C$; the fundamental theorem of model $\infty$-categories then implies that we can use this model structure to compute hom-spaces in the ($\infty$-categorical) localization $\Derp(\C) \simeq \loc{s\C}{\bW_\res}$ at its subcategory of weak equivalences.

In this paper, we make free use of the theory of model $\infty$-categories.  Given a working knowledge of the classical theory of model categories, the terminology and the main theorems surrounding model $\infty$-categories are all as one would expect, though of course the proofs are substantially more involved; we refer the reader to \cite[\sec 0.2]{AMG-thesis} for a quick overview.  However, we note here that the central role of the model category $s\Set_\KQ$ of simplicial sets equipped with the Kan--Quillen model structure (e.g.\! in the definition of a simplicial model category) is played by the model $\infty$-category $s\S_\KQ$ of simplicial \textit{spaces} equipped with a likewise Kan--Quillen model structure -- both of which present the $\infty$-category $\S$ of spaces.

\subsection{Conventions}

\citelurie \ \luriecodenamesforGHOsT

\butinvt \ \seeappendixforGHOsT

\subsection{Outline}

We now provide a more detailed outline of the contents of this paper.
\begin{itemize}
\item In \Cref{section.informal.overview}, we provide an informal overview of our generalized Goerss--Hopkins obstruction theory.
\item In \Cref{section.resn.model.str}, we introduce resolution model structures on $\infty$-categories of simplicial objects, and give sufficient conditions for their existence.
\item In \Cref{section.topology}, we lay out our foundations and assumptions regarding the ambient presentably symmetric monoidal stable $\infty$-category $\C$, and we construct an auxiliary resolution model structure on the $\infty$-category $s\C$.
\item In \Cref{section.alg.top}, we add operadic structures into the story: if our main goal is to construct algebras in $\C$ over some operad $\oO \in \Op$, we obtain a simplicial resolution $T \in s\Op$ of $\oO$ and lift the above resolution model structure to one on the $\infty$-category $\Alg_T(s\C)$ of $T$-algebras in $s\C$.
\item In \Cref{section.algebra}, we turn to the algebraic part of the story, introducing a certain category $\A$ of comodules and positing monads on the categories $\A$ and $s\A$ that respectively govern the structures present on the homologies of $\oO$-algebras in $\C$ and of $T$-algebras in $s\C$.
\item In \Cref{section.homotopical.algebra}, we study Postnikov theory, Andr\'e--Quillen cohomology, and moduli spaces in the model $\infty$-category $s\A$.
\item In \Cref{section.homotopical.topology}, we study Postnikov theory, Andr\'e--Quillen cohomology, and moduli spaces in the model $\infty$-category $s\C$.
\item In \Cref{section.decomp.of.mod.spaces}, we prove our main theorems.
\end{itemize}

\subsection{Acknowledgments}
\label{subsection.ack}


This project was born purely by chance, on a train ride that I happened to share with Markus Spitzweck in late 2012, during which he introduced me to the world of motivic homotopy theory and first piqued my interest in the idea of producing a motivic Goerss--Hopkins obstruction theory (and, someday, motivic modular forms!).  It is a pleasure to thank him for his inspiration and collaboration.  I would also like to thank Dave Carchedi and Justin Noel for their friendship and continued mathematical support in those early days of this project back in Bonn.

The next phase of this project took place during my time at MIT.  I am grateful to many people who helped this project along during that period: Mark Behrens, for believing in me, and for generously sharing his time and expertise; Gijs Heuts, for convincing me to ditch the hunt for a suitable model category of motivic spectra and instead learn to work with $\infty$-categories; Omar Antol\'in Camarena and Gijs Heuts (again), for their abundant and cheerful enthusiasm throughout our many conversations, even as we failed over and over to really understand the spiral exact sequence and reconstruct it homotopy-invariantly; and Dustin Clausen, Jacob Lurie, Akhil Mathew, and Clark Barwick and the Bourbon Seminar for their many helpful insights.

I would like to express my deepest thanks to Paul Goerss, Mike Hopkins, and Haynes Miller for working out such a beautiful and compelling story, and to Paul Goerss in particular for his many patient and thorough explanation of its finer points.  David Gepner has also been an essential resource throughout the duration of this project.

Lastly, I would like to thank the Max Planck Institute, MIT, the NSF graduate research fellowship program (grant DGE-1106400), UC Berkeley's geometry and topology RTG (grant DMS-0838703), Philz, and 1369 for their hospitality and/or financial support during the time that this work was carried out.

\section{Informal overview}
\label{section.informal.overview}

Suppose we are given a flat homotopy commutative ring spectrum
\[ E \in \CAlg(\ho(\Sp)) \]
satisfying \textit{Adams's condition} (which we will describe in \Cref{subsection.detecing.homology.theory}); we will refer to the its corresponding homology theory $E_*$ as our ``detecting'' homology theory.  Suppose moreover that we are given a commutative algebra
\[ A \in \CAlg(\Comod_{(E_*,E_*E)}) \]
in comodules.
Then, Goerss--Hopkins obstruction theory provides a method for computing the \bit{moduli space of} (\bit{$E$-local}) \bit{realizations} of $A$ as an $\bbE_\infty$ ring spectrum -- the first question being whether it is nonempty.  In fact, it applies to algebras over \textit{any} operad $\oO$, though this changes the nature of the algebraic object in comodules in comodules that we must consider.

The purpose of the present section is to explain this story in detail, as well as the generalization from $\Sp$ to an arbitrary presentably symmetric monoidal stable $\infty$-category which is the purpose of this paper.  We begin by explaining the obstruction theory as a black box in \sec\sec \ref{subsection.moduli.space.of.realizations}-\ref{subsection.obstructions.to.realization}, focusing for simplicity only on the $\bbE_\infty$ case.  We then proceed to unpack the inner workings of the obstruction theory in \sec\sec \ref{subsection.informal.overview.nonab.der.infty.cats}-\ref{subsection.informal.overview.tower.of.moduli.spaces}.

\subsection{The moduli space of realizations}
\label{subsection.moduli.space.of.realizations}

First of all, a \bit{realization} of $A$ is an $\bbE_\infty$ ring spectrum $X$ for which there exists an isomorphism $E_*X \cong A$ (of algebras in comodules).  These are our objects of interest.  Note that we do \textit{not} require the existence of a spectrum realizing the underlying comodule of $A$: that is, we start with \textit{purely algebraic} data.

Next, an \bit{$E$-equivalence} is a map $X \ra Y$ of spectra that induces an isomorphism $E_* X \xra{\cong} E_* Y$ of $E_*E$-comodules (or equivalently of $E_*$-modules).  In a universal way, we can invert the $E$-equivalences in the $\infty$-category of spectra to form the $\infty$-category $\leftloc_E (\Sp)$ of \textit{$E$-local spectra}.  The terminology stems from the fact that this localization actually participates in a reflective localization
\[ \leftloc_E : \Sp \adjarr \leftloc_E(\Sp) : \forget_E , \]
i.e.\! an adjunction whose right adjoint is fully faithful; in particular, we can consider $\leftloc_E(\Sp) \subset \Sp$ as a full subcategory.
\footnote{This is the underlying $\infty$-categorical content of the theory of \textit{Bousfield localization} of spectra, as introduced in the classic paper \cite{Bousloc}.}  In other words, $E$-local spectra are just particular sorts of spectra, but $E$-equivalences between them are necessarily equivalences.

Finally, the \bit{moduli space of $E$-local realizations} of $A$ is the full subgroupoid
\[ \ms{M}_A \subset \CAlg(\leftloc_E(\Sp))  \]
on the $E$-local $\bbE_\infty$ ring spectra which are realizations of $A$; its morphisms are the $E$-equivalences (which are also equivalences) between them.  As indicated above, we will generally leave the descriptor ``$E$-local'' implicit.

\begin{rem}
Of course, this necessarily only produces $E$-local spectra.  Thus, if one is interested in obtaining an $\bbE_\infty$ ring structure on a particular spectrum $X \in \Sp$, one must choose a detecting homology theory $E_*$ for which $X$ is $E$-local.  On the other hand, this locality is not so hard to satisfy in practice: crucially, any $E$-module is necessarily $E$-local.  Note that this is a relatively weak (and in particular, unstructured) hypothesis: we have only assumed that $E$ is a \textit{homotopy} commutative ring spectrum, and thus by ``module'' we can only possibly mean an object $X \in \Mod_E (\ho(\Sp))$.

In particular, it follows that $E$ is $E$-local.  This implies the nearly unbelievable conclusion that if we would like to endow a homotopy commutative ring spectrum $E \in \CAlg(\ho(\Sp))$ with an $\bbE_\infty$-structure, then $E$ can \textit{itself} serve as the detecting homology theory!
\end{rem}




\subsection{Andr\'e--Quillen cohomology}

Given our commutative algebra $A$ in comodules, one can speak of \textit{modules} over $A$ (in comodules); we mention now that for any $n \geq 1$ one can define a canonical $A$-module $\Omega^n A$, which will play a role in our story shortly.  For any $A$-module $M$ and any augmented commutative algebra
\[ X \in \CAlg(\Comod_{(E_*,E_*E)})_{/A} \]
in comodules, we can define the corresponding \bit{Andr\'e--Quillen cohomology groups} $H^*(X;M)$.  In fact, these are given by the homotopy groups of a certain spectrum
\[ \ms{H}(X;M) = \{ \ms{H}^n(X;M) \}_{n \geq 0} , \]
in the sense that
\[ H^n(X;M) = \pi_{-n} \ms{H}(X;M) \cong \pi_0 \ms{H}^n(X;M) \cong \pi_1 \ms{H}^{n+1}(X;M) \cong \cdots \]
for any $n \geq 0$ (or really for any $n \in \bbZ$: this spectrum has vanishing positive-dimensional homotopy groups, not unlike $\enrhom_\Sp(\Sigma^\infty_+ X , E)$ for any $X \in \S$ and any $E \in \Sp$).  The group
\[ \Aut(A,M) \]
of automorphisms of the pair $(A,M)$ (whose elements are pairs of an isomorphism $\varphi : A \xra{\cong} A$ and an isomorphism $M \ra \varphi^* M$) naturally acts on this spectrum.  In particular, it acts on each constituent space $\ms{H}^n(X;M)$, and we write
\[ \what{\ms{H}}^n(X;M) = \left( \ms{H}^n(X;M) \right)_{\Aut(A,M)} \]
for the (homotopy) quotient.  This action fixes the basepoint of $\ms{H}^n(X;M)$ (whose path component corresponds to the zero element $0 \in H^n(X;M)$), and so the inclusion of the basepoint is $\Aut(A,M)$-equivariant and hence determines a map
\[ B \Aut(A,M) \ra \what{\ms{H}}^n(X;M) \]
on quotients.  We note for future reference that this map, whose source is connected, lands entirely in the path component selected by the composite
\[ \pt \xra{0} \ms{H}^n(X;M) \ra \what{\ms{H}}^n(X;M) . \]

\subsection{Obstructions to realization}
\label{subsection.obstructions.to.realization}


As we will describe in more depth in \Cref{subsection.informal.overview.tower.of.moduli.spaces}, our understanding of the moduli space $\ms{M}_A$ actually comes from a sequence of moduli spaces $\ms{M}_n(A)$ of ``$n$-stage approximations'' to a realization of $A$.  These moduli spaces are related by pullback squares
\[ \begin{tikzcd}
\ms{M}_n(A) \arrow{r} \arrow{d} & B\Aut(A,\Omega^n A) \arrow{d} \\
\ms{M}_{n-1}(A) \arrow{r} & \what{\ms{H}}^{n+2}(A;\Omega^n A)
\end{tikzcd} \]
(for all $n \geq 1$), in which the left vertical map is induced by an ``$(n-1)\st$ Postnikov trunction'' functor and the lower map is induced by an ``$n\th$ k-invariant'' functor
\[ \ms{M}_{n-1}(A) \xra{\chi_n} \ms{H}^{n+2}(A;\Omega^n A) . \]
Moreover, we have a canonical identification
\[ \ms{M}_A \xra{\sim} \lim \left( \cdots \ra \ms{M}_2(A) \ra \ms{M}_1(A) \ra \ms{M}_0(A) \right) \]
of our moduli space of realizations as the limit of the resulting tower.  Finally, as the base for our inductive understanding, we have an equivalence
\[ \ms{M}_0(A) \simeq B \Aut(A) . \]

We can now describe the sense in which we can ``compute'' the moduli space $\ms{M}_A$.  Observe that the above pullback square implies that an $(n-1)$-stage $X$ can be lifted to an $n$-stage if and only if the k-invariant
\[ [ \chi_n (X) ] \in H^{n+2}(A;\Omega^n A) \]
vanishes: this is the only case in which there exists a nonempty fiber in the diagram
\[ \begin{tikzcd}
& & & B \Aut(A,\Omega^n A) \arrow{d} \\
\pt \arrow{r}[swap]{X} & \ms{M}_{n-1}(A) \arrow{r}[swap]{\chi_n} & \ms{H}^{n+2}(A;\Omega^n A) \arrow{r} & \what{\ms{H}}^{n+2}(A;\Omega^n A) ,
\end{tikzcd} \]
which is necessary and sufficient for there to exist a nonempty fiber in the diagram
\[ \begin{tikzcd}
& \ms{M}_n(A) \arrow{d} \\
\pt \arrow{r}[swap]{X} & \ms{M}_{n-1}(A) .
\end{tikzcd} \]

\begin{rem}
Of course, this is most useful in the 
\'etale case, i.e.\! when the relevant Andr\'e--Quillen cohomology groups all vanish.  Under this assumption, the entire tower collapses to an equivalence
\[ \ms{M}_A \xra{\sim} \ms{M}_0(A) \simeq B \Aut(A) . \]
This is visibly the case with Goerss--Hopkins's original application to the Morava $E$-theories.  In fact, after enough algebraic manipulation, it also becomes the case in the construction of the sheaf $\oO^\der$
of $\bbE_\infty$ ring spectra over the moduli stack of elliptic curves, whose global sections are $\tmf$ 
(but these manipulations are themselves not completely trivial).

In fact, this is also the case in another prominent application of Goerss--Hopkins obstruction theory as well.  In his inspiring monograph \cite{RogGal}, Rognes develops the \textit{Galois theory} of $\bbE_\infty$ ring spectra.  This may be seen as the study of \textit{covering spaces} among affine spectral schemes, and provides a remarkably effective framework for the organization of chromatic homotopy theory from the viewpoint of spectral algebraic geometry.  Just as classical Galois theory, this is governed by a Galois correspondence, i.e.\! a contravariant equivalence of posets.  In order to prove this fundamental theorem, Rognes uses Goerss--Hopkins obstruction theory to obtain the desired intermediate Galois extension from a subgroup of the Galois group.
\end{rem}


\subsection{Nonabelian derived $\infty$-categories and resolution model structures}
\label{subsection.informal.overview.nonab.der.infty.cats}


We now explain what exactly we
meant
by the notation ``$\Derp(\C)$'' used in \Cref{subsection.model.infty.cats}.  In fact, this notation is slightly misleading: this construction does not depend on the $\infty$-category $\C$ alone.  Rather, we must first choose a full subcategory $\G \subset \C$ which is closed under finite coproducts, which should be thought of as a subcategory of ``projective generators''.  Out of this, we define the (\bit{nonnegatively-graded}) \bit{nonabelian derived $\infty$-category} of $\C$ to be
\[ \ms{D}_{\geq 0}(\C) = \ms{D}_{\geq 0}(\C,\G) = \PS(\G) = \Fun^\times(\G^{op},\S), \]
the $\infty$-category of \textit{product-preserving presheaves of spaces} on $\G$, i.e.\! the full subcategory of $\Fun(\G^{op},\S)$ on those contravariant functors that take finite coproducts in $\G$ to finite products in $\S$.  (We will use the various notations interchangeably, depending on our desired emphasis.)

Observe that there is a canonical functor
\[ s\C \xra{ X_\bullet \mapsto \left( Y \mapsto |\hom^\lw_\C(Y,X_\bullet)| \right) } \PS(\G) \]
from the $\infty$-category of simplicial objects in $\C$, the levelwise restricted Yoneda functor followed by geometric realization. 
In the case that $\C$ is an ordinary category and $\G \subset \C$ is a full subcategory of small projective generators, in 
\cite[\sec II.4]{QuillenHA} Quillen defined a model structure on $s\C$ which (in hindsight) is precisely a presentation of the $\infty$-category $\PS(\G)$.  For example, if we take $\C = \Set$ to be the category of sets and $\G = \Fin$ to be the full subcategory of finite sets, this recovers the standard Kan--Quillen model structure $s\Set_\KQ$, which presents the $\infty$-category
\[ \Derp(\Set) = \Derp(\Set,\Fin) = \Fun^\times(\Fin^{op},\S ) \simeq \Fun(\pt^{op},\S) \simeq \S \]
of spaces as the nonabelian derived $\infty$-category of the category of sets.  On the other hand, if $\C$ is an abelian category, then $\Derp(\C)$ recovers the usual (nonnegatively-graded) derived $\infty$-category of $\C$, which through the Dold--Kan correspondence is equivalent to the usual definition in terms of nonnegatively-graded chain complexes in $\C$.
In general, cofibrant replacements in these model structures may thus be thought of as \bit{nonabelian projective resolutions}.

In fact, this same idea has been carried further in homotopy theory.  In \cite{DKS-E2}, Dwyer--Kan--Stover defined a \bit{resolution model structure} on the category $s\Top_*$ of simplicial pointed topological spaces based on the set of generators
\[ \{S^n \in \Top_* \}_{n \geq 1} , \]
and in \cite{BousCosimp} Bousfield generalized this to a general (pointed, right proper) model category equipped with a set of h-cogroup objects satisfying certain conditions.  In both cases, the restriction to h-cogroup objects is motivated by the desire for spectral sequences converging to the ``homotopy groups'' (with respect to the generators and their finite coproducts) of the geometric realization of an object (in the model-categorical sense).  The levelwise weak equivalences are weak equivalences in these model structures, but there are strictly more of the latter.

From the perspective of model $\infty$-categories, it is clear that these model 1-categories are fairly inefficient: it is wholly unnecessary to distinguish between objects which are levelwise weakly equivalent.  On the other hand, the resolutions that these model structures afford \textit{are} necessary -- indeed, they are the entire point.  Thus, one might expect to freely invert the levelwise weak equivalences while keeping track of the remaining resolution weak equivalences.  To this end, we have the following theorem.

\begin{thm}[\ref{many-object sNres} \and \ref{identify localization of sNres}]\label{thm resn model str of intro}
Let $\C$ be a presentable $\infty$-category, let $\{ Z_\alpha \in \C \}$ be a set of compact objects, and write $\G \subset \C$ for the full subcategory generated by the objects $Z_\alpha$ and their finite coproducts.  Then there exists a \bit{resolution model structure} on the $\infty$-category $s\C$, denoted $s\C_\res$.  This model structure is simplicial (i.e., it is compatibly enriched over $s\S_\KQ$).  Moreover, it participates in a Quillen adjunction
\[ \Fun(\G,s\S_\KQ)_\projective \adjarr s\C_\res , \]
whose derived adjunction 
is precisely the canonical adjunction
\[ \P(\G) \adjarr \PS(\G) . \]
\end{thm}

\begin{rem}
The resolution model $\infty$-category structure of
\Cref{thm resn model str of intro}
is indeed much more efficient than its 1-categorical analogs.  For example, every object in $s\C_\res$ is fibrant; by contrast, in the resolution model structures of Dwyer--Kan--Stover and Bousfield, the fibrant objects are precisely the Reedy fibrant objects.  (This is by no means a decisive advantage, but it seems worth pointing out nonetheless.)
\end{rem}


\subsection{The detecting homology theory and resolutions}
\label{subsection.detecing.homology.theory}

Let us fix
a presentably symmetric monoidal stable $\infty$-category $\C$.  This replaces a model 1-category of spectra, which in the original construction
of Goerss--Hopkins obstruction theory
must be assumed to satisfy a long list of technical assumptions.  We assume that $\C$ is equipped with a full subcategory $\G \subset \C$ of generators, which we assume to be sufficiently nice (e.g.\! its objects must all have inverses with respect to the symmetric monoidal structure -- thereafter, our assumptions will imply that its objects are compact).  This generalizes the set of sphere spectra.  These generators define a ``homotopy groups'' functor $\pi_\star$.

We now discuss our detecting homology theory, which we assume to be given by a flat homotopy commutative algebra $E \in \CAlg(\ho(\C))$.  We can now explain the all-important \bit{Adams's condition}. 
This is the requirement that $E$ be obtainable as a filtered colimit
\[ \colim_\J E_\alpha \xra{\sim} E \]
of \textit{dualizable} objects $E_\alpha$, such that their duals $\Dual E_\alpha$ have projective $E$-homology.  This condition allows us to treat $E$-homology as being given by ``homotopy groups with respect to these duals''.  More precisely, our assumptions guarantee that for any generator $S^\beta \in \G$ we have a string of isomorphisms
\begin{align*}
\colim_{\alpha \in \J} [ \Sigma^\beta \Dual E_\alpha , X]_\C & \cong  \colim_{\alpha \in \J} [S^\beta , E_\alpha \otimes X]_\C
\cong [S^\beta , \colim_{\alpha \in \J}(E_\alpha \otimes X) ]_\C \\
& \cong [S^\beta , \colim_{\alpha \in \J}(E_\alpha) \otimes X ]_\C
\cong [S^\beta , E \otimes X ]_\C
= E_\beta X
\end{align*}
(where we suggestively write $\Sigma^\beta$ for the functor $S^\beta \otimes -$).  Therefore, if a map $X \ra Y$ induces ``$\Dual E_\alpha$-homotopy'' isomorphisms
\[ [ \Sigma^\beta \Dual E_\alpha , X ]_\C \xra{\cong} [ \Sigma^\beta \Dual E_\alpha,Y ]_\C \]
for all $S^\beta \in \G$ and all $\alpha \in \J$, then it induces an isomorphism on $E$-homology.  On the other hand, the converse will not generally hold.  This subtlety can be handled with a little bit of care (or with a lot of care, in the original model 1-categorical case), and we will return to it in due time.

Let us write $\GE \subset \C$ for the full subcategory generated by the subcategory $\G$ and the objects $\Dual E_\alpha$ under finite coproducts.  Then, our resolutions will be based on the nonabelian derived $\infty$-category
\[ \ms{D}_{\geq 0} ( \C , \GE ) . \]
However, we will need to make computations using actual simplicial resolutions (i.e.\! objects of $s\C$) instead of their images under the functor
\[ s\C \ra \PS(\GE) = \ms{D}_{\geq 0}(\C,\GE) , \]
and for this we will use the resolution model structure provided by \Cref{thm resn model str of intro}.

As we will explain 
in \Cref{subsection.operadic.structures},
we will not actually be using this model $\infty$-category directly, but rather a generalization of it.  However, even in this special case we can point out an essential feature of the story.  Let us write $\tA$ for the category of $(E_\star,E_\star E)$-comodules, and let us write $\GA \subset \tA$ for the full subcategory on objects of the form $E_\star S^\eps$ for some $S^\eps \in \GE$; by our assumptions, these will be projective as $E_\star$-modules.  As we have assumed that $\C$ is presentably symmetric monoidal, it follows that the induced functor
\[ E_\star : \GE \ra \GA \]
preserves finite coproducts.  It follows formally that the induced functor
\[ \PS(E_\star) : \PS(\GE) \ra \PS(\GA) \]
preserves all colimits.  Ultimately, this fact will be (a shadow of) the reason that our topological obstructions can be computed purely algebraically.  At the level of model $\infty$-categories, this can be seen as resulting from the fact that the functor
\[ E_\star^\lw : s\C_\res \ra s\tA_\res \]
preserves cofibrations between cofibrant objects relative to an analogous resolution model structure on $s\tA_\res$.

\subsection{Operadic structures and resolutions}
\label{subsection.operadic.structures}

We use the term ``operad'' to refer to a (single-colored) $\infty$-operad; the $\infty$-category $\Op$ of operads is presented by the relative category $\Op(s\Set_\KQ)$
of operads in simplicial sets, 
whose weak equivalences are determined levelwise on underlying objects (i.e.\! ignoring the symmetric group actions).  This relative category structure enhances to a \textit{Boardman--Vogt} model structure, which (using a generalization of \Cref{thm resn model str of intro}) we incidentally generalize to the $\infty$-category $\Op(s\V)$ of internal operads (for a suitable symmetric monoidal $\infty$-category $\V$) as \Cref{model structure on sV-operads}.

Now, our obstruction theory can be used to construct ($E$-local) $\oO$-algebras in $\C$, for any operad $\oO \in \Op$.  Given a choice of $\oO$, however, we must choose a monad $\Phi$ on $\tA$ which will parametrize our ``algebraic structures'': in other words, we must have a lift
\[ \begin{tikzcd}
\Alg_\oO(\C) \arrow[dashed]{r}{E_\star} \arrow{d}[swap]{\forget_\oO} & \Alg_\Phi(\tA) \arrow{d}{\forget_\Phi} \\
\C \arrow{r}[swap]{E_\star} & \tA
\end{tikzcd} \]
of our $E$-homology functor.  For instance, in the special case where $\oO = \Comm = \bbE_\infty$ that we described in \sec\sec \ref{subsection.moduli.space.of.realizations}-\ref{subsection.obstructions.to.realization}, we also took $\Phi = \Comm$.  However, even in the case that we take $\oO = \Comm$, it can be useful -- essential, even -- to have this added generality.\footnote{The construction of $\tmf$ (as the global sections of a sheaf of $\bbE_\infty$ ring spectra over the moduli stack of elliptic curves),
which was spelled out in full detail by Behrens in \cite{tmfbook}, makes essential use of such generality.  In order to construct the height-1 component of the sheaf (which is necessary in order to ``interpolate'' between the supersingular loci at distinct primes, and which is by far the most technical aspect of the construction), one must take the $p$-adic complex $K$-theory spectrum $KU^\sm_p$ as the detecting homology theory, and one must enhance the nature of the algebraic input from a commutative algebra in comodules to what is called a \textit{$\theta$-algebra} (which structure is canonically present on the $p$-adic $K$-theory of an $\bbE_\infty$ ring spectrum).}

So of course, we will not be interested in resolving objects of $\C$, but rather objects of $\Alg_\oO(\C)$.  However, it will not suffice to simply resolve them by \textit{simplicial} objects of $\Alg_\oO(\C)$: at no point will this allow us to gain control over their levelwise $E$-homology (in the model category $s\tA_\res$).

On the other hand, there is a special case in which this does hold, namely when the operad $\oO$ is \textit{\PSG-free}: by definition, this means that for every $n \geq 0$, the symmetric group $\SG_n$ acts freely on the set $\pi_0 (\oO(n))$ of path components of the $n\th$ constituent space of $\oO$.  When this is the case, the ``free $\oO$-algebra'' functor
\[ X \mapsto \coprod_{n \geq 0} ( \oO(n) \tensoring X^{\otimes n})_{\SG_n} \]
simplifies dramatically.  Even better, if we assume that $E_\star X$ is projective -- such as when $X = \Dual E_\alpha$ --, then the K\"unneth spectral sequence for the $E$-homology of this free $\oO$-algebra (which is guaranteed by Adams's condition) immediately collapses!

Thus, a key insight of Goerss--Hopkins obstruction theory (over its predecessors) was, for a general operad $\oO$, to take a simplicial resolution $T_\bullet \in s\Op$ by \PSG-free operads.  Amusingly, this can be achieved by choosing a cofibrant representative of $\oO$ in the model category $\Op(s\Set_\KQ)_\BV$ via the embedding
\[ \Op(s\Set) \simeq s \left( \Op(\Set) \right) \hookra s\Op . \]
A simplicial operad can be made to act on simplicial objects in $\C$, and from here we obtain (as \Cref{thm AlgTsCres}) a lifted resolution model structure through the adjunction
\[ \free_T : s\C_\res \adjarr \Alg_T(s\C)_\res : \forget_T \]
This is the model $\infty$-category we have been seeking.  On the one hand, its objects are resolutions of $\oO$-algebras in $\C$: we have a canonical lift
\[ \begin{tikzcd}
\Alg_T(s\C) \arrow[dashed]{r}{|{-}|} \arrow{d}[swap]{\forget_T} & \Alg_\oO(\C) \arrow{d}{\forget_\oO} \\
s\C \arrow{r}[swap]{|{-}|} & \C
\end{tikzcd} \]
of the geometric realization functor.  On the other hand, we will assume enough so that there is a monad $\tT_E$ on $s\tA$ admitting a lift
\[ \begin{tikzcd}
\Alg_T(s\C) \arrow[dashed]{r}{E_\star^\lw} \arrow{d}[swap]{\forget_T} & \Alg_{\tT_E}(s\tA) \arrow{d}{\forget_{\tT_E}} \\
s\C \arrow{r}[swap]{E_\star^\lw} & s\tA .
\end{tikzcd} \]
Just as our unstructured functor
\[ E_\star^\lw : s\C_\res \ra s\tA_\res \]
preserves cofibrations between cofibrant objects, so will this lifted functor $E_\star^\lw$ (with respect to an analogously lifted resolution model structure $\Alg_{\tT_E}(s\tA)_\res$), which crucially implies that its localization
\[ E_\star^\lw : \locresAlgT \ra \locresAlgtTE \]
preserves colimits.  Although there will be one more small wrinkle that must be smoothed out, this fact is very nearly the true reason that our topological obstructions can be computed purely algebraically.

\subsection{$E_\star$-localization}

Given our algebraic object $A \in \Alg_\Phi(\tA)$, we can now explain that our ``$n$-stage approximations'' to $A$ will be objects of the $\infty$-category $\locresAlgT$, and our Andr\'e--Quillen cohomology spaces will be certain mapping spaces extracted from the $\infty$-category $\locresAlgtTE$.  However, these facts are technically true but slightly misleading.

To clarify both at once, let us recall for the sake of analogy that in the $\infty$-category $\C$, a map becoming an isomorphism under all of functors $[\Sigma^\beta \Dual E_\alpha , - ]_\C$ implies that it also becomes an isomorphism under the functor $E_\star$, but that the converse is generally false.  Then, in the algebraic case, note that there exists a forgetful functor
\[ \Alg_{\tT_E}(s\tA) \xra{\forget_{\tT_E}} s\tA \xra{s(\forget_{\tA})} s\Set_* , \]
which takes the subcategory $\bW_\res \subset \Alg_{\tT_E}(s\tA)$ into the subcategory $\bW_\KQ \subset s\Set_*$, but not only this subcategory; defining
\[ \bW_{\pi_*} \subset \Alg_{\tT_E}(s\tA) \]
to be the pullback of $\bW_\KQ \subset s\Set_*$, we obtain a reflective localization
\[ \locresAlgtTE \adjarr \locpiAlgtTE . \]
Similarly, in the topological case, the functor
\[ E_\star^\lw : \Alg_T(s\C) \ra \Alg_{\tT_E}(s\tA) \]
takes the subcategory $\bW_\res \subset \Alg_T(s\C)$ into the subcategory $\bW_{\pi_*} \subset \Alg_{\tT_E}(s\tA)$, but not only this subcategory; defining
\[ \bW_{E_\star^\lw} \subset \Alg_T(s\C) \]
to be the pullback of $\bW_{\pi_*} \subset \Alg_{\tT_E}(s\tA)$, we obtain a reflective localization
\[ \locresAlgT \adjarr \locEAlgT . \]

Now, we can clarify that in that the moduli spaces of $n$-stages for $A$ are naturally subgroupoids
\[ \ms{M}_n(A) \subset \locEAlgT \subset \locresAlgT \]
of the reflective localization, while the relevant Andr\'e--Quillen cohomology spaces are computed by mapping in $\locresAlgtTE$ to an object of the reflective subcategory $\locpiAlgtTE \subset \locresAlgtTE$.  Moreover, these two reflective localization functors participate as the downwards arrows in a commutative square
\[ \begin{tikzcd}
\locresAlgT \arrow{r}{E_\star^\lw} \arrow{d} & \locresAlgtTE \arrow{d} \\
\locEAlgT \arrow[dashed]{r}[swap]{E_\star^\lw} & \locpiAlgtTE ,
\end{tikzcd} \]
in which the dotted arrow exists by the universal property of localization and preserves colimits by an easy diagram chase.  This, finally, is the \textit{true} reason that our topological obstructions can be computed purely algebraically.  However, in order to explain this, we must introduce the \textit{spiral exact sequence}.

\subsection{Bigraded $E$-homology groups and the spiral exact sequence}

Given a simplicial object $X \in s\C$, there are two sorts of $E$-homology groups that one might extract: the \textit{classical} $E$-homology groups
\[ \pi_n E_\beta^\lw X = \pi_n [ S^\beta , E \otimes X]^\lw_\C \]
and the \textit{natural} $E$-homology groups
\[ E_{n,\beta}^\natural X = \pi_n \left( \hom_{\ms{D}_{\geq 0}(\C,\GE)}(S^\beta , (E \otimes X)^\lw) \right) . \]
These serve dual purposes.

On the one hand, the classical $E$-homology groups assemble into the $\Etwo$ page of a spectral sequence
\[ \Etwo = \pi_n E_\beta^\lw X \Rightarrow \Einfty = E_{\beta+n}|X| , \]
where we write $S^{\beta+n} = S^\beta \otimes S^n = \Sigma^n S^\beta$.  Of course, this spectral sequence allows us to obtain control over the $E$-homology of the geometric realization $|X|$.

On the other hand, the natural $E$-homology groups are by their very definition much more directly related to the $\infty$-category
\[ \ms{D}_{\geq 0}(\C,\GE) \simeq \locressC . \]
Thus, they participate in a ``cells and disks'' obstruction theory within this $\infty$-category.  In order to explain this, we introduce the notation
\[ D^n_\bD = \Delta^n / \Lambda^n_0 \in (s\Set_*)_\KQ \]
and
\[ S^n_\bD = \Delta^n / \partial \Delta^n \in (s\Set_*)_\KQ . \]
There are evident cofibrations
\[ S^n_\bD \cofibn D^{n+1}_\bD \]
in $(s\Set_*)_\KQ$, which present the maps
\[ S^n \ra D^{n+1} \simeq \pt \]
in $\S_*$.  Moreover, for any $K \in s\S_*$ and any $X \in s\C$, there exists a ``based tensor'' object $K \redtensoring X \in s\C$, which is compatible with the canonical enrichment of $s\C$ over $s\S_*$ (where the basepoint is given by the zero morphism).  Writing $S^\eps \in \GE$ for an arbitrary object, the fact that the model $\infty$-category $s\C_\res$ is \textit{simplicial} implies that the ``cells'' given by
\[ S^n_\bD \redtensoring \const(S^\eps) \in s\C_\res \]
and the ``disks'' given by
\[ D^n_\bD \redtensoring \const(S^\eps) \in s\C_\res \]
together control the theory of \textit{Postnikov towers} in $\locressC$.\footnote{Examining the structure maps of the simplicial sets $D^n_\bD$ and $S^n_\bD$, one sees that they may be seen as corepresenting the \textit{nonabelian $n$-cycles} and \textit{nonabelian normalized $n$-chains} objects of an object $X \in s\C$ (via a ``based cotensor'' bifunctor $- \redcotensoring - : s\S_* \times s\C \ra \C$ which we will not make precise here).}

Now, the (``\bit{localized}'') \bit{spiral exact sequence} relates these two types of $E$-homology, running
\[ \begin{tikzcd}[row sep=0.1cm]
\cdots \arrow{r} & \pi_{i+1}E_\beta X \arrow{r}{\delta} & E_{i-1,\beta+1}^\natural X \arrow{r} & E_{i,\beta}^\natural X \arrow{r} & \pi_iE_\beta X \arrow{r}{\delta} & \cdots \\
& \cdots \arrow{r}{\delta} & E_{0,\beta+1}^\natural X \arrow{r} & E_{1,\beta}^\natural X \arrow{r} & \pi_1 E_\beta X \arrow{r} & 0 .
\end{tikzcd} \]
Note that it is two-thirds natural $E$-homology, and one-third classical $E$-homology.\footnote{In fact, these long exact sequences are what organize into the exact couple defining the above spectral sequence.}  Thus, via the spiral exact sequence, \textit{by controlling the natural $E$-homology groups} (via ``cells and disks'') \textit{we can also control the classical $E$-homology groups} (which assemble into the $\Etwo$ page of the spectral sequence).

\subsection{The tower of moduli spaces}
\label{subsection.informal.overview.tower.of.moduli.spaces}

We can now explain the connection with ``$n$-stages'' for our chosen object $A \in \Alg_\Phi(\tA)$ of which we are interested in realizations.  First of all, an \textit{$\infty$-stage} for $A$ is an object of $\locEAlgT$ whose $\Etwo$ page is simply given by $A$, concentrated in the bottom row; these assemble into a moduli space
\[ \ms{M}_\infty(A) \subset \locEAlgT . \]
We then have the following result, which cements the relationship between realizations of $A$ and their (approximate) resolutions.

\begin{thm}[\ref{infty-stages give realizations}]\label{intro thm realizing infty-stages}
The geometric realization functor
\[ |{-}| : \locEAlgT \ra \Alg_\oO(\leftloc_E(\C)) \]
induces an equivalence
\[ \ms{M}_\infty(A) \xra{\sim} \ms{M}_A . \]
\end{thm}

\noindent We emphasize that the moduli space $\ms{M}_\infty(A) \subset \locEAlgT$ will \textit{not} generally contain all of the objects whose geometric realizations are realizations of $A$: rather, it only contains those whose geometric realizations are realizations of $A$ ``for obvious reasons'' (namely that their spectral sequences collapse immediately).

Let us now move to the bottom of the tower.  A 0-stage for $A$ is an object $X \in \locEAlgT$ whose natural $E$-homology is given by
\[ E_{i,\star} X \cong \left\{ \begin{array}{ll} A , & i = 0 \\ 0 , & i > 0 . \end{array} \right. \]
As the natural $E$-homology groups govern cellular approximations in $\locEAlgT$, the following result should be plausible.

\begin{thm}[\ref{thm moduli space of 0-stages is algebraic}]\label{intro thm moduli space of 0-stages}
The moduli space of 0-stages for $A$ admits a canonical equivalence
\[ \ms{M}_0(A) \simeq B\Aut(A) . \]
\end{thm}

Now, if $X \in \locEAlgT$ is a 0-stage for $A$, then its natural $E$-homology is extremely simple.  On the other hand, as dictated by the spiral exact sequence, its classical $E$-homology -- and hence its $\Etwo$ page -- is not quite correct for it to be an $\infty$-stage: instead, we will have
\[ \pi_i E_\star^\lw X \cong \left\{ \begin{array}{ll} A , & i = 0 \\ \Omega A , & i = 2 \\ 0 , & i \notin \{ 0,2 \} . \end{array} \right. \]
In fact, more generally, if $X$ is an $n$-stage for $A$, then we will have
\[ \pi_i E_\star^\lw X \cong \left\{ \begin{array}{ll} A , & i = 0 \\ \Omega^{n+1} A , & i = n+2 \\ 0 , & i \notin \{ 0, n + 2 \} . \end{array} \right. \]
Thus, to move upwards through the tower of moduli spaces is to push the failure of $X$ to be an $\infty$-stage ``further and further away''.\footnote{In fact, the spectral sequence for an $n$-stage will collapse after the $\Enplustwo$ page, directly after cancelling out the entire $(n+2)\nd$ row with the corresponding entries of the $0\th$ row.}  However, we emphasize that the above identification of natural homotopy groups does not alone imply that $X$ is an $n$-stage: it must also have the correct k-invariants (or equivalently, it must also have the correct natural $E$-homology).

We now explain why this iterative \textit{topological} procedure is indeed governed by \textit{algebraic} computations.  (In fact, a somewhat simpler argument will also justify \Cref{intro thm moduli space of 0-stages}.)  This is where we will use the cocontinuity of the functor
\[ E_\star^\lw : \locEAlgT \ra \locpiAlgtTE \]
between presentable $\infty$-categories.\footnote{The adjoint functor theorem implies that this functor admits a right adjoint.  However, it appears extremely unlikely that this lifts to the level of model $\infty$-categories.  And even if it does, the functor $E_\star^\lw$ will not generally be a left Quillen functor, since it generally only preserves weak equivalences between cofibrant objects (instead of all acyclic cofibrations between arbitrary objects).}

Suppose that $X \in \locEAlgT$ is an $(n-1)$-stage for $A$.  As we have just seen, its image
\[ Y = E_\star^\lw X \in \locpiAlgtTE \]
will have its homotopy concentrated in degrees 0 and $n+1$: for brevity, we simply write
\[ \pi_* Y \cong A \times (\Omega^n A)[n+1] . \]
We are interested in modifying $X$ to obtain an $n$-stage for $A$: this entails simultaneously peeling off this copy of $(\Omega^n A)[n+1]$ and replacing it with a copy of $(\Omega^{n+1} A)[n+2]$, all in a way that behaves correctly with respect to the natural $E$-homology groups.

In order to address this question, we first examine the levelwise $E$-homology object $Y = E_\star^\lw X$.  Now, in the $\infty$-category $\locpiAlgtTE$, homotopy groups alone do not characterize equivalence classes: just as with (based) spaces, one must also keep track of the k-invariants.  In this case, since $Y$ only has potentially nonvanishing homotopy in dimensions 0 and $(n+1)$, it participates in a uniquely determined pullback square
\[ \begin{tikzcd}
Y \arrow{r} \arrow{d} & K_A \arrow{d} \\
A \arrow{r}[swap]{\chi_n(Y)} & K_A(\Omega^n A,n+2)
\end{tikzcd} \]
in $\locpiAlgtTE$, in which the objects on the right are \bit{algebraic Eilenberg--Mac Lane objects} with $\pi_* K_A \cong A$ and $\pi_*K_A(\Omega^n A, n+2) \cong A \times (\Omega^n A)[n+2]$, the right vertical map between them is an isomorphism on $\pi_0$, and the map $\chi_n(Y)$ is the unique potentially nontrivial k-invariant of $Y$.  This defines a class
\[ [\chi_n(Y)] \in H^{n+2}(A;\Omega^n A) \]
in the indicated Andr\'e--Quillen cohomology group, and taken over all $(n-1)$-stages $X \in \ms{M}_{n-1}(A)$ this defines a map
\[ \ms{M}_{n-1}(A) \xra{\chi_n} \ms{H}^{n+2}(A;\Omega^n A) \]
to the indicated Andr\'e--Quillen cohomology space.

Returning to topology, we now come to the crucial point: for any object $Z \in \locpiAlgtTE$, the composite functor
\[ \locEAlgT^{op} \xra{E_\star^\lw} \locpiAlgtTE^{op} \xra{ \hom_{\locpiAlgtTE}(-,Z)} \S \]
preserves limits (i.e.\! takes colimits in $\locEAlgT$ to limits in $\S$) and so must be representable (by presentability).  When $Z = K_A$ or $Z = K_A(\Omega^n A,n+2)$, we obtain \bit{topological Eilenberg--Mac Lane objects}, which we respectively denote by $B_A$ and $B_A(\Omega^n A,n+2)$.

Now, if there exists an $n$-stage $\widetilde{X}$ lifting $X$, then Postnikov theory in $\locEAlgT$ implies that it must fit into a pullback square
\[ \begin{tikzcd}
\widetilde{X} \arrow{r} \arrow{d} & B_A \arrow{d} \\
X \arrow{r} & B_A(\Omega^n A,n+2) ,
\end{tikzcd} \]
in which the right vertical map classifies the standard map $K_A \ra K_A(\Omega^n A,n+2)$.  Conversely, if we define $\widetilde{X}$ to be such a pullback, then it will be an $n$-stage if and only if the lower map corresponds to an equivalence
\[ E_\star^\lw X = Y \xra{\sim} K_A(\Omega^n A,n+2) . \]
As we have just seen, the equivalence class of $Y$ is entirely classified by a k-invariant
\[ [\chi_n(Y) ] \in H^{n+2}(A;\Omega^n A) , \]
and it is not hard to show that such an equivalence $Y \xra{\sim} K_A(\Omega^n A ,n+2)$ exists if and only this k-invariant vanishes.

All in all, an expansion of this argument can be used to prove the following.

\begin{thm}[\ref{pullback square to climb tower}]\label{thm intro pullback square to climb tower}
For any $n \geq 1$, there is a natural pullback square
\[ \begin{tikzcd}
\ms{M}_n(A) \arrow{r} \arrow{d} & B\Aut(A,\Omega^n A) \arrow{d} \\
\ms{M}_{n-1}(A) \arrow{r} & \what{\ms{H}}^{n+2}(A;\Omega^n A).
\end{tikzcd} \]
\end{thm}

\noindent This is the final ingredient in our generalized Goerss--Hopkins obstruction theory, which allows us to compute the purely algebraic obstructions to the inductive passage up the tower of moduli spaces
\[ \begin{tikzcd}[column sep=-0.5cm]
\ms{M}_A & \ms{M}_\infty(A) \arrow{l}[swap]{\sim} \arrow{d}{\lim} \\
& \vdots \arrow{d} \\
& \ms{M}_n(A) \arrow{r} \arrow{d} & B\Aut(A,\Omega^n A) \arrow{d} \\
& \ms{M}_{n-1}(A) \arrow{r} \arrow{d} & \what{\ms{H}}^{n+2}(A;\Omega^n A) \\
& \vdots \arrow{d} \\
& B\Aut(A) \simeq \ms{M}_0(A). \ \ \ \ \ \ \ \ \ \ \ \ \ \ \ \
\end{tikzcd} \]

\section{The resolution model structure}
\label{section.resn.model.str}

We lift results from \cite[Chapter II]{GJnew} in order to provide sufficient conditions for the existence of certain simplicial model $\infty$-category structures.

\begin{rem}
In this section, we will be constructing certain \textit{resolution} model structures.  These are closely related to the model structures of \cite{DKS-E2} and \cite{BousCosimp}; indeed, it is straightforward (but tedious) to verify that the proof of \cite[Theorem 3.3]{BousCosimp} immediately generalizes to an arbitrary right proper model $\infty$-category $\M$ equipped with a set of h-cogroup objects (in the model $\infty$-categorical sense).  However, those model structures are in a sense more difficult: they're built by modifying $(s\M)_\Reedy$, and in the end the fibrant objects are exactly the Reedy fibrant objects.  By contrast, using model $\infty$-categories effectively allows us to obtain a model structure presenting the desired $\infty$-category by starting with a \textit{trivial} model $\infty$-category (so that the Reedy model structure on simplicial objects therein will also be trivial).
\end{rem}

\subsection{Enrichments and bitensorings in the presence of presentability}

We begin by providing sufficient conditions for constructing enrichments and bitensorings among presentable $\infty$-categories, and for lifting adjunctions between $\infty$-categories equipped with these to enriched adjunctions.

\begin{prop}\label{action in PrL gives enr and bitens}
Let $\V \in \Alg(\PrL)$ be a presentably monoidal $\infty$-category, and let $\D \in \Mod_\V(\PrL)$ be a presentable $\infty$-category equipped with a left action of $\V$.  Then this action $- \tensoring - : \V \times \D \ra \D$ extends to an enrichment and bitensoring of $\D$ over $\V$, encoded by a two-variable adjunction
\[ \left( \V \times \D \xra{- \tensoring -} \D \ , \ \V^{op} \times \D \xra{ - \cotensoring - } \D \ , \D^{op} \times \D \xra{\enrhom_\D(-,-)} \V \right) . \]
\end{prop}

\begin{proof}
The fact that the action takes place in the symmetric monoidal $\infty$-category $\PrL$ guarantees that it commutes with colimits separately in each variable.  From here, presentability guarantees the co/representability required by the definition of a two-variable adjunction.
\end{proof}

\begin{lem}\label{lw tensoring is bicocontinuous}
Let $\D$ be a bicomplete $\infty$-category, and let $\I \in \Cati$ be a diagram $\infty$-category.  Then the levelwise tensoring of $\Fun(\I,\D)$ over $\Fun(\I,\S)$ commutes with colimits separately in each variable and extends to an action $\Fun(\I,\D) \in \LMod_{\Fun(\I,\S)}(\PrL)$.
\end{lem}

\begin{proof}
The levelwise tensoring is given by the composite
\[ \Fun(\I,\S) \times \Fun(\I,\D) \simeq \Fun(\I,\S \times \D) \xra{\Fun(\I, - \tensoring - )} \Fun(\I,\D) ; \]
indeed, we obtain $\Fun(\I,\D) \in \LMod_{\Fun(\I,\S)}(\Cati)$ by applying $\Fun(\I,-)$ to the data of $\D \in \LMod_\S(\Cati)$.  Moreover, by definition the tensoring $- \tensoring - : \S \times \D \ra \D$ commutes with colimits separately in each variable; as colimits in a functor $\infty$-category are computed pointwise, the above composite commutes with colimits separately in each variable as well.
\end{proof}

\begin{cor}\label{simp objs in presentable are enr and bitens over sspaces}
For any $\D \in \PrL$, the levelwise tensoring of $s\D$ over $s\S$ extends to an enrichment and bitensoring.
\end{cor}

\begin{proof}
By \Cref{lw tensoring is bicocontinuous}, the levelwise tensoring defines an action $s\D \in \Mod_{s\S}(\PrL)$, and so the claim follows from \Cref{action in PrL gives enr and bitens}.
\end{proof}

\begin{obs}
Given two $\infty$-categories $\D$ and $\E$, one can define an adjunction $\D \adjarr \E$ to be a functor $A : \D^{op} \times \E \ra \S$ satisfying certain co/representability conditions (see \cite[item (25) of \sec A]{MIC-sspaces}).  If for some closed monoidal $\infty$-category $\V$ these $\infty$-categories are equipped with lifts $\ul{\D}$ and $\ul{\E}$ to $\V$-enriched $\infty$-categories, then an enriched adjunction $\ul{\D} \adjarr \ul{\E}$ can be defined as a functor $\ul{A} : \D^{op} \times \E \ra \V$ satisfying analogous co/representability conditions.  (This recovers an ordinary adjunction between the underlying unenriched $\infty$-categories by postcomposition with the functor $\enrhom_\V(\unit_\V , -) : \V \ra \S$.)
\end{obs}

\begin{lem}\label{lift unenr adjn betw enr and bitensd cats lifts to enr adjn}
Let $\V \in \Alg(\Cati)$ be a presentable monoidal $\infty$-category, suppose that two $\infty$-categories $\D$ and $\E$ are enriched and bitensored over $\V$, and suppose we are given an adjunction $F : \D \adjarr \E : G$ between their underlying $\infty$-categories.  Suppose further that we have a natural equivalence $F(- \tensoring_\D -) \simeq (-) \tensoring_\E F(-)$ in $\Fun(\V \times \D , \E)$.  Then the adjunction $F \adj G$ lifts to a $\V$-enriched adjunction $\ul{F} : \ul{\D} : \adjarr \ul{\E} : \ul{G}$, and moreover we have a natural equivalence $G( - \cotensoring_\E - ) \simeq (-) \cotensoring_\D G(-)$ in $\Fun(\V^{op} \times \E,\D)$.
\end{lem}

\begin{proof}
First of all, the final claim follows from our assumption (and the Yoneda lemma) by the string of natural equivalences
\begin{align*}
\hom_\D(d,G(v \cotensoring_\E e)) & \simeq \hom_\E ( F(d) , v \cotensoring_\E e) \simeq \hom_\E(v \tensoring_\E F(d) , e) \\
& \simeq \hom_\E(F(v \tensoring_\D d) , e) \simeq \hom_\D(v \tensoring_\D d,G(e)) \\
& \simeq \hom_\D(d,v \cotensoring_\D G(e)) .
\end{align*}
Now, consider the functor $\D^{op} \times \E \ra \P(\V)$ which takes a pair of objects $(d^\circ,e) \in \D^{op} \times \E$ to the presheaf taking $v^\circ \in \V^{op}$ to the space
\[ \hom_\D(v \tensoring_\D d , Ge) \simeq \hom_\E(F(v \tensoring_\D d),e ) \simeq \hom_\E(v \tensoring_\E F(d) , e) \simeq \hom_\E(F(d),v \cotensoring_\E e) . \]
Since $\V$ is presentable, this factors through the Yoneda embedding $\V \xhookra{\Yo_\V} \P(\V)$.  By construction, this defines an enriched adjunction $\ul{F} : \ul{\D} \adjarr \ul{\E} : \ul{G}$ lifting the original adjunction $F \adj G$.
\end{proof}

\begin{cor}\label{enr and bitens of algs over a monad on sD}
For any $\D \in \PrL$ and any monad $t \in \Alg(\End(s\D))$, we obtain a canonical enrichment and bitensoring of $\Alg_t(s\D)$ over $s\S$, and moreover the adjunction $\free_t : s\D \adjarr \Alg_t(s\D) : \forget_t$ is canonically enriched over $s\S$.
\end{cor}

\begin{proof}
As any object of $\Alg_t(s\D)$ is a colimit of free objects, for any $K \in s\S$ and any $Y \in \Alg_t(s\D)$ we define
\[ K \tensoring Y = \colim_{(X \ra \forget_t(Y)) \in s\N_{/\forget_t(Y)}} \free_t(K \tensoring X) \]
(using the action $s\D \in \LMod_{s\S}(\PrL)$ of \Cref{simp objs in presentable are enr and bitens over sspaces}).  This defines a bifunctor $- \tensoring - : s\S \times \Alg_t(s\D) \ra \Alg_t(s\D)$ which by construction commutes with colimits separately in each variable.  Thus it defines an action $\Alg_t(s\D) \in \LMod_{s\S}(\PrL)$, and so by \Cref{action in PrL gives enr and bitens} extends to an enrichment and bitensoring of $\Alg_t(s\D)$ over $s\S$.  Then, the enrichment of the adjunction $\free_t \adj \forget_t$ follows from \Cref{lift unenr adjn betw enr and bitensd cats lifts to enr adjn}.
\end{proof}

\subsection{Simplicial model structures}

We now provide a lifting theorem for constructing simplicial model $\infty$-category structures.  This requires two auxiliary pieces of terminology.

\begin{defn}
Given a set $I$ of homotopy classes of maps in $\C$, the subcategory $I\dashproj$ of \bit{$I$-projectives} is the subcategory of maps with $\llp(I)$.
\end{defn}

\begin{defn}
Let $\V$ be a monoidal model $\infty$-category, and suppose that $\M$ and $\N$ are $\V$-enriched model $\infty$-categories.  Then a \bit{$\V$-enriched Quillen adjunction} between $\M$ and $\N$ is a $\V$-enriched adjunction $\ul{F} : \ul{\M} \adjarr \ul{\N} : \ul{G}$ such that the underlying adjunction $F : \M \adjarr \N : G$ is a Quillen adjunction.
\end{defn}

\begin{thm}\label{thm Mres}
Let $\M$ be a bicomplete $\infty$-category, and let $F : s\S \adjarr \M : G$ be an adjunction such that $G$ commutes with filtered colimits.  Write $\bW^\M = G^{-1}(\bW^{s\S}_\KQ)$, $\bF^\M = G^{-1}(\bF^{s\S}_\KQ)$, and $\bC^\M = (\bW \cap \bF)^\M \dashproj$.  Suppose that the following condition holds:
\begin{equation}
\left( \bC^\M \cap (\bF^\M \dashproj \right)) \subset \bW^\M . \tag{$*$}
\end{equation}
Then $\M$ admits a \bit{resolution model structure}, denoted $\M_\res$, with $\bW_\res^\M = \bW^\M$, $\bC^\M_\res = \bC^\M$, and $\bF_\res^\M = \bF^\M$, and the above adjunction becomes a Quillen adjunction $F : s\S_\KQ \adjarr \M_\res : G$.
\end{thm}

\begin{proof}
The proof is almost identical to that of \cite[Theorem II.4.1]{GJnew} (despite the fact that there they only work in the special case of a category of simplicial objects); the only modification which must be made is that in the proofs of \cite[Lemmas II.4.2 and II.4.3]{GJnew} (which construct required factorizations) one must take a coproduct over \textit{homotopy classes} of commutative squares.
\end{proof}

\begin{rem}
In practice, there seems to more-or-less always be (at least) one thing that's difficult to check in constructing a model structure.  In this case, condition $(*)$ of \Cref{thm Mres} effectively requires that those would-be cofibrations that moreover have the left lifting property for all would-be fibrations are also would-be weak equivalences.  We will give sufficient conditions for this condition to hold in \Cref{subsection sufficient criteria for condition star to get resn model str}.
\end{rem}

\begin{rem}
It follows from the proof of \Cref{thm Mres} that one can replace the condition $(*)$ with the following pair of conditions:
\begin{itemizesmall}
\item[$(*')$] for every map $\Lambda^n_i \ra \Delta^n$ in $J^{s\S}_\KQ$, the induced map $F(\Lambda^n_i) \ra F(\Delta^n)$ lies in $\bW^\M \subset \M$;
\item[$(*'')$] the maps in $(\bW \cap \bC)^\M$ are closed under coproducts, pushouts, and sequential colimits.
\end{itemizesmall}
This is explained in \cite[Remark II.4.5]{GJnew}.
\end{rem}

\begin{thm}\label{thm Mres simp}
In the setting of \Cref{thm Mres}, suppose that we have an action $\M \in \LMod_{s\S}(\Cati)$, denoted $- \tensoring - : s\S \times \M \ra \M$, such that this bifunctor commutes with colimits separately in each variable, and suppose that we have a natural equivalence $F(- \times - ) \simeq (-) \tensoring F(-)$ in $\Fun(s\S \times s\S , \M)$.  Then the resolution model structure canonically enhances to a simplicial model $\infty$-category $\ul{\M}_\res$, and the Quillen adjunction canonically enhances to an $s\S_\KQ$-enriched Quillen adjunction $\ul{F} : \ul{s\S}_\KQ \adjarr \ul{\M}_\res : \ul{G}$.
\end{thm}

\begin{proof}
Using \Cref{lift unenr adjn betw enr and bitensd cats lifts to enr adjn}, the proof is identical to that of \cite[Theorem II.4.4]{GJnew}.
\end{proof}

\subsection{Sufficient criteria for the satisfaction of condition $(*)$ of \Cref{thm Mres}}
\label{subsection sufficient criteria for condition star to get resn model str}

We now provide various conditions guaranteeing that condition $(*)$ of \Cref{thm Mres} is satisfied.

The key result is the following.

\begin{prop}\label{suff to have fibt repl functor}
In the setting of \Cref{thm Mres}, suppose that there exists an endofunctor $\bbR : \M \ra \M$ which factors through the subcategory $\bF^\M \subset \M$ and which admits a map $\id_\M \ra \bbR$ whose components lie in $\bW^\M$.  Then condition $(*)$ holds.
\end{prop}

\begin{proof}
The proof is identical to that of \cite[Lemma II.5.1]{GJnew}.
\end{proof}

\begin{cor}\label{suff to have all objects fibt}
In the setting of \Cref{thm Mres}, suppose that for every object $X \in \M$ the terminal map $X \ra \pt_\M$ lies in $\bF^\M$.  Then condition $(*)$ holds.
\end{cor}

\begin{proof}
This follows from \Cref{suff to have fibt repl functor}, taking $\bbR = \id_\M$ (equipped with the identity coaugmentation).
\end{proof}

\begin{cor}\label{lw hom out of small obj}
Let $\N$ be a bicomplete $\infty$-category, and for any object $Z \in \N$ consider the adjunction
\[ - \tensoring \const(Z) : s\S \adjarr s\N : \hom^\lw_\N(Z,-) . \]
If the object $Z \in \N$ is small, then this adjunction satisfies condition $(*)$ of \Cref{thm Mres}.
\end{cor}

\begin{proof}
With the theory of the $\Ex^\infty$ functor for $s\S_\KQ$ of \cite[\sec 6]{MIC-sspaces} in hand (specifically \cite[Proposition 6.22 and Remark 6.23]{MIC-sspaces}), this follows from \Cref{suff to have fibt repl functor} by an identical argument to that of \cite[Proposition II.5.5]{GJnew}.
\end{proof}

\begin{rem}
The technique of \Cref{lw hom out of small obj} cannot work for a general (bicomplete) $\infty$-category equipped with a right adjoint functor to $s\S$: it must be an $\infty$-category of simplicial objects.  In effect, this is because the endofunctor $\Ex$ is a right adjoint, but it is not an \textit{enriched} right adjoint.  Indeed, the functor $\enrhom_{s\S}(\Delta^1,-) : s\S \ra s\S$ is an example of an enriched limit and so commutes with any enriched right adjoint, but the canonical map $\Ex(\enrhom_{s\S}(\Delta^1,-)) \ra \enrhom_{s\S}(\Delta^1,\Ex(-))$ is not an equivalence; this can be seen by evaluating on $\Delta^1$, since the source has three 0-simplices but the target has five.
\end{rem}

\begin{cor}\label{sNres}
Let $\N \in \PrL$, and let $Z \in \N$ be a small object.  Then with the enrichment and bitensoring of $s\N$ over $s\S$ of \Cref{simp objs in presentable are enr and bitens over sspaces}, there exists a simplicial model structure on $s\N$ created by the $s\S$-enriched Quillen adjunction
\[ - \tensoring \const(Z) : \ul{s\S}_\KQ \adjarr \ul{s\N}_\res : \hom^\lw_\N(Z,-) . \]
\end{cor}

\begin{proof}
By \Cref{lw hom out of small obj}, this adjunction satisfies condition $(*)$ of \Cref{thm Mres} and hence creates a model structure on $s\N$.  By \Cref{lw tensoring is bicocontinuous}, this adjunction furthermore satisfies the hypotheses of \Cref{thm Mres simp}, so that $s\N_\res$ and the Quillen adjunction becomes compatibly $s\S_\KQ$-enriched.
\end{proof}

We will also be interested in the following ``many-object'' version of \Cref{sNres}.

\begin{thm}\label{many-object sNres}
Let $\N \in \PrL$, and suppose we are given a set of small objects $Z_\alpha \in \N$.  Then with the enrichment and bitensoring of $s\N$ over $s\S$ of \Cref{simp objs in presentable are enr and bitens over sspaces}, there exists a simplicial model structure on $s\N$ created by the $s\S$-enriched Quillen adjunction
\[ \coprod_\alpha \pr_\alpha(-) \tensoring \const(Z_\alpha) : \prod_\alpha \ul{s\S}_\KQ \adjarr \ul{s\N}_\res : \left( \hom^\lw_\N(Z_\alpha,-) \right) . \]
\end{thm}

\begin{proof}
Given the above results, the proof is essentially identical to that of \cite[Proposition II.5.9]{GJnew}.
\end{proof}

\begin{rem}
In \Cref{many-object sNres}, if the objects $Z_\alpha$ form a set of compact projective generators (in the sense of Definition T.5.5.8.23) and the $\infty$-category $\N$ has enough projectives, then weak equivalences and fibrations in $s\N_\res$ will be detected by \textit{all} projective objects (see \cite[Example II.5.10]{GJnew}).
\end{rem}

We now identify the underlying $\infty$-category of the resolution model structure of \Cref{many-object sNres}.


\begin{defn}
For an $\infty$-category $\D$ admitting finite coproducts, we write $\PS(\D) = \Fun^\times(\D^{op},\S)$ for its \bit{nonabelian derived $\infty$-category} of product-preserving presheaves (i.e.\! of functors taking finite coproducts in $\D$ to finite products in $\S$).  We write $\PSd(\D) \subset \PS(\D)$ for its subcategory of discrete objects; thus $\PSd(\D) \simeq \Fun^\times(\D^{op},\Set) \simeq \Fun^\times(\ho(\D)^{op},\Set)$.
\end{defn}

\begin{thm}\label{identify localization of sNres}
In the situation of \Cref{many-object sNres}, writing $\G \subset \N$ for the full subcategory generated by the objects $Z_\alpha$ under finite coproducts, we have a canonical Quillen adjunction
\[ \Fun(\G^{op},s\S_\KQ)_\projective \adjarr s\N_\res \]
with derived adjunction given by the canonical adjunction
\[ \P(\G) \adjarr \PS(\G) \]
whose right adjoint is the defining inclusion.
\end{thm}

\begin{proof}
The projective model structure can also be seen as lifted via \Cref{many-object sNres} from the same product of copies of the model $\infty$-category $s\S_\KQ$, which implies that this is indeed a Quillen adjunction.  As the functor $|{-}|: s\S \ra \S$ commutes with finite products, it follows that the derived right adjoint factors through the subcategory $\PS(\G) \subset \P(\G)$.  Moreover, as $\N$ is presentable, the restricted Yoneda embedding participates in an adjunction $\PS(\G) \adjarr \N$, from which it follows that this derived right adjoint surjects onto $\PS(\G)$ (by taking the constant simplicial object on a given object of $\N$, seen as a product-preserving presheaf on $\G$).  So, it will suffice to show that the functor $\loc{s\N}{\bW_\res} \ra \P(\G)$ is fully faithful.  First of all, taking any $X \in s\N_\res^f$, since $s\N_\res$ is simplicial, for any $K \in s\Set = s\S_\KQ^c$ we have that
\begin{align*}
\hom_{\loc{s\N}{\bW_\res}}(K \tensoring \const(Z_\alpha) , X) & \simeq \enrhom_{s\N}(K \tensoring \const(Z_\alpha),X) \\
& \simeq \enrhom_{s\S}(K , \enrhom_{s\N}(\const(Z_\alpha),X)) \\
& \simeq \enrhom_{s\S}(K,\hom^\lw_\N(Z_\alpha,X)) \\
& \simeq \hom_\S(|K| , | \hom^\lw_\N(Z_\alpha,X)|)
\end{align*}
(where the last equivalence uses the fact that $s\S_\KQ$ is a simplicial model $\infty$-category).  The claim now follows from the fact that $I^{s\N}_\res = \{ I^{s\S}_\KQ \tensoring \const(Z_\alpha) \}$ forms a set of generating cofibrations of $s\N_\res$, so that we can construct a cofibrant replacement of any object as a transfinite composition of pushouts of these maps.
\end{proof}

We end this subsection with the following result, which gives a convenient class of examples for which the condition of \Cref{suff to have all objects fibt} holds (i.e.\! that all objects are (``would-be'') fibrant).  It is an $\infty$-categorical analog of the classical fact that every simplicial group is in particular a Kan complex.

\begin{lem}\label{simplicial infty-groups are fibt}
In the adjunction $\free_{s\Grp(\S)} : s\S \adjarr s\Grp(\S) : \forget_{s \Grp(\S)}$, the right adjoint factors through the subcategory $s\S^f_\KQ \subset s\S$ of fibrant objects with respect to the Kan--Quillen model structure.
\end{lem}

\begin{proof}
Observe that the adjunction $\free_{\Grp(\S)} : \S \adjarr \Grp(\S) : \forget_{\Grp(\S)}$ factors as the composite adjunction
\[ \begin{tikzcd}[column sep=2cm]
\S \arrow[transform canvas={yshift=0.8ex}]{r}{\free_{\Mon(\S)}} \arrow[leftarrow, transform canvas={yshift=-0.8ex}]{r}[swap]{\forget_{\Mon(\S)}} &
\Mon(\S) \arrow[transform canvas={yshift=0.8ex}]{r}{(-)^\gp} \arrow[hookleftarrow, transform canvas={yshift=-0.8ex}]{r} &
\Grp(\S) .
\end{tikzcd} \]
We claim that the diagram
\[ \begin{tikzcd}[column sep=1.5cm]
\Set \arrow[hook]{d} \arrow{r}{\free_\Mon} & \Mon \arrow{r}{(-)^\gp}  & \Grp \arrow[hook]{d} \\
\S \arrow{r}[swap]{\free_{\Mon(\S)}} & \Mon(\S) \arrow{r}[swap]{(-)^\gp} & \Grp(\S)
\end{tikzcd} \]
commutes.\footnote{If we were to add in the middle vertical inclusion $\Mon \hookra \Mon(\S)$, the left square would commute (simply by inspection of the functor $\free_{\Mon(\S)}$), but the right square would not: its extreme failure to do so is encoded by \cite[Theorem 1]{McDuffMonoids}.}  Indeed, recall the factorization
\[ \begin{tikzcd}
\Mon(\S) \arrow{rr}{(-)^\gp} \arrow{rd}[swap]{B} & & \Grp(\S) , \\
& \S_* \arrow{ru}[swap]{\Omega}
\end{tikzcd} \]
and recall that the functor $\Mon(\S) \xra{B} \S_*$ can itself be obtained as the composite
\[ \Mon(\S) \xra{\mf{B}} (\Cati)_* \xra{(-)^\gpd} (\Gpd_\infty)_* \simeq \S_* \]
(where $\mf{B}$ denotes the ``categorical delooping'' functor).  The claim now follows from the commutativity of the diagram
\[ \begin{tikzcd}[column sep=1.5cm]
\Set \arrow{r}{\free_\Mon} \arrow[hook]{d} & \Mon \arrow{r}{\mf{B}} & \strcat_* \arrow{r}{(-)^\gpd} & \strgpd_* \arrow{d} \\
\S \arrow{r}[swap]{\free_{\Mon(\S)}} & \Mon(\S) \arrow{r}[swap]{\mf{B}} & (\Cati)_* \arrow{r}[swap]{(-)^\gpd} & (\Gpd_\infty)_* ,
\end{tikzcd} \]
which itself follows from \cite[5.4]{DKSimpLoc}.

Now, applying $\Fun(\bD^{op},-)$ to the original commutative rectangle, we obtain a commutative square
\[ \begin{tikzcd}[column sep=1.5cm]
s\Set \arrow{r}{\free_{s\Grp}} \arrow[hook]{d} & s\Grp \arrow[hook]{d} \\
s\S \arrow{r}[swap]{\free_{s\Grp(\S)}} & s\Grp(\S) .
\end{tikzcd} \]
In particular, the image of any element $\Lambda^n_i \ra \Delta^n$ of $J^{s\Set}_\KQ = J^{s\S}_\KQ$ under the composite
\[ s\Set \hookra s\S \xra{\free_{s\Grp(\S)}} s\Grp(\S) \]
admits a retraction (see e.g.\! \cite[Lemma I.3.4]{GJnew}).  This proves the claim.
\end{proof}


\section{Topology}
\label{section.topology}

In this section, we lay out the basic topological framework (absent any operadic structure).

\subsection{Foundations of topology}

\begin{ass}
We begin with a presentably symmetric monoidal stable $\infty$-category $\C = (\C,\otimes,\unit)$.  By presentability, this will automatically be closed (i.e.\! admit an internal hom bifunctor).
\end{ass}

\begin{rem}
When it is convenient, we will consider $\C$ as being enriched over the symmetric monoidal $\infty$-category $(\S_*,\sm,S^0)$ of pointed spaces equipped with the smash product: the basepoint $0 \in \hom_\C(X,Y)$ is given by the unique ``zero map'' $X \ra 0_\C \ra Y$, and the fact that the composition maps factor through the smash products amounts to the observation that any sequence of composable maps in which at least one of the maps is a zero map composes canonically to another zero map.  Moreover, $\C$ admits a canonical bitensoring over $\S_*$ which is compatible with this enrichment.  (It is not hard to make these assertions precise using the formalism of \cite{GepHaug}.)
\end{rem}

\begin{notn}
We write $\Dual = \enrhom_\C(-,\unit) : \C^{op} \ra \C$ for the ``linear dual'' functor, and we write $\C^\inv \subset \C^\dual \subset \C$ for the full subcategories of invertible objects and of dualizable objects.
\end{notn}

\begin{ass}\label{ass compact unit}
We assume that the unit object $\unit \in \C$ is compact, i.e.\! that the functor $\hom_\C(\unit,-) : \C \ra \S$ commutes with filtered colimits.
\end{ass}

\begin{obs}\label{obs inv and dzble also cpct}
It follows immediately from \Cref{ass compact unit} that any invertible object of $\C$ is necessarily compact.  In fact, because of the assumption that the symmetric monoidal structure commutes with colimits separately in each variable, it follows that any dualizable object is compact as well: this is a consequence of the natural equivalence $\hom_\C(X,-) \simeq \hom_\C(\unit,\Dual X \otimes - )$ in $\Fun(\C,\S)$.
\end{obs}

\begin{ass}\label{ass generators}
We assume the existence of a small subcategory $\G \subset \C$ of (strong) generators, which we generally denote by $S^\beta \in \G$ (with the ``$S$'' and ``$\beta$'' chosen to evoke the notion of a ``bigraded sphere'' (from motivic stable homotopy theory)); that is, we assume that the functors
\[ \hom_\C(S^\beta,-) : \C \ra \S \]
are jointly conservative.  We moreover assume that the subcategory $\G \subset \C$
\begin{itemizesmall}
\item contains the unit object $\unit \in \C$,
\item is closed under de/suspensions,
\item consists of invertible objects, and
\item is closed under the monoidal product of $\C$.
\end{itemizesmall}
We write $S^{n+\beta} = \Sigma^n S^\beta$ for any $n \in \bbZ$.
\end{ass}

\begin{notn}
We write $\G^\delta = \pi_0(\G^\simeq) \in \Ab\Grp$ for the abelian group of equivalence classes of objects of $\G$, with addition given by the monoidal product of $\C$.  We denote the element corresponding to $S^\beta \in \G$ simply by $\beta \in \G^\delta$.
\end{notn}

\begin{defn}
For any $\beta \in \G^\delta$, we refer to the equivalence $S^\beta \otimes - : \C \xra{\sim} \C$ as the \bit{$\beta$-fold suspension}.  The ordinary notion of suspension is recovered as $(\Sigma^n \unit) \otimes - : \C \xra{\sim} \C$.  We will henceforth refer to \textit{any} $\beta$-fold suspension as a ``suspension'', and refer to this latter more restrictive notion as a \bit{categorical suspension}.  We denote $\beta$-fold suspension by $\Sigma^\beta$, and categorical suspension simply by $\Sigma^n$.  (Note that these conventions jibe with those of \Cref{ass generators}.)  While through this definition the term ``desuspension'' technically becomes superfluous, we will nevertheless continue to employ it for aesthetic reasons.
\end{defn}

\begin{notn}
We write $\uA = \Fun(\G^\delta,\Ab)$ for the category of $\G^\delta$-graded abelian groups, equipped with the Day convolution monoidal structure relative to $(\G^\delta,+) = (\G^\delta,\otimes_\C)$ and $(\Ab,\otimes_\bbZ)$.  This receives a ``homotopy'' functor $\pi_\star : \C \ra \uA$, given by $\pi_\beta X = (\pi_\star X)(S^\beta) = [S^\beta,X]_\C$.\footnote{This is the composite of the canonical projection $\C \ra \ho(\C)$ followed by the restricted Yoneda embedding along the functor $\G^\delta \ra \ho(\C)$; note that we have a canonical equivalence $\G^\delta \simeq (\G^\delta)^{op}$ since this category has no nonidentity morphisms.}  This functor is is itself lax monoidal, and in fact descends along the monoidal functor $\C \ra \ho(\C)$ to another lax monoidal functor $\pi_\star : \ho(\C) \ra \uA$.
\end{notn}

\begin{rem}\label{htpy creates equivces in C}
As a result of \Cref{ass generators}, to say that $\G \subset \C$ is a subcategory of strong generators is precisely to say that the functor $\pi_\star : \C \ra \uA$ creates the equivalences in $\C$.
\end{rem}

\begin{rem}
One could alternatively consider the ``homotopy'' functor as taking values in $\PSd(\G^\vee) = \Fun((\G^\vee)^{op},\Set)$, the category of product-preserving presheaves of sets on the closure of $\G \subset \C$ under finite coproducts (which remain coproducts in $\ho(\C)$ since $\pi_0 : \S \ra \Set$ preserves products).  This is analogous to the ``$\Pi$-algebra'' perspective taken by Dwyer--Kan--Stover in \cite{DKS-bigraded} and by Blanc--Dwyer--Goerss in \cite{BDG-moduli-problem}.  However, in order to obtain a computable obstruction theory, Goerss--Hopkins take an alternative route, considering the homotopy groups of a spectrum simply as a $\bbZ$-graded abelian group (rather than as a module over the stable homotopy groups of spheres).\footnote{Nevertheless, product-preserving presheaves pervade this story.  We will mostly suppress them, but we will need to discuss them explicitly in \Cref{module structure on unlocalized spiral}.}
\end{rem}

We conclude this subsection with a few remarks concerning the choice of ambient $\infty$-category.

\begin{rem}
If we remove the requirement that $\C$ be stable, it becomes necessary to assume that the generators admit desuspensions in order for \Cref{normalized chains commutes with homotopy} to hold. 
It also becomes necessary to assume that the generators are h-cogroup objects (with respect to the wedge sum) in order to construct the relevant spectral sequence, but of course this is a strictly weaker assumption.  More broadly, a great many of the arguments would become substantially more delicate.
\end{rem}

\begin{rem}
If we only require $\C$ to be monoidal (instead of symmetric monoidal), then by the so-called ``microcosm principle'' it will only make sense to discuss associative algebras in $\C$, instead of commutative algebras.  In the setting of ordinary spectra, associative algebras can be constructed via Hopkins--Miller obstruction theory (see \cite{RezkHopMil}), which is far simpler than Goerss--Hopkins obstruction theory since it is not necessary to resolve the associative operad (see \Cref{subsection resolutions of operads}).  On the other hand, if we set our sights lower and remove the operad from the picture entirely, we simply recover an abstract version of Blanc--Dwyer--Goerss obstruction theory (see \cite{BDG-moduli-problem}).  In any case, we expect that practical examples of interest will carry symmetric monoidal structures anyways.
\end{rem}

\subsection{The resolution model structure}

\begin{notn}
Let $E \in \CAlg(\ho(\C))$ be a homotopy commutative algebra object in $\C$. 
This induces $E_\star = \pi_\star E \in \CAlg(\uA)$, and we write $\A = \Mod_{E_\star}(\uA)$ for its category of modules.  Then we obtain a ``homology'' functor $E_\star : \C \ra \A$ by $E_\star X = \pi_\star (E \otimes X)$.
\end{notn}

\begin{defn}
An \bit{$E_\star$-equivalence} in $\C$ is a morphism which becomes an isomorphism under the functor $E_\star : \C \ra \A$.
\end{defn}

\begin{notn}
By definition, the $E_\star$-equivalences are created by the composite $\C \xra{E \otimes -} \C \xra{\pi_*} \uA$ (as isomorphisms in $\A$ are created in $\uA$).  However, \Cref{htpy creates equivces in C} implies that they are also created by the functor $\C \xra{E \otimes -} \C$.  Our assumption that $\C$ is presentably symmetric monoidal immediately implies that the $E_\star$-equivalences are strongly saturated (in the sense of Definition T.5.5.4.5), and so by Proposition T.5.5.4.15 there exists a left localization adjunction $\leftloc_{E_\star} : \C \adjarr \leftloc_E(\C) : \forget_{E_\star}$.
\end{notn}

\begin{defn}
We define the subcategory $\A_\proj \subset \A$ of \bit{projective} objects just as in classical algebra.
\end{defn}

\begin{ass}\label{adams's condition}
We assume henceforth that $E$ satisfies \bit{Adams's condition}, and fix a witnessing datum: this consists of a filtered diagram 
$E_\bullet : \J \ra \C^\dual_{/E} = \C^\dual \times_\C \C_{/E}$ with $\colim(\J \xra{E_\bullet} \C) \xra{\sim} E$, such that for every $\alpha \in \J$,
\begin{itemize}
\item $E_\star \Dual E_\alpha \in \A_\proj$, and
\item for every $M \in \Mod_E(\ho(\C))$, the canonical map
\begin{align*}
[ \Dual E_\alpha, M ]_\C & \ra \hom_\A ( E_\star \Dual E_\alpha, \pi_\star M ) \\
\left( \Dual E_\alpha \xra{f} M \right) & \mapsto \left( E_\star \Dual E_\alpha \xra{E_\star(f)} E_\star M = \pi_\star(E \otimes M) \ra \pi_\star M \right)
\end{align*}
is an isomorphism. 
\end{itemize}
\end{ass}

\begin{rem}
The canonical map of \Cref{adams's condition} can be equivalently seen as the composite
\[ [ \Dual E_\alpha , M ]_\C \cong [ E \otimes \Dual E_\alpha , M ]_{\Mod_E(\ho(\C))} \xra{\pi_\star} \hom_\A ( E_\star \Dual E_\alpha , \pi_\star M ) . \]
\end{rem}

\begin{obs}\label{DEalphas detect E-homology}
For any $X \in \C$ and any $\beta \in \G^\delta$, we have the string of isomorphisms
\begin{align*}
\colim_{\alpha \in \J} [ \Sigma^\beta \Dual E_\alpha , X]_\C & \cong  \colim_{\alpha \in \J} [S^\beta , E_\alpha \otimes X]_\C
\cong [S^\beta , \colim_{\alpha \in \J}(E_\alpha \otimes X) ]_\C \\
& \cong [S^\beta , \colim_{\alpha \in \J}(E_\alpha) \otimes X ]_\C
\cong [S^\beta , E \otimes X ]_\C
= E_\beta X
\end{align*}
in $\Ab$.
\end{obs}

\begin{notn}
Strings of adjunction isomorphisms having the same flavor as that of \Cref{DEalphas detect E-homology} will frequently be useful to us.  Rather than spell out the isomorphisms each time, we simply refer to this line of reasoning as a \bit{colimit argument}.
\end{notn}

\begin{notn}
We write $\GE \subset \C$ for the smallest full subcategory containing $\G$ and $\{ \Dual E_\alpha \}_{\alpha \in \J}$ that is closed under de/suspension and finite coproducts.  We generally write $S^\eps \in \GE$ for an arbitrary object (the letter ``$\eps$'' being suggestive of the letter ``$E$''), although we continue to write $S^\beta \in \G \subset \GE$ for an arbitrary object of $\G$ when considered as an object of $\GE$.  We write $\GE^\delta = \pi_0 ( (\GE)^\simeq)$, and so (just as we write $\beta \in \G$) we also simply write $\eps \in \GE^\delta$ to denote an arbitrary element.
\end{notn}

\begin{obs}
For any $S^\eps \in \GE$ and any $M \in \Mod_E(\ho(\C))$, we have an isomorphism
\[ [S^\eps , M ] \xra{\cong} \hom_\A(E_\star S^\eps, \pi_\star M) . \]
This can be seen as follows.
\begin{itemize}
\item For $S^\eps = \Dual E_\alpha$, this follows from \Cref{adams's condition}.
\item For $S^\eps = S^\beta \in \G$, note that $E_\star S^\beta \cong E_\star \otimes_{\unit_\star} \pi_\star S^\beta$, and so we are interested in the composite
\[ [S^\beta ,M]_\C \cong [E \otimes S^\beta,M]_{\Mod_E(\ho(\C))} \xra{\pi_\star} \hom_\A(E_\star S^\beta,\pi_\star M) \cong \hom_{\Mod_{\unit_\star}(\uA)} (\pi_\star S^\beta,\pi_\star M) , \]
which is an isomorphism with inverse given by evaluation at the universal element of $\pi_\beta S^\beta$.
\item In general, this property is preserved both by de/suspension and by the formation of finite coproducts.
\end{itemize}
\end{obs}

\begin{notn}\label{sC enr and bitensored over sS}
Recall that $s\C$ is canonically enriched and bitensored over $s\S$ (see \Cref{simp objs in presentable are enr and bitens over sspaces}); these data assemble into a two-variable adjunction, which we denote by
\[ \left( s\S \times s\C \xra{- \tensoring -} s\C \ , \ s\S^{op} \times s\C \xra{ - \cotensoring - } s\C \ , \ s\C^{op} \times s\C \xra{\enrhom_\C(-,-)} s\S \right) .\]
\end{notn}

\begin{defn}
We fix the following terminology.
\begin{enumerate}
\item A morphism in $\ho(\C)$ is called a \bit{$\GE$-epimorphism} if the restricted Yoneda functor $\ho(\C) \ra \PSd(\GE)$ takes it to a componentwise surjection.
\item An object of $\ho(\C)$ is called \bit{$\GE$-projective} if it has the extension property for all $\GE$-epimorphisms.
\item A morphism in $\ho(\C)$ is called a \bit{$\GE$-projective cofibration} if it has the left lifting property for all $\GE$-epimorphisms.
\end{enumerate}
\end{defn}

\begin{thm}\label{thm sCres}
There is a \bit{resolution model structure} on $s\C$, denoted $s\C_\res$, which enjoys the following properties.
\begin{enumerate}
\item Its weak equivalences and fibrations are created by the functor
\[ s\C \xra{ X \mapsto ( S^\eps \mapsto \hom^\lw_\C(S^\eps,X))} \prod_{\GE^\delta} s\S_\KQ . \]
\item It is simplicial.
\item Its cofibrations are precisely those morphisms whose relative latching maps are $\GE$-projective cofibrations.
\item All objects are fibrant in it.
\item It is cofibrantly generated by the sets
\[ I^{s\C}_\res = \{ I^{s\S}_\KQ \tensoring \const(S^\eps) \}_{S^\eps \in \GE} = \{ \partial \Delta^n \tensoring \const(S^\eps) \ra \Delta^n \tensoring \const(S^\eps) \}_{n \geq 0 , S^\eps \in \GE} \]
and
\[ J^{s\C}_\res = \{ J^{s\S}_\KQ \tensoring \const(S^\eps) \}_{S^\eps \in \GE} = \{ \Lambda^n_i \tensoring \const(S^\eps) \ra \Delta^n \tensoring \const(S^\eps) \}_{0 \leq i \leq n \geq 1 , S^\eps \in \GE} . \]
\end{enumerate}

\end{thm}

\begin{proof}
This follows from \Cref{many-object sNres} and \Cref{simplicial infty-groups are fibt}.
\end{proof}

\begin{rem}\label{weak equivalences of sCres also pulled back from sAb}
It will follow from the localized spiral exact sequence of \Cref{construct localized spiral} that the weak equivalences of $s\C_\res$ are created by the functor
\[ s\C \xra{ [=,-]_\C^\lw } s \Fun(\GE,\Ab) \simeq \Fun(\GE , s\Ab_\KQ)_\projective . \]
(In fact, the fibrations are as well.)
\end{rem}

\begin{defn}\label{E-equivces in sC}
We define the subcategory of \bit{$E_\star$-equivalences}, denoted $\bW_{E_\star^\lw} = \bW^{s\C}_{E_\star^\lw} \subset s\C$, to be created by pulling back the subcategory $\bW^{s\A}_\KQ \subset s\A_\KQ$ under the functor $E_\star^\lw : s\C \ra s\A_\KQ$.
\end{defn}

\begin{notn}
Rather than overburden notation, we simply write $\pi_n : s\Ab \ra \Ab$ for the composite
\[ s\Ab \xra{|{-}|} \Ab\Grp(\S_*) \xra{\Ab\Grp(\pi_n)} \Ab\Grp(\Set_*) = \Ab . \]
This can be obtained more abstractly as a ``homotopy'' functor from a derived $\infty$-category to its heart, and indeed we use this same notation $\pi_n$ to denote all corresponding functors $s\Set_* \ra \Set_*$, $s\uA \ra \uA$, $s\A \ra \A$, etc.
\end{notn}

\begin{obs}\label{E-equivce implies isomorphism on Etwo pages}
Suppose that $X \we Y$ is a weak equivalence in $s\C_\res$.  By \Cref{weak equivalences of sCres also pulled back from sAb}, this means that for every $S^\eps \in \GE$ we obtain a weak equivalence $[S^\eps,X]^\lw_\C \we [S^\eps,Y]^\lw_\C$ in $s\Ab_\KQ$, i.e.\! that we obtain isomorphisms $\pi_n([S^\eps,X]^\lw_\C) \xra{\cong} \pi_n([S^\eps,Y]^\lw_\C)$ in $\Ab$ for all $n \geq 0$.  In particular, letting $S^\eps$ range over the set $\{ \Sigma^\beta \Dual E_\alpha \}_{\beta \in \G^\delta , \alpha \in \J}$, by \Cref{DEalphas detect E-homology} and since homotopy groups in $s\Set_*$ commute with filtered colimits, we obtain a weak equivalence $E_\star^\lw X \we E_\star^\lw Y$ in $s\A_\KQ$.  In other words, we have an inclusion $\bW_\res \subset \bW_{E_\star^\lw}$ of subcategories of $s\C$.
\end{obs}

\begin{obs}
In our setting, after a colimit argument the standard filtration spectral sequence for an object $X \in s\C$ runs $\pi_n E_\beta^\lw X \Rightarrow E_{\beta+n}|X|$.  (This agrees with the spectral sequence associated to the localized spiral exact sequence of \Cref{construct localized spiral} (see \cite[Lemma 3.1.5 and Remark 3.1.6]{GH}).)  Thus, an $E_\star$-equivalence in $s\C$ (for instance a weak equivalence in $s\C_\res$, by \Cref{E-equivce implies isomorphism on Etwo pages}) induces an isomorphism on $\Etwo$ pages of this spectral sequence.  In other words, there exists a factorization
\[ \begin{tikzcd}
s\C \arrow{r}{|{-}|} \arrow{d} & \C \arrow{r}{E_\star} & \A \\
\locEsC \arrow[dashed, bend right=5]{rru}
\end{tikzcd} \]
through the localization functor.
\end{obs}

\begin{defn}
We refer to this spectral sequence $\Etwo = \pi_n E_\beta^\lw X \Rightarrow \Einfty = E_{\beta +n} |X|$ as the \bit{spiral spectral sequence}.
\end{defn}

\begin{rem}
By \Cref{identify localization of sNres}, the resolution model structure presents the nonabelian derived $\infty$-category $\PS(\GE)$.  Moreover, the composite $\C \xra{\const} s\C \ra \locressC \simeq \PS(\GE)$ clearly coincides with the restricted Yoneda embedding.  We will generally omit this from the notation.
\end{rem}

\subsection{The spiral exact sequence}\label{subsection spiral}

\begin{defn}\label{define both homotopy groups}
Choose any $n \geq 0$ and any $\eps \in \GE^\delta$.
\begin{enumerate}
\item We define the corresponding \bit{classical homotopy group} functor to be the composite
\[ \pi_n \pi_\eps : s\C \xra{[S^\eps,-]^\lw_\C} s\Ab \xra{\pi_n} \Ab . \]
\item We define the corresponding \bit{natural homotopy group} functor to be either equivalent composite
\[ \begin{tikzcd}
& \locressC \arrow{rd}{\hom_\locressC(S^\eps,-)} \\
\pi_{n,\eps}^\natural : s\C \arrow{ru} \arrow{rd}[swap]{\enrhom_{s\C}(\const(S^\eps),-)} & & \Grp(\ho(\S_*)) \arrow{r}{\pi_n} & \Ab , \\
& \Grp(\ho(s\S_*)) \arrow{ru}[swap]{|{-}|}
\end{tikzcd} \]
where
\begin{itemize}
\item the commutativity of the square follows from the fact that $s\C_\res$
\begin{itemize}
\item is simplicial,
\item has $\const(S^\eps) \in s\C_\res^c$ cofibrant, and
\item has all objects fibrant,
\end{itemize}
and
\item the fact that the down-and-right functors land in h-group objects follows from the fact that $S^\eps \in \C$ is an h-cogroup object (so that $\const(S^\eps) \in s\C$ is as well).
\end{itemize}
\end{enumerate}
\end{defn}

\begin{defn}
Let $K \in s\S_*$, and let $X \in s\C$.  We define the \bit{reduced tensoring} of $X$ over $K$ to be the pushout
\[ \begin{tikzcd}
\pt_{s\S} \tensoring X \arrow{r} \arrow{d} & K \tensoring X \arrow{d} \\
\pt_{s\S} \tensoring 0_{s\C} \arrow{r} & K \redtensoring X
\end{tikzcd} \]
in $s\C$.  This assembles into an action $s\S_* \times s\C \ra s\C$.
\end{defn}

\begin{notn}
We write $D^n_\bD = \Delta^n / \Lambda^n_0 \in s\Set_* \subset s\S_*$ for the ``reduced pointed simplicial $n$-disk'' and $S^n_\bD = \Delta^n / \partial \Delta^n \in s\Set_* \subset s\S_*$ for the ``reduced pointed simplicial $n$-sphere''.
\end{notn}

\begin{obs}\label{obs cofiber seq both of ptd ssets and of ptd sspaces}
The canonical composite
\[ S^{n-1}_\bD \ra D^n_\bD \ra S^n_\bD \]
(where the first map is obtained by considering $\Delta^{n-1} \cong \Delta^{\{0,\ldots,n-1\}} \subset \Delta^n$) is a cofiber sequence not just in $s\Set_*$ but also in $s\S_*$.
\end{obs}

\begin{lem}
For any $n \geq 0$ and any $S^\eps \in \GE$, there is a natural isomorphism
\[ \pi_{n,\eps}^\natural(-) \cong [ S^n_\bD \redtensoring \const(S^\eps) , - ]_\locressC \]
in $\Fun(s\C,\Ab)$.
\end{lem}

\begin{proof}
In light of the facts
\begin{itemize}
\item that $s\C_\res$ is simplicial,
\item that $S^n_\bD \redtensoring \const(S^\eps) \in s\C_\res^c$ is cofibrant, and
\item that all objects of $s\C_\res$ are fibrant,
\end{itemize}
we have the string of natural isomorphisms
\begin{align*}
[ S^n_\bD \redtensoring \const(S^\eps) , - ]_\locressC & \cong \pi_0 | \enrhom_{s\C} ( S^n_\bD \redtensoring \const(S^\eps) , - ) | \\
& \cong \pi_0 \left| \lim \left( \begin{tikzcd}[ampersand replacement=\&]
\& \enrhom_{s\C} ( S^n_\bD \tensoring \const(S^\eps) , -) \arrow{d} \\
\enrhom_{s\C} ( \pt_{s\S} \tensoring 0_\C , - ) \arrow{r} \& \enrhom_{s\C}(\pt_{s\S} \tensoring \const(S^\eps) , - )
\end{tikzcd} \right) \right| \\
& \cong \pi_0 \left| \lim \left( \begin{tikzcd}[ampersand replacement=\&]
\& \enrhom_{s\S} ( S^n_\bD , \enrhom_{s\C}(\const(S^\eps) , - )) \arrow{d}{\ev_*} \\
\pt_{s\S} \arrow{r}[swap]{\ul{0}} \& \enrhom_{s\C}(\const(S^\eps),-)
\end{tikzcd} \right) \right| . \\
\end{align*}
In order to continue the string of isomorphisms, we make the following observations.
\begin{itemize}
\item The compatibility of $s\C_\res$ with $s\S_\KQ$ implies that the vertical map in this last expression is a fibration, so that we can commute the limit with the geometric realization.
\item As $\const(S^\eps) \in s\C_\res^c$ is cofibrant and all objects of $s\C_\res$ are fibrant, then $\enrhom_{s\C}(\const(S^\eps),-) : s\C \ra s\S_\KQ^f$ takes values in fibrant objects of $s\S_\KQ$.
\item The object $S^n_\bD \in s\S_\KQ^c$ is cofibrant.
\end{itemize}
Using these, we continue as
\begin{align*}
& \cong \pi_0 \lim \left( \begin{tikzcd}[ampersand replacement=\&]
\& { | \enrhom_{s\S} ( S^n_\bD , \enrhom_{s\C}(\const(S^\eps),-) ) | } \arrow{d}{|\ev_*|} \\
{|\pt_{s\S}|} \arrow{r}[swap]{|\ul{0}|} \& {| \enrhom_{s\C}(\const(S^\eps),-) |}
\end{tikzcd} \right) \\
& \cong \pi_0 \lim \left( \begin{tikzcd}[ampersand replacement=\&]
\& \hom_\S(|S^n_\bD| , | \enrhom_{s\C}(\const(S^\eps),-)|) \arrow{d}{|\ev_*|} \\
\pt_\S \arrow{r}[swap]{|\ul{0}|} \& {| \enrhom_{s\C}(\const(S^\eps),-) |}
\end{tikzcd} \right) \\
& \cong \pi_0 \lim \left( \begin{tikzcd}[ampersand replacement=\&]
\& \hom_\S ( S^n , \hom_\locressC(S^\eps,-)) \arrow{d}{\ev_*} \\
\pt_\S \arrow{r}[swap]{0} \& \hom_\locressC(S^\eps,-)
\end{tikzcd} \right) \\
& \cong \pi_0 \hom_{\S_*} ( S^n , \hom_\locressC(S^\eps,-)) \\
& \cong \pi_n \hom_\locressC(S^\eps,-) ,
\end{align*}
proving the claim.
\end{proof}

\begin{defn}
Let $K \in s\S_*$, and let $X \in s\C$.  We define the \bit{reduced cotensoring} of $K$ into $X$ to be the pullback
\[ \begin{tikzcd}
K \redcotensoring X \arrow{r} \arrow{d} & K \cotensoring X \arrow{d} \\
\pt_{s\S} \cotensoring 0_{s\C} \arrow{r} & \pt_{s\S} \cotensoring X
\end{tikzcd} \]
in $s\C$.  This assembles into an action $(s\S_*)^{op} \times s\C \ra s\C$.
\end{defn}

\begin{obs}
The reduced co/tensoring bifunctors participate into an evident two-variable adjunction
\[ \left( s\S_* \times s\C \xra{- \redtensoring -} s\C \ , \ (s\S_*)^{op} \times s\C \xra{ - \redcotensoring - } s\C \ , \ s\C^{op} \times s\C \xra{\enrhom_\C(-,-)} s\S_* \right) , \]
obtained by recognizing that the (enriched) hom-objects of $s\C$ are naturally pointed since $s\C$ has a zero object.
\end{obs}

\begin{obs}
If
\begin{itemize}
\item on the one hand we restrict the reduced tensoring bifunctor to the constant simplicial objects of $\C$ via the composite
\[ s\S_* \times \C \xra{\id_{s\S_*} \times \const} s\S_* \times s\C \xra{- \redtensoring -} s\C , \]
while
\item on the other hand we postcompose the reduced cotensoring bifunctor with the limit functor to obtain the composite
\[ (s\S_*)^{op} \times s\C \xra{- \redcotensoring - } s\C \xra{(-)_0} \C , \]
\end{itemize}
then we similarly obtain a two-variable adjunction
\[ \left( s\S_* \times \C \xra{- \redtensoring \const(-)} s\C \ , \ (s\S_*)^{op} \times s\C \xra{ (- \redcotensoring -)_0 } \C \ , \ \C^{op} \times s\C \xra{\enrhom_\C(-,-)} s\S_* \right) . \]
\end{obs}

\begin{notn}
In analogy with the ``generalized matching object'' bifunctor
\[ \Match_{(-)}(-) : s\S^{op} \times s\C \xra{(- \cotensoring -)_0} \C , \]
we write
\[ \redMatch_{(-)}(-) : (s\S_*)^{op} \times s\C \xra{(- \redcotensoring -)_0} \C \]
for the ``reduced generalized matching object'' bifunctor.
\end{notn}

\begin{defn}
We define the \bit{(nonabelian) normalized $n$-chains} functor to be
\[ N_n : s\C \xra{\redMatch_{D^n_\bD}(-)} \C , \]
and we define the \bit{(nonabelian) $n$-cycles} functor to be
\[ Z_n : s\C \xra{\redMatch_{S^n_\bD}(-)} \C . \]
Note that these would reduce to the usual notions if $\C$ were an abelian category.
\end{defn}

\begin{obs}\label{obs fiber seq of nonabelian cycles and chains}
The cofiber sequence $S^{n-1}_\bD \ra D^n_\bD \ra S^n_\bD$ in $s\S_*$ of \Cref{obs cofiber seq both of ptd ssets and of ptd sspaces} induces a fiber sequence
\[ Z_n \ra N_n \ra Z_{n-1} \]
in $\Fun(s\C,\C)$.
\end{obs}

\begin{lem}\label{normalized chains commutes with homotopy}
For any $S^\eps \in \GE$, there is a natural isomorphism
\[ [S^\eps , N_n(-)]_\C \cong N_n [S^\eps,-]^\lw_\C \]
in $\Fun(s\C,\Ab)$.
\end{lem}

\begin{proof}
Fix a test object $X \in s\C$.  As by definition $N_n(X) = \redMatch_{D^n_\bD}(X)$, we have a pullback square
\[ \begin{tikzcd}
N_n(X) \arrow{r} \arrow{d} & \Match_{D^n_\bD}(X) \arrow{d} \\
\Match_{\pt_{s\S}}(0_{s\C}) \arrow{r} & \Match_{\pt_{s\S}}(X)
\end{tikzcd} \]
in $\C$.  In light of the pushout square
\[ \begin{tikzcd}
\Lambda^n_0 \arrow{r} \arrow{d} & \Delta^n \arrow{d} \\
\Delta^0 \arrow{r} & D^n_\bD
\end{tikzcd} \]
both in $s\Set$ and in $s\S$, we also have a pullback square
\[ \begin{tikzcd}
\Match_{D^n_\bD}(X) \arrow{r} \arrow{d} & \Match_{\Delta^n}(X) \arrow{d} \\
\Match_{\Delta^0}(X) \arrow{r} & \Match_{\Lambda^n_0}(X)
\end{tikzcd} \]
in $\C$, which simplifies to a pullback square
\[ \begin{tikzcd}
\Match_{D^n_\bD}(X) \arrow{r} \arrow{d} & X_n \arrow{d} \\
X_0 \arrow{r} & \Match_{\Lambda^n_0}(X)
\end{tikzcd} \]
As the relevant corepresenting maps $\pt_{s\S} \ra D^n_\bD$ and $\Delta^0 \ra D^n_\bD$ in $s\Set \subset s\S$ coincide, we obtain the composite pullback square
\[ \begin{tikzcd}
N_n(X) \arrow{r} \arrow{d} & \Match_{D^n_\bD}(X) \arrow{r} & \Match_{\Delta^n}(X) \arrow{d} \\
\Match_{\pt_{s\S}}(0_{s\C}) \arrow{r} & \Match_{\pt_{s\S}}(X) \simeq \Match_{\Delta^0}(X) \arrow{r} & \Match_{\Lambda^n_0}(X)
\end{tikzcd} \]
in $\C$, which simplifies to a pullback square
\[ \begin{tikzcd}
N_n(X) \arrow{r} \arrow{d} & X_n \arrow{d} \\
0_\C \arrow{r} & \Match_{\Lambda^n_0}(X)
\end{tikzcd} \]
in $\C$.  Moreover, replacing $0 \in [n]$ with any $i \in [n]$, we obtain analogous pullback squares
\[ \begin{tikzcd}
\redMatch_{(\Delta^n/\Lambda^n_i)}(X) \arrow{r} \arrow{d} & X_n \arrow{d} \\
0_\C \arrow{r} & \Match_{\Lambda^n_i}(X)
\end{tikzcd} \]
in $\C$.  From here, the (dual of the corresponding cosimplicial) double induction argument of \cite[Chapter VIII, Lemma 1.8]{GJnew} yields the claim.
\end{proof}

\begin{lem}\label{present natural htpy}
For any $S^\eps \in \GE$, there is a natural exact sequence
\[ [ S^\eps , N_{n+1}(-) ]_\C \ra [ S^\eps , Z_n(-) ]_\C \ra \pi_{n,\eps}^\natural ( - ) \ra 0 \]
in $\Fun(s\C,\Ab)$.
\end{lem}

\begin{proof}
For any test object $X \in s\C$, we have
\[ \pi_{n,\eps}^\natural X = \pi_n \hom_\locressC(S^\eps,X) \cong \pi_0 \hom_{\S_*}(S^n , \hom_{\locressC}(S^\eps,X) ) . \]
Now, since $\const(S^\eps) \in s\C_\res^c$ and $X \in s\C_\res^f$, we have that $\enrhom_{s\C}(\const(S^\eps),X) \in s\S_\KQ^f$ and moreover $|\enrhom_{s\C}(\const(S^\eps),X)| \simeq \hom_{\locressC}(S^\eps,X)$.  On the other hand, $S^n_\bD \in s\S_\KQ^c$.  Since co/fibrancy in $(s\S_*)_\KQ$ is created in $s\S_\KQ$, the fundamental theorem of model $\infty$-categories applied to $(s\S_*)_\KQ$ implies that we have a surjection
\[ \hom_{s\S_*}(S^n_\bD , \enrhom_{s\C}(\const(S^\eps),X) ) \ra \hom_{\S_*}(S^n , \hom_{\locressC}(S^\eps,X) ) \]
in $\S$.  Applying $\pi_0$, by adjunction this yields a surjection
\[ [S^\eps , Z_n(X)]_\C \ra \pi_{n,\eps}^\natural X \]
in $\Set$.  As epimorphisms are $\Ab$ are created in $\Set$, this proves exactness at $\pi_{n,\eps}^\natural(-)$.

Now, suppose we are given an element of $\ker ( [S^\eps,Z_n(X)]_\C \ra \pi_{n,\eps}^\natural X)$: this is witnessed by an extension
\[ \begin{tikzcd}
S^n \arrow{r} \arrow{d} & \hom_{\locressC}(S^\eps,X) \\
\pt_{\S_*} \arrow[dashed]{ru}
\end{tikzcd} \]
in $\S_*$.  Since $D^{n+1}_\bD \in (s\S_{S^n_\bD /})_\KQ^c$ and $\enrhom_{s\C}(\const(S^\eps),X) \in (s\S_{S^n_\bD /})_\KQ^f$, the fundamental theorem of model $\infty$-categories applied to $(s\S_{S^n_\bD /})_\KQ$ implies that the above extension in $\S_*$ is presented by an extension
\[ \begin{tikzcd}
S^n_\bD \arrow{r} \arrow{d} & \enrhom_{s\C}(\const(S^\eps),X) \\
D^{n+1}_\bD \arrow[dashed]{ru}
\end{tikzcd} \]
in $s\S_*$.  This proves exactness at $[S^\eps,Z_n(-)]_\C$.
\end{proof}

\begin{cor}\label{nat iso in dimension 0}
There is a natural isomorphism $\pi_0 \pi_\eps ( - ) \cong \pi_{0,\eps}^\natural(-)$ in $\Fun(s\C,\Ab)$.
\end{cor}

\begin{proof}
Fix a test object $X \in s\C$.  Applying \Cref{present natural htpy} in the case that $n=0$, we obtain an isomorphism
\[ \coker ( [S^\eps,N_1 (X)]_\C \ra [S^\eps,Z_0 (X)]_\C ) \xra{\cong} \pi_{0,\eps}^\natural X \]
in $\Ab$.  Unwinding the definition of $Z_0 (X)$, we see that $Z_0 (X) \simeq X_0 \in \C$, so that
\[ [S^\eps,Z_0 (X)]_\C \cong [ S^\eps , X_0 ]_\C = ([ S^\eps , X]^\lw_\C)_0 . \]
Under this identification, unwinding the definition of $N_1 X$, we see that the image of the map
\[ [S^\eps,N_1 X]_\C \ra [S^\eps,Z_0 X]_\C \cong ([S^\eps,X]^\lw_\C)_0 \]
is the set of those 0-simplices in $[S^\eps,X]^\lw_\C \in s\Ab$ that are the ``source'' of a 1-simplex with ``target'' the basepoint 0-simplex $0 \in ([S^\eps,X]^\lw_\C)_0 \in \Ab$.  So we obtain an isomorphism
\[ \coker ( [S^\eps,N_1 (X)]_\C \ra [S^\eps,Z_0 (X)]_\C ) \cong \pi_0 \pi_\eps X , \]
from which the claim follows.
\end{proof}

\begin{constr}
For any object $X \in s\C$ and any $S^\eps \in \GE$, by \Cref{obs fiber seq of nonabelian cycles and chains} we have long exact sequences
\[ \cdots \ra [S^{\eps+1},Z_{n-1}(X)]_\C \ra [S^\eps,Z_n(X)]_\C \ra [S^\eps,N_n(X)]_\C \ra [S^\eps,Z_{n-1}(X)]_\C \]
in $\Ab$ (which actually continue indefinitely to the right as well since $\C$ is stable).  These splice together into an exact couple
\[ \begin{tikzcd}[column sep=0cm, row sep=1.5cm]
{[S^{\eps+i+1},Z_{n-1}(X)]_\C} \arrow[dashed]{rr}{(\eps+i+1) \squigra (\eps+i)} & & {[S^{\eps+i+1},Z_n(X)]_\C} \arrow{ld} \\
& {[S^{\eps+i+1},N_n(X)]_\C .} \arrow{lu}
\end{tikzcd} \]
Using Lemmas \ref{normalized chains commutes with homotopy} \and \ref{present natural htpy}, we can identify its derived long exact sequence as
\[ \begin{tikzcd}[row sep=0.1cm]
\cdots \arrow{r} & \pi_{i+1}\pi_\eps(X) \arrow{r}{\delta} & \pi_{i-1,\eps+1}^\natural(X) \arrow{r} & \pi_{i,\eps}^\natural(X) \arrow{r} & \pi_i\pi_\eps(X) \arrow{r}{\delta} & \cdots \\
& \cdots \arrow{r}{\delta} & \pi_{0,\eps+1}^\natural(X) \arrow{r} & \pi_{1,\eps}^\natural(X) \arrow{r} & \pi_1 \pi_\eps(X) \arrow{r} & 0 .
\end{tikzcd} \]
We refer to this as the \bit{spiral exact sequence}.
\end{constr}

\subsection{The localized spiral exact sequence}

In the end, we will not be interested in the natural and classical homotopy groups, but rather in their corresponding $E$-homology groups.

\begin{notn}
We simply write $E : s\C \xra{(E \otimes -)^\lw} s\C$ for the ``tensor levelwise with $E$'' functor.
\end{notn}

\begin{defn}
Choose any $n \geq 0$ and any $\beta \in \G^\delta$.
\begin{enumerate}
\item We define the corresponding \bit{classical $E$-homology group} functor to be the composite
\[ \pi_n E_\beta : s\C \xra{E} s\C \xra{\pi_n \pi_\beta} \Ab . \]
\item We define the corresponding \bit{natural $E$-homology group} functor to be the composite
\[ E_{n,\beta}^\natural : s\C \xra{E} s\C \xra{\pi_{n,\beta}^\natural} \Ab . \]
\end{enumerate}
When considered as indexed over all $\beta \in \G$ simultaneously, we write these functors simply as $\pi_n E_\star$ and $E_{n,\star}^\natural$, respectively.
\end{defn}

\begin{lem}\label{only one type of 0th E-homology}
There is a natural isomorphism $\pi_0 E_\beta(-) \cong E_{0,\beta}^\natural(-)$ in $\Fun(s\C,\Ab)$.
\end{lem}

\begin{proof}
This follows from \Cref{nat iso in dimension 0} and a colimit argument.
\end{proof}

\begin{constr}\label{construct localized spiral}
For any $X \in s\C$, the spiral exact sequence for $EX \in s\C$ with respect to any $\beta \in \G^\delta$ becomes 
\[ \begin{tikzcd}[row sep=0.1cm]
\cdots \arrow{r} & \pi_{i+1}E_\beta X \arrow{r}{\delta} & E_{i-1,\beta+1}^\natural X \arrow{r} & E_{i,\beta}^\natural X \arrow{r} & \pi_iE_\beta X \arrow{r}{\delta} & \cdots \\
& \cdots \arrow{r}{\delta} & E_{0,\beta+1}^\natural X \arrow{r} & E_{1,\beta}^\natural X \arrow{r} & \pi_1 E_\beta X \arrow{r} & 0 .
\end{tikzcd} \]
We refer to this as the \bit{localized spiral exact sequence}.
\end{constr}


\section{Algebraic topology}
\label{section.alg.top}

In this section, we add operadic structures to the mix.

\subsection{Foundations of algebraic topology}

\begin{defn}
By \bit{operad} we mean what might otherwise be called a ``single-colored $\infty$-operad''.  These are presented by monoids for the composition product in symmetric sequences in topological spaces or in simplicial sets (via the ``operadic nerve'' of Definition A.2.1.1.23).  We write $\Op$ for the $\infty$-category of operads. 
For any $\oO \in \Op$, we write $\oO(n) \in \Fun(B\SG_n,\S)$ for the space of $n$-ary operations, equipped with its canonical action of the symmetric group $\SG_n$.
\end{defn}

\begin{notn}
For any $\oO \in \Op$, we write $\Alg_\oO(\C)$ for the $\infty$-category of $\oO$-algebras in $\C$, and we write
\[ \free_\oO : \C \adjarr \Alg_\oO(\C) : \forget_\oO \]
for the corresponding free/forget monadic adjunction.
\end{notn}

\begin{obs}
The monad corresponding to the monadic adjunction $\free_\oO \adj \forget_\oO$ can be computed as
\[ \forget_\oO(\free_\oO(X)) \simeq \coprod_{n \geq 0} ( \oO(n) \tensoring X^{\otimes n})_{\SG_n} \]
(where we use the diagonal action to form the quotient).
\end{obs}

\begin{obs}
Any map $\oO \xra{\varphi} \oO'$ in $\Op$ determines an adjunction
\[ \varphi_* : \Alg_\oO(\C) \adjarr \Alg_{\oO'}(\C) : \varphi^* \]
between $\infty$-categories of algebras in $\C$, whose right adjoint is given by restriction of structure.  The assignment $\varphi \mapsto \varphi_*$ assembles into a functor
\[ \Alg_{(-)}(\C) : \Op \ra \PrL . \]
\end{obs}

\begin{rem}
We restrict to single-colored operads for simplicity, and because most operads of interest are single-colored. 
However, note that if one were interested in obtaining e.g.\! a commutative algebra $A \in \CAlg(\C)$ as well as a module $M \in \Mod_A(\C)$, one might proceed in steps, first using a single-colored obstruction theory in $\C$ to produce $A$, and then using a single-colored obstruction theory in $\Mod_A(\C)$ to produce $M$.
\end{rem}

\subsection{Simplicial algebraic topology}

\begin{defn}
Let $T \in s\Op$ be a simplicial object in operads.  We define the $\infty$-category $\Alg_T(s\C)$ of \bit{simplicial $T$-algebras in $\C$} to be the lax limit of the composite
\[ \bD^{op} \xra{T} \Op \xra{\Alg_{(-)}(\C)} \PrL . \]

\end{defn}

\begin{rem}
The composite
\[ \bD^{op} \xra{T} \Op \xra{\Alg_{(-)}(\C)} \PrL \xra{\forget_{\PrL}} \Cati \]
classifies a cocartesian fibration, which is in fact a bicartesian fibration; by (the dual of) \cite[Proposition 7.1]{GHN} (combined with Proposition T.5.5.3.13), its $\infty$-category of sections is precisely $\Alg_T(s\C)$.  Thus, we can think of a simplicial $T$-algebra $X = X_\bullet \in \Alg_T(s\C)$ as being specified by the following data:
\begin{itemize}
\item for each object $[n]^\circ \in \bD^{op}$, an object $X_n \in \Alg_{T_n}(\C)$;
\item for each morphism $[n]^\circ \xra{\varphi} [m]^\circ$ in $\bD^{op}$, a morphism from $X_n \in \Alg_{T_n}(\C)$ to $X_m \in \Alg_{T_m}(\C)$ in (the bicartesian fibration over $[1]$ corresponding to) the adjunction
\[ (T_\varphi)_* : \Alg_{T_n}(\C) \adjarr \Alg_{T_m}(\C) : (T_\varphi)^* \]
arising from the induced map $T_n \xra{T_\varphi} T_m$ in $\Op$, i.e.\! a point in the space
\[ \hom_{\Alg_{T_n}(\C)}(X_n,(T_\varphi)^*X_m) \simeq \hom_{\Alg_{T_m}(\C)}((T_\varphi)_*X_n,X_m) ; \]
\item higher coherence data for these structure maps corresponding to strings of composable morphisms in $\bD^{op}$.
\end{itemize}
\end{rem}

\begin{obs}\label{obs map between simplicial operads gives adjn}
Any map $T \xra{\varphi} T'$ in $s\Op$ determines an adjunction
\[ \varphi_* : \Alg_T(s\C) \adjarr \Alg_{T'}(s\C) : \varphi^* \]
between $\infty$-categories of simplicial algebras in $\C$, whose right adjoint is given by restriction of structure. In particular, taking $T$ to be trivial yields a monadic adjunction
\[ \free_{T'} : s\C \adjarr \Alg_{T'}(s\C) : \forget_{T'} , \]
whose underlying monad is computed levelwise. 
\end{obs}

\begin{obs}\label{obs algebras over a constant simplicial operad}
Let $\oO \in \Op$ be an operad, and consider the the corresponding constant simplicial operad $\const(\oO) \in s\Op$.  Since the resulting composite
\[ \bD^{op} \xra{\const(\oO)} \Op \xra{\Alg_{(-)}(\C)} \PrL \]
is constant at $\Alg_\oO(\C)$, it follows that we have a canonical equivalence
\[ \Alg_{\const(\oO)}(s\C) \simeq s(\Alg_\oO(\C)) . \]
\end{obs}

\begin{obs}\label{obs adjn from simplicial T-algebras}
For any $T \in s\Op$, we have a canonical composite adjunction
\[ \begin{tikzcd}[column sep=2cm]
\Alg_T(s\C) \horizadjnlabels{(\eta_T)_*}{(\eta_T)^*} & \Alg_{\const(|T|)}(s\C) \simeq s(\Alg_{|T|}(\C)) \horizadjnlabels{|{-}|}{\const} & \Alg_{|T|}(\C) ,
\end{tikzcd} \]
where
\begin{itemize}
\item the first adjunction follows by applying \Cref{obs map between simplicial operads gives adjn} to the component $T \xra{\eta_T} \const(|T|)$ of the unit of the adjunction $|{-}| : s\Op \adjarr \Op : \const(-)$;
\item the equivalence is that of \Cref{obs algebras over a constant simplicial operad}; and
\item the second adjunction is the colimit/constant adjunction in $\Alg_{|T|}(\C)$.
\end{itemize}
\end{obs}

\begin{notn}
For simplicity, we simply write
\[ |{-}| : \Alg_T(s\C) \adjarr \Alg_{|T|}(\C) : \const \]
for the composite adjunction of \Cref{obs adjn from simplicial T-algebras}.  When convenient and unambiguous, we will omit the right adjoint from the notation.
\end{notn}

\begin{lem}
The diagram
\[ \begin{tikzcd}
\Alg_T(s\C) \arrow{r}{|{-}|} \arrow{d}[swap]{\forget_T} & \Alg_{|T|}(\C) \arrow{d}{\forget_{|T|}} \\
s\C \arrow{r}[swap]{|{-}|} & \C
\end{tikzcd} \]
commutes.
\end{lem}

\begin{proof}
Both vertical functors are right adjoints which commute with sifted colimits.
\end{proof}

\begin{thm}\label{thm AlgTsCres}
There is a \bit{resolution model structure} on $\Alg_T(s\C)$, denoted $\Alg_T(s\C)_\res$; it is obtained by lifting the resolution model structure $s\C_\res$ along the adjunction 
\[ \free_T : s\C \adjarr \Alg_T(s\C) : \forget_T , \]
which therefore becomes a Quillen adjunction.  It enjoys the following properties.
\begin{enumerate}
\item Its weak equivalences and fibrations are created by pullback along the right adjoint $\forget_T$.
\item It is simplicial.
\item All objects are fibrant in it.
\item It is cofibrantly generated by the sets
\begin{align*}
I^{\Alg_T(s\C)}_\res &= \free_T(I^{s\C}_\res) = \{ \free_T(I^{s\S}_\KQ \tensoring \const(S^\eps)) \}_{S^\eps \in \GE} \\ & = \{ \free_T(\partial \Delta^n \tensoring \const(S^\eps)) \ra \free_T(\Delta^n \tensoring \const(S^\eps)) \}_{n \geq 0,S^\eps \in \GE}
\end{align*}
and
\begin{align*}
J^{\Alg_T(s\C)}_\res &= \free_T(J^{s\C}_\res) = \{ \free_T(J^{s\S}_\KQ \tensoring \const(S^\eps)) \}_{S^\eps \in \GE} \\ & = \{ \free_T(\Lambda^n_i \tensoring \const(S^\eps)) \ra \free_T(\Delta^n \tensoring \const(S^\eps)) \}_{0 \leq i \leq n \geq 1,S^\eps \in \GE} .
\end{align*}
\end{enumerate}
\end{thm}

\begin{proof}
The model structure follows from \Cref{thm Mres}, the enrichment and bitensoring over $s\S$ follows from \Cref{enr and bitens of algs over a monad on sD}, and their compatibility follows from \Cref{thm Mres simp}.
\end{proof}





\begin{notn}
Extending \Cref{E-equivces in sC}, we write $\bW_{E_\star^\lw} = \bW^{\Alg_T(s\C)}_{E_\star^\lw} \subset \Alg_T(s\C)$ for the preimage of $\bW^{s\C}_{E_\star^\lw} \subset s\C$ under the forgetful functor $\forget_T : \Alg_T(s\C) \ra s\C$.  Since $\bW^{s\C}_\res \subset \bW^{s\C}_{E_\star^\lw}$ by \Cref{E-equivce implies isomorphism on Etwo pages}, then also $\bW^{\Alg_T(s\C)}_\res \subset \bW^{\Alg_T(s\C)}_{E_\star^\lw}$.
\end{notn}

\begin{obs}\label{going from locpi to locE is a left localization}
In the end, our moduli spaces of interest will not be subgroupoids of the localization $\locresAlgT$, but rather of the further localization $\locEAlgT$.  However, in order to compute hom-spaces in this latter localization, it suffices to observe that the induced functor $\locresAlgT \ra \locEAlgT$ is actually a left localization: then, we can simply work in $\Alg_T(s\C)_\res$ but require that our target objects present local objects in $\locresAlgT$ (with respect to this left localization).  It follows from \Cref{identify localization of sNres} (and the monadic derived adjunction underlying the monadic Quillen adjunction $\free_T \adj \forget_T$) that $\locresAlgT$ is presentable, so we can apply the recognition result Proposition T.5.5.4.15: it suffices to show that the image in $\locresAlgT$ of $\bW_{E_\star^\lw} \subset \Alg_T(s\C)$ is strongly saturated (in the sense of Definition T.5.5.4.5).  The first two conditions follow from \cite[Lemma 1.5.2]{GH}, 
while the two-out-of-three property follows from the fact that it is ultimately pulled back from a subcategory $\bW_\KQ \subset s\A$ which has the two-out-of-three property.
\end{obs}

\begin{notn}
We will write $\leftloc_{E_\star^\lw} : \locresAlgT \adjarr \locEAlgT : \forget_{E_\star^\lw}$ for the left localization adjunction of \Cref{going from locpi to locE is a left localization}.
\end{notn}

\begin{rem}
The existence of a fully faithful right adjoint to the canonical functor $\locresAlgT \ra \locEAlgT$ should not be surprising: in \cite{GH}, this is constructed as a left Bousfield localization (cf.\! \cite[Theorems 1.4.9 and 1.5.1]{GH}).
\end{rem}

\begin{rem}
Taking $T$ to be trivial, we obtain a left localization adjunction $\leftloc_{E_\star^\lw} : \locressC \adjarr \locEsC : \forget_{E_\star^\lw}$.
\end{rem}

\begin{rem}
Whereas we have identified $\locressC$ as a nonabelian derived $\infty$-category, it appears that $\locEsC$ does not generally take this form.  It will become clear over the course of the construction that we really do need to be working in a nonabelian derived $\infty$-category.
\end{rem}

\subsection{Operads, revisited}

We give a brief unified treatment of all of the sorts of operads, their homotopy, and their related structures that we will be considering.  The material in this subsection is undergirded by the foundational work \cite{CH-enr-opds}.

\subsubsection{Operads and their algebras}

\begin{notn}
For an $\infty$-category $\V$, we write $\V^\SG = \Fun(\Set^\simeq , \V)$ for the $\infty$-category of symmetric sequences in $\V$.  Given $\oO \in \V^\SG$, we write $\oO(n) =\oO(\{1,\ldots,n\})$ for simplicity.  Assuming $\V$ has an initial object, we consider $\V \subset \V^\SG$ via left Kan extension along $\{ \pt_\Set \} \hookra \Set^\simeq$.  When $\V$ additionally admits a symmetric monoidal structure that commutes with colimits separately in each variable (e.g.\! if the symmetric monoidal structure is closed), the $\infty$-category $\V^\SG$ acquires a composition product monoidal structure $(\V^\SG , \circ , \unit_\V)$, algebras for which are precisely (``single colored'') $\V$-operads (a/k/a ``operads internal to $\V$'').  We denote the $\infty$-category of these by $\Op(\V)$, and write
\[ \free_{\Op(\V)} : \V^\SG \adjarr \Op(\V) : \forget_{\Op(\V)} \]
for the resulting monadic adjunction.  For brevity, we will simply say that $\V$ ``admits operads'' in this case.

When $\V$ is the $\infty$-category $\S$ of spaces (equipped with the cartesian symmetric monoidal structure), we (continue to) omit it from all our notation and terminology; in particular, we (continue to) refer to the objects of $\Op$ simply as ``operads''.  For emphasis, we may refer to objects of $\Op(\V)$ for some possibly unspecified $\V$ as ``internal operads''.
\end{notn}


\begin{notn}
Let $\D \in \LMod_\V(\Cati)$ be an $\infty$-category admitting an action of $\V$, and assume that $\D$ is cocomplete and finitely complete. 
Then for any $\oO \in \Op(\V)$ we denote by $\Alg_\oO(\D)$ the $\infty$-category of $\oO$-algebras in $\D$.  This is monadic over $\D$, and we write
\[ \free_\oO : \D \adjarr \Alg_\oO(\D) : \forget_\oO \]
for the monadic adjunction.
\end{notn}

\begin{obs}
Let $\V$ be an $\infty$-category that admits operads, and let $\I$ be any diagram $\infty$-category.  Then $\Fun(\I,\V)$ also admits operads: it inherits a componentwise symmetric monoidal structure from $\V$, and colimits (including the empty colimit) are computed componentwise.  In fact, it is not hard to see that we have an equivalence
\[ \Op(\Fun(\I,\V)) \simeq \Fun(\I,\Op(\V)) . \]
\end{obs}

\begin{prop}\label{model structure on sV-operads}
Let $\V$ be a symmetric monoidal $\infty$-category that admits operads and admits finite limits, and suppose that the unit object $\unit_\V \in \V$ is compact.  Then there exists a \bit{Boardman--Vogt model structure} on the $\infty$-category of $s\V$-operads, denoted $\Op(s\V)_\BV$, which is simplicial and participates in a Quillen adjunction
\[ \prod_{n \geq 0} s\S_\KQ \adjarr \Op(s\V)_\BV : \left( \hom^\lw_\V( \unit_\V , \forget_{\SG_n}( (-)(n) ) ) \right)_{n \geq 0} \]
of simplicial model $\infty$-categories, where $\free_{\SG_n} : \V \adjarr \Fun(B\SG_n , \V) : \forget_{\SG_n}$ denotes the left Kan extension adjunction for the canonical functor $\pt_\Cati \ra B\SG_n$.
\end{prop}

\begin{proof}
This follows from Theorems \ref{thm Mres} \and \ref{thm Mres simp}.
\end{proof}

\begin{rem}
In the end, we will only use \Cref{model structure on sV-operads} in situations when $\V$ is a 1-category.  In this case, the result is ultimately more-or-less just a consequence of \cite[Chapter II, \sec 4, Theorem 4]{QuillenHA}.  The name of the model structure pays homage to the foundational work \cite{BVHtpyInvt}, which introduced the study of homotopy-coherent algebraic structures.  The Boardman--Vogt model structure of \Cref{model structure on sV-operads} is also closely related to those of \cite[Theorems 3.1 and 3.2]{BMopds}, as explained in \cite[Example 3.3.1]{BMopds}.
\end{rem}

\begin{obs}
Let $\V$ and $\V'$ be two $\infty$-categories equipped with symmetric monoidal structures that commute with colimits separately in each variable.  Then any lax symmetric monoidal functor $\V \ra \V'$ induces a functor $\Op(\V) \ra \Op(\V')$.

We single out two particular cases of interest.
\begin{itemize}
\item The functor $- \tensoring \unit : \S \ra \C$ is symmetric monoidal (with respect to $(\S,\times,\pt_\S)$ and $(\C,\otimes,\unit)$).
\item The homology functor $E_\star : \C \ra \A$ is lax symmetric monoidal: for any $X , Y \in \C$, we have a canonical map $E_\star X \otimes_{E_\star} E_\star Y \ra E_\star(X \otimes Y)$ in $\A$, which takes the element
\[ \left( S^\beta \xra{\varphi} E \otimes X \right) \otimes \left( S^{\beta'} \xra{\varphi'} E \otimes Y \right) \]
to the element
 \[ \left( S^{\beta+\beta'} \simeq S^\beta \otimes S^{\beta'} \xra{\varphi \otimes \varphi'} E \otimes X \otimes E \otimes Y \simeq E^{\otimes 2} \otimes X \otimes Y \xra{\mu_E \otimes \id_X \otimes \id_Y} E \otimes X \otimes Y \right) . \]
\end{itemize}
It follows that the composite functor
\[ \S \xra{- \tensoring \unit} \C \xra{E_\star} \A \]
is lax symmetric monoidal, and hence induces a composite functor on internal operads, which for brevity we denote simply as
\[ E_\star : \Op = \Op(\S) \xra{\Op(-\tensoring \unit)} \Op(\C) \xra{\Op(E_\star)} \Op(\A) . \]
\end{obs}

\subsubsection{Resolutions of operads}\label{subsection resolutions of operads}

\begin{defn}\label{defn SG-free}
We say that an operad $\oO \in \Op$ is \bit{\PSG-free} if for each $n \geq 0$ the induced action of $\SG_n$ on $\pi_0(\oO(n))$ is free.
\end{defn}

\begin{rem}
As early in the literature as \cite[Definition 1.1]{MayGILS}, the term ``$\SG$-free'' is used to describe a point-set operad (e.g.\! in topological spaces) whose symmetric group actions are free at the point-set level.  Of course, such an operad need not present a \PSG-free operad in the sense of \Cref{defn SG-free}. 
\end{rem}

\begin{lem}\label{free opd is PSG-free}
The functor $\free_\Op : \S^\SG \ra \Op$ takes values in \PSG-free operads.
\end{lem}

\begin{proof}
This is immediate from the explicit description of $\free_\Op$ that follows from \cite[Proposition A.0.2 and Remark A.0.1]{Rezkthesis}.
\end{proof}

\begin{notn}
We simply write
\[ \Bar(-)_\bullet : \Op \xra{\Bar(\pt_\S , \forget_\Op \free_\Op, -)_\bullet} s\Op \]
for the bar construction on the monad $\forget_\Op \free_\Op \in \Alg(\End(\S^\SG))$ with respect to the left module given by the unit $\pt_\S \in \S^\SG$ and an unspecified operad considered as a right module.
\end{notn}

\begin{cor}
The functor $\Bar : \Op \ra s\Op$ takes values in levelwise \PSG-free simplicial operads, and admits a natural equivalence $|\Bar(-)_\bullet | \simeq \id_{\Op}$ in $\Fun(\Op,\Op)$.
\end{cor}

\begin{proof}
This follows from \Cref{free opd is PSG-free}.
\end{proof}

\begin{cor}
Given an operad $\oO$, suppose that $E_\star ( \oO(n)) \in \A_\proj$ for all $n \geq 0$.  Then $E_\star^\lw \Bar(\oO)_\bullet \in s\Op(\A) \simeq \Op(s\A)_\BV$ is cofibrant, and the augmentation $\Bar(\oO)_\bullet \ra \const(\oO)$ induces a weak equivalence $E_\star^\lw\Bar(\oO)_\bullet \we \const(E_\star \oO)$ in $\Op(s\A)_\BV$.
\end{cor}

\begin{proof}
This is immediate from the explicit description of $\free_\Op$ that follows from \cite[Proposition A.0.2 and Remark A.0.1]{Rezkthesis}.
\end{proof}

\begin{rem}
While we will ultimately be interested in a simplicial operad resolving our operad of primary interest, much of the theory goes through equally well for any simplicial operad.
\end{rem}

\section{Algebra}
\label{section.algebra}

\subsection{Foundations of algebra}


Recall that we write $\G^\delta = \pi_0(\G)$ for our chosen group of Picard elements, $\uA = \Fun(\G^\delta,\Ab)$ for the category of $\G^\delta$-graded abelian groups, and $\A = \Mod_{E_\star}(\uA)$ for the category of $E_\star$-modules in $\uA$.

\begin{ass}\label{assume flat}
We assume that $E_\star E \in \A$ is flat.
\end{ass}

\begin{notn}
It follows from \Cref{assume flat} that $(E_\star,E_\star E)$ is a Hopf algebroid in $\uA$.  We write $\tA = \Comod_{(E_\star,E_\star E)}$ for its category of left comodules (which in light of \Cref{assume flat} is abelian 
by \cite[Theorem A1.1.3]{RavGreen}), and we consider our homology theory as a functor $E_\star : \C \ra \tA$ taking values in $(E_\star,E_\star E)$-comodules.
\end{notn}

\begin{rem}
In general, the forgetful functor $\tA \xra{\forget_\tA} \A$ does not admit a left adjoint (e.g.\! it does not preserve products (see \cite[\sec 1.2]{HoveyComods})).
\end{rem}


\begin{obs}
For any $\beta \in \G^\delta$ we obtain an evident endofuctor $\Sigma^\beta : \tA \xra{\sim} \tA$.  This allows us to consider $\tA$ as enriched over $\A$, where for $M,N \in \tA$ we set
\[ \enrhom_\tA(M,N) = \{ \hom_\tA(\Sigma^\beta M,N) \}_{\beta \in \G^\delta} \in \A . \]
\end{obs}

\subsection{Compatibility}

The resolutions of operads considered in \Cref{subsection resolutions of operads} are necessary but not alone sufficient to render the obstruction theory to be tractable: we have introduced a new simplicial direction on the topology side, but we have not yet exerted any control on the simplicial direction that results on the algebra side.  Indeed, this will bring our $E$-homology computations into the realm of homotopical algebra, with its own attendant notions of ``cofibrant resolution'', and we must ensure that our homology functor $E_\star$ preserves resolutions.

We introduce three increasingly general notions of compatibility; the first is merely to fix ideas, the second is auxiliary, and the last is our real goal.

\begin{defn}
We say that an operad $\oO \in \Op$ is \bit{adapted} to $E$ if it comes with a corresponding monad $\oO_E \in \Alg(\End(\A))$ admitting a lift
\[ \begin{tikzcd}
\Alg_\oO(\C) \arrow[dashed]{r}{E_\star} \arrow{d}[swap]{\forget_\oO} & \Alg_{\oO_E}(\A) \arrow{d}{\forget_{\oO_E}} \\
\C \arrow{r}[swap]{E_\star} & \A
\end{tikzcd} \]
such that the following condition holds:
\begin{itemize}
\item for any $Z \in \C$ with $E_\star Z \in \A_\proj$, the natural map $\free_{\oO_E}(E_\star Z) \ra E_\star(\free_\oO(Z))$ is an isomorphism in $\Alg_{\oO_E}(\A)$.
\end{itemize}
\end{defn}

\begin{defn}
We say that a simplicial operad $T \in s\Op$ is \bit{adapted} to $E$ if it comes with a corresponding monad $T_E \in \Alg(\End(s\A))$ admitting a lift
\[ \begin{tikzcd}
\Alg_T(s\C) \arrow[dashed]{r}{E_\star^\lw} \arrow{d}[swap]{\forget_T} & \Alg_{T_E}(s\A) \arrow{d}{\forget_{T_E}} \\
s\C \arrow{r}[swap]{E_\star} & s\A
\end{tikzcd} \]
such that the following condition holds:
\begin{itemize}
\item for any $Z \in s\C$ with $E_\star^\lw Z \in s\A_\KQ^c$, the natural map $\free_{T_E}(E_\star^\lw Z) \ra E_\star^\lw(\free_T(Z))$ is an isomorphism in $\Alg_{T_E}(s\A)$.
\end{itemize}
\end{defn}

This has the following consequence.

\begin{lem}[{\cite[Lemma 1.4.15]{GH}}]\label{consequence of adaptedness}
If $T \in s\Op$ is adapted to $E$, then any cofibration between cofibrant objects in $\Alg_T(s\C)_\res$ is a retract of a map $X \xra{\varphi} Y$ such that the underlying map of degeneracy diagrams of $E_\star^\lw(\varphi)$ is isomorphic to one of the form $E_\star^\lw(X) \ra E_\star^\lw(X) \coprod T_E(M)$, where $M$ is s-free on an object of $\A_\proj$. \qed
\end{lem}

\begin{defn}\label{defn of homotopically adapted}
Suppose that the simplicial operad $T \in s\Op$ is adapted to $E_\star^\lw : s\C \ra s\A$.  We then say that $T$ is \bit{homotopically adapted} to $E$ if there exists a monad $\tT_E \in \Alg(\End(s\tA))$ which lifts the monad $T_E \in \Alg(\End(s\A))$ (i.e.\! they're intertwined by $s(\forget_\tA)$) and which admits a lift
\[ \begin{tikzcd}
\Alg_T(s\C) \arrow[dashed]{r}{E_\star^\lw} \arrow{d}[swap]{\forget_T} & \Alg_{\tT_E}(s\tA) \arrow{d}{\forget_{\tT_E}} \\
s\C \arrow{r}[swap]{E_\star} & s\tA
\end{tikzcd} \]
such that the following conditions hold:
\begin{itemize}
\item the adjunction $\free_{T_E} : s\A \adjarr \Alg_{T_E}(s\A) : \forget_{T_E}$ creates a simplicial model structure on $\Alg_{T_E}(s\A)$; and
\item there exists a simplicial model structure on $\Alg_{\tT_E}(s\tA)$ such that the forgetful functor $\Alg_{\tT_E}(s\tA) \ra \Alg_{T_E}(s\A)$ creates weak equivalences and preserves fibrations.
\end{itemize}
\end{defn}

Building on \Cref{consequence of adaptedness}, this has the following key consequence.

\begin{lem}[{\cite[Corollary 1.4.18]{GH}}]\label{model-categorical consequence of homotopical adaptedness}
If $T \in s\Op$ is homotopically adapted to $E$, then the induced functor $E_\star^\lw : \Alg_T(s\C)_\res \ra \Alg_{\tT_E}(s\tA)_{\pi_*}$ preserves both weak equivalences as well as cofibrations between cofibrant objects. 
\qed
\end{lem}

This result, in turn, has the following $\infty$-categorical significance.

\begin{cor}\label{infty-categorical consequence of homotopical adaptedness}
If $T \in s\Op$ is homotopically adapted to $E$, then the functor $E_\star^\lw : \locresAlgT \ra \locpiAlgtTE$ preserves colimits.
\end{cor}

\begin{proof}
This follows by combining \Cref{model-categorical consequence of homotopical adaptedness} with the theory of homotopy colimits in model $\infty$-categories of \cite[\sec 1.2]{MIC-qadjns}; more specifically, the model $\infty$-categories $\Alg_T(s\C)_\res$ and $\Alg_{\tT_E}(s\tA)_{\pi_*}$ are both cofibrantly generated and hence admit projective model structures, and the functor of model $\infty$-categories preserves projective cofibrancy by \Cref{model-categorical consequence of homotopical adaptedness}.
\end{proof}

\begin{rem}
Given two $\infty$-categories that admit finite coproducts and a functor between them that preserves these, applying the functor $\PS$ automatically gives a cocontinuous functor: up to further left localizations (which commute with colimits), this is precisely the situation that \Cref{infty-categorical consequence of homotopical adaptedness} addresses.  However, it is only through \Cref{identify localization of sNres} that we can identify it as such.
\end{rem}

\begin{ass}\label{assume homotopically adapted}
We henceforth assume that $T$ is homotopically adapted to $E$, and fix the corresponding monad $\tT_E \in \Alg(\End(s\tA))$.
\end{ass}

\begin{ex}
For any $\oO \in \Op$, we can take $T$ to be a cofibrant object of $\Op(s\Set)_\BV$ which presents it: each $T(n)$ will have a free $\SG_n$-action (as a simplicial set), and we can take $\tT_E$ to be the monad corresponding to the operad $E_\star T \in \Op(s\tA)$.
\end{ex}

\subsection{The module structure on the localized spiral exact sequence}\label{subsection module structure on localized spiral}

\begin{defn}\label{defn augmentation}
An \bit{augmentation} of the monad $\tT_E \in \Alg(\End(s\tA))$ is the data of a monad $\Phi \in \Alg(\End(\tA))$ and a natural isomorphism making the diagram
\[ \begin{tikzcd}
s\tA \arrow{r}{\tT_E} \arrow{d}[swap]{\pi_0} & s\tA \arrow{d}{\pi_0} \\
\tA \arrow{r}[swap]{\Phi} & \tA
\end{tikzcd} \]
commute, satisfying the diagrammatic coherence conditions of \cite[Definition 2.5.7]{GH}.  We write this as $\tT_E \da \Phi$, though note that this does not depict a morphism in any category.
\end{defn}

\begin{ass}\label{assume augmentation}
We henceforth assume the existence of an augmentation $\tT_E \da \Phi$.
\end{ass}

In order to describe the key consequence of \Cref{assume augmentation}, we must introduce some terminology.

\begin{defn}\label{define modules wrt Phi}
For any $A \in \Alg_\Phi(\tA)$, we define the category of \bit{$A$-modules} (\bit{relative to $\Phi$}) as the category $\Mod_A^\Phi(\tA) = \Ab(\Alg_\Phi(\tA)_{/A})$ of abelian group objects in its overcategory.  To align our notation with standard intuition, we write
\[ \begin{tikzcd}[row sep=0.1cm, column sep=1.5cm]
\tA & \Mod_A^\Phi(\tA) \arrow{l}[swap]{\forget_A} \arrow{r}{- \ltimes A} & \Alg_\Phi(\tA) \\
\ker^\tA(\varphi) & (B \xra{\varphi} A) \arrow[mapsto]{l} \arrow[mapsto]{r} & B
\end{tikzcd} \]
for the two forgetful functors.
\end{defn}

\begin{lem}[{\cite[Propositions 2.5.9 and 2.5.10]{GH}}]\label{augn gives adjn and module structures}
There exists a canonical lift
\[ \begin{tikzcd}[column sep=1.5cm]
& & \Alg_\Phi(\tA) \arrow{d}{\forget_\Phi} \\
\Alg_{\tT_E}(s\tA) \arrow[dashed]{rru}{\pi_0} \arrow{r}[swap]{\forget_{\tT_E}} & s\tA \arrow{r}[swap]{\pi_0} & \tA ,
\end{tikzcd} \]
and this lift is the left adjoint in an adjunction
\[ \pi_0 : \Alg_{\tT_E}(s\tA) \adjarr \Alg_\Phi(\tA) : \const . \]
Moreover, for any $X \in \Alg_{\tT_E}(s\tA)$ and any $n \geq 1$, the object $\pi_n X \in \tA$ admits a canonical lift through the functor
\end{lem}
\eqnqed{ \Mod_{\pi_0 X}^\Phi(\tA) \xra{\forget_{\pi_0 X}} \tA . }

\begin{cor}\label{0th classical homology lifts to Phi-algs}
There exists a canonical lift
\[ \begin{tikzcd}[column sep=1.5cm]
& & & \Alg_\Phi(\tA) \arrow{d}{\forget_\Phi} \\
\Alg_T(s\C) \arrow[dashed]{rrru}{\pi_0 E_\star^\lw} \arrow{r}[swap]{E_\star^\lw} & \Alg_{\tT_E}(s\tA) \arrow{r}[swap]{\forget_{\tT_E}} & s\tA \arrow{r}[swap]{\pi_0} & \tA .
\end{tikzcd} \]
Moreover, for any $X \in \Alg_T(s\C)$ and any $n \geq 1$, the object $\pi_n E_\star^\lw X \in \tA$ admits a canonical lift through the functor
\end{cor}
\eqnqed{ \Mod_{\pi_0 E_\star^\lw X}^\Phi(\tA) \xra{\forget_{\pi_0 E_\star^\lw X}} \tA . }

We record a useful fact about the adjunction of \Cref{augn gives adjn and module structures}.

\begin{lem}\label{pi0 and const are a qadjn on tTE and Phi algs}
The adjunction of \Cref{augn gives adjn and module structures} lifts to a Quillen adjunction
\[ \pi_0 : \Alg_{\tT_E}(s\tA)_{\pi_*} \adjarr \Alg_\Phi(\tA)_\triv : \const , \]
whose derived adjunction is a left localization adjunction.
\end{lem}

\begin{proof}
To see that this is a Quillen adjunction, we observe that the left adjoint
\begin{itemizesmall}
\item trivially preserves cofibrations, and
\item preserves acyclic cofibrations by definition of the subcategory $\bW_{\pi_*} \subset \Alg_{\tT_E}(s\tA)$.
\end{itemizesmall}
Then, to see that the derived adjunction is a left localization adjunction, we check that its counit is a componentwise equivalence.  Since every object of $\Alg_\Phi(\tA)_\triv$ is fibrant, the composite
\[ \Alg_\Phi(\tA) \xra{\const} \Alg_{\tT_E}(s\tA) \ra \locpiAlgtTE \]
computes the derived right adjoint $\bbR \const$.  Now, let
\[ \es_{\Alg_{\tT_E}(s\tA)} \cofibn \bbQ \const(A) \we \const(A) \]
be a cofibrant replacement in $\Alg_{\tT_E}(s\tA)_{\pi_*}$.  Then by definition the induced map
\[ \pi_0(\bbQ \const(A)) \ra \pi_0(\const(A)) \cong A \]
is an isomorphism in $\Alg_\Phi(\tA)$.  So the counit is indeed an equivalence.
\end{proof}

\begin{notn}
As both functors in the Quillen adjunction of \Cref{pi0 and const are a qadjn on tTE and Phi algs} preserve all weak equivalences, we will simply write
\[ \pi_0 : \locpiAlgtTE \adjarr \Alg_\Phi(\tA) : \const \]
for its derived adjunction (as opposed to $\bbL \pi_0 \adj \bbR \const$).  Moreover, we will often leave implicit both the right Quillen functor as well as its derived right adjoint.
\end{notn}

We have just seen that classical homology groups admit certain algebraic structure.  In fact, natural homology groups do too.

\begin{lem}[{\cite[Examples 3.1.14 and 3.1.17]{GH}}]\label{0th natural homology lifts to Phi-algs}
There exists a canonical lift
\[ \begin{tikzcd}[column sep=1.5cm]
& \Alg_\Phi(\tA) \arrow{d}{\forget_\A \circ \forget_\tA \circ \forget_\Phi} \\
\Alg_T(s\C) \arrow[dashed]{ru}{E_{0,\star}^\natural} \arrow{r}[swap]{E_{0,\star}^\natural} & \uA .
\end{tikzcd} \]
Moreover, for any $X \in \Alg_T(s\C)$ and any $n \geq 1$, the object $E_{n,\star}^\natural X \in \uA$ admits a canonical lift through the functor
\end{lem}
\eqnqed{ \Mod^\Phi_{E_{0,\star}^\natural X}(\tA) \xra{\forget_\A \circ \forget_\tA \circ \forget_{E_{0,\star}^\natural X}} \uA . }

Moreover, these algebraic structures are compatible in the following way.

\begin{lem}[{\cite[Corollary 3.1.18]{GH}}]\label{iso of two types of 0th E-homology is as Phi-algs}
The isomorphism $\pi_0 E_\star^\lw(-) \cong E_{0,\star}^\natural(-)$ in $\Fun(\Alg_T(s\C),\uA)$ of \Cref{only one type of 0th E-homology} is compatible with the lifts to $\Fun(\Alg_T(s\C) , \Alg_\Phi(\tA))$ of \Cref{0th classical homology lifts to Phi-algs} \and \Cref{0th natural homology lifts to Phi-algs}. \qed
\end{lem}

\begin{notn}
For simplicity, we may write $E_0 : \Alg_T(s\C) \ra \Alg_\Phi(\tA)$ for the functor $\pi_0 E_\star^\lw \cong E_{0,\star}^\natural$.
\end{notn}

\begin{lem}[{\cite[Example 3.1.13]{GH}}]\label{shifts lift to A-modules}
For any $A \in \Alg_\Phi(\tA)$ and any $n \geq 1$, the endofunctor $\Omega^n : \tA \xra{\sim} \tA$ lifts to an endofunctor $\Omega^n : \Mod_A^\Phi(\tA) \xra{\sim} \Mod_A^\Phi(\tA)$. \qed
\end{lem}

\begin{rem}
In fact, if we define $\Sigma^\beta_+ S^\eps = (\unit \oplus S^\beta) \otimes S^\eps$, then the construction of \cite[Example 3.1.13]{GH} generalizes to define lifted endofunctors $\Omega^\beta : \Mod_A^\Phi(\tA) \xra{\sim} \Mod_A^\Phi(\tA)$ for any $\beta \in \G^\delta$. 
\end{rem}

We can now give the module structure on the localized spiral exact sequence.

\begin{prop}[{\cite[Corollary 3.1.18]{GH}}]
For any $X \in \Alg_T(s\C)$, assembling the localized spiral exact sequence in $\Ab$ over all $\beta \in \G^\delta$, we obtain an exact sequence
\[ \begin{tikzcd}[row sep=0.1cm]
\cdots \arrow{r} & \pi_{i+1}E_\star X \arrow{r}{\delta} & \Omega ( E_{i-1,\star}^\natural X )\arrow{r} & E_{i,\star}^\natural X \arrow{r} & \pi_iE_\star X \arrow{r}{\delta} & \cdots \\
& \cdots \arrow{r}{\delta} & \Omega ( E_{0,\star}^\natural X) \arrow{r} & E_{1,\star}^\natural X \arrow{r} & \pi_1 E_\star X \arrow{r} & 0
\end{tikzcd} \]
in $\Mod_{E_0 X}^\Phi(\tA)$. \qed
\end{prop}

\subsection{The module structure on the spiral exact sequence}\label{module structure on unlocalized spiral}

We will make certain computations \textit{before} appealing to a colimit argument, and for these we will need to obtain analogous structure on the unlocalized spiral exact sequence.  In fact, this is an input to the module structure on the localized spiral exact sequence (via a colimit argument, as always), but the algebraic objects at play are slightly less familiar so we have reversed their order here.  However, the story is nearly identical to that of \Cref{subsection module structure on localized spiral}, and so we only highlight the key points.

\begin{notn}
We write $T(\GE) \subset \locresAlgT$ for the full subcategory spanned by the image of the composite
\[ \GE \hookra \C \xhookra{\const} s\C \xra{\free_T} \Alg_T(s\C) \ra \locresAlgT . \]
\end{notn}

\begin{obs}
The functor $\GE \xra{\free_T} T(\GE)$ preserves coproducts, and so induces a forgetful functor $\PSd(T(\GE)) \xra{\forget_{T(\GE)}} \PSd(\GE)$.
\end{obs}

\begin{defn}\label{define modules wrt TGE}
For any $A \in \PSd(T(\GE))$, we define the category of \bit{$A$-modules} (\bit{relative to $T(\GE)$}) as the category $\Mod_A^{T(\GE)}(\PSd(\GE)) = \Ab(\PSd(T(\GE))_{/A})$ of abelian group objects in its overcategory.  This admits two forgetful functors, which we denote by
\[ \begin{tikzcd}[row sep=0.1cm, column sep=1.5cm]
\PSd(\GE) & \Mod_A^{T(\GE)}(\PSd(\GE)) \arrow{l}[swap]{\forget_A} \arrow{r}{- \ltimes A} & \PSd(T(\GE)) . \\
\ker^{\PSd(\GE)}(\varphi) & (B \xra{\varphi} A) \arrow[mapsto]{l} \arrow[mapsto]{r} & B
\end{tikzcd} \]
\end{defn}

The following example will be of use later.

\begin{notn}
Let $A \in \Alg_\Phi(\tA)$.  Then we obtain an object $\Yo^E(A) \in \PSd(T(\GE))$ by declaring that
\[ \Yo^E(A)(\free_T(S^\eps)) = \hom_{\Alg_\Phi(\tA)}(\pi_0 E_\star^\lw \free_T(S^\eps) , A). \]
Similarly, if $M \in \Mod_A^\Phi(\tA)$, we obtain an object $\Yo^E(M) \in \Mod_{\Yo^{E_\star^\lw}(A)}^{T(\GE)}(\PSd(\GE))$ by declaring that
\[ \Yo^E(M)(S^\eps) = \hom_{\tA}(\pi_0 E_\star^\lw S^\eps , M) ; \]
technically, the $A$-action arises through Definitions \ref{define modules wrt Phi} \and \ref{define modules wrt TGE} (in terms of abelian objects in overcategories), but morally it just comes from postcomposition.
\end{notn}

\begin{obs}\label{both types of 0th htpy lift to PSd(TGE)}
As the functor $\Alg_T(s\C) \ra \ho(\Alg_T(s\C))$ preserves finite coproducts, by adjunction both composite functors
\[ \begin{tikzcd}[column sep=1.5cm]
\Alg_T(s\C) \arrow{r}{\forget_T} & s\C \arrow[transform canvas={yshift=0.5ex}]{r}{\pi_{0,\star}^\natural} \arrow[transform canvas={yshift=-0.5ex}]{r}[swap]{\pi_0 \pi_\star^\lw} & \PSd(\GE)
\end{tikzcd} \]
admit lifts through $\PSd(T(\GE)) \xra{\forget_{T(\GE)}} \PSd(\GE)$ for any $n \geq 0$.
\end{obs}

\begin{lem}
The isomorphisms $\pi_0 \pi_\eps^\lw (-) \cong \pi_{0,\eps}^\natural(-)$ in $\Fun(\Alg_T(s\C) , \Fun(\GE^\delta,\Ab))$ of \Cref{nat iso in dimension 0} are compatible with the lifts to $\Fun(\Alg_T(s\C),\PSd(T(\GE)))$ of \Cref{both types of 0th htpy lift to PSd(TGE)}.
\end{lem}

\begin{notn}
For simplicity, we may write $\pi_0 : \Alg_T(s\C) \ra \PSd(T(\GE))$ for the functor $\pi_0 \pi_\star^\lw \cong \pi_{0,\star}^\natural$.
\end{notn}

\begin{prop}[{\cite[Theorem 3.1.15]{GH}}]
For any $X \in \Alg_T(s\C)$, assembling the spiral exact sequence in $\Ab$ over all $\eps \in \GE^\delta$, we obtain an exact sequence
\[ \begin{tikzcd}[row sep=0.1cm]
\cdots \arrow{r} & \pi_{i+1}\pi_\star X \arrow{r}{\delta} & \Omega ( \pi_{i-1,\star}^\natural X )\arrow{r} & \pi_{i,\star}^\natural X \arrow{r} & \pi_i \pi_\star X \arrow{r}{\delta} & \cdots \\
& \cdots \arrow{r}{\delta} & \Omega ( \pi_{0,\star}^\natural X) \arrow{r} & \pi_{1,\star}^\natural X \arrow{r} & \pi_1 \pi_\star X \arrow{r} & 0
\end{tikzcd} \]
in $\Mod_{\pi_0 X}^{T(\GE)}(\PSd(\GE))$. \qed
\end{prop}

\section{Homotopical algebra}
\label{section.homotopical.algebra}


\subsection{Postnikov towers in algebra}

\begin{defn}
For any $n \geq 0$, an object $X \in \locpiAlgtTE$ is called \bit{$n$-truncated} if $\pi_{>n} X = 0$.  Such objects form a full subcategory $\locpiAlgtTE^{\leq n} \subset \locpiAlgtTE$, and as $n$ varies these subcategories are evidently nested as
\[ \locpiAlgtTE \hookla \cdots \hookla \locpiAlgtTE^{\leq 1} \hookla \locpiAlgtTE^{\leq 0} . \]
By presentability, 
these inclusions admit left adjoints, and we denote the corresponding left localization adjunctions by
\[ P_n^\alg : \locpiAlgtTE \adjarr \locpiAlgtTE^{\leq n} : \forget_n^\alg . \]
We therefore obtain a tower of functors
\[ \id_{\locpiAlgtTE} \ra \cdots \ra P_1^\alg \ra P_0^\alg . \]
We refer to its value on an object of $\locpiAlgtTE$ as its \bit{Postnikov tower}.  We write
\[ \id_{\locpiAlgtTE} \xra{\tau_n^\alg} P_n^\alg \]
for the natural transformation (or for its composite with $\forget_m^\alg$ for any $m \geq 0$), which we refer to as the \bit{$n$-truncation map}.
\end{defn}

\subsection{Cohomology}

Our obstructions will take place in (Andr\'e--Quillen) cohomology groups in $\locpiAlgtTE$.  We will only need to consider them with respect to a base object lying in $\Alg_\Phi(\tA)$, so we restrict to this special case. 

We begin by defining the representing objects for cohomology.

\begin{defn}
Let $A \in \Alg_\Phi(\tA)$, let $M \in \Mod_A^{\Phi}(\tA)$, and let $n \geq 1$.
\begin{enumerate}
\item We say that an object $X \in \locpiAlgtTE$ is of \bit{type $K_A$} if there exists an equivalence $X \simeq A$, i.e.\! if
\begin{itemize}
\item there exists an isomorphism $\pi_0 X \cong A$ in $\Alg_\Phi(\tA)$, and
\item $\pi_i X =0$ for $i > 0$.
\end{itemize}
\item We say that an object $Y \in \locpiAlgtTE$ is of \bit{type $K_A(M,n)$} if
\begin{itemize}
\item there exists an isomorphism $\pi_0 Y \cong A$ in $\Alg_\Phi(\tA)$,
\item there exists an isomorphism $\pi_n Y \cong M$ via the resulting equivalence of categories $\Mod_{\pi_0 Y}^\Phi(\tA) \simeq \Mod_A^\Phi(\tA)$, and
\item $\pi_i Y = 0$ for $i \notin \{ 0,n\}$.
\end{itemize}
\item We say that a morphism $X \ra Y$ in $\locpiAlgtTE$ is of \bit{type $\vec{K}_A(M,n)$} if
\begin{itemize}
\item $X$ is of type $K_A$,
\item $Y$ is of type $K_A(M,n)$, and
\item the map $\pi_0 X \ra \pi_0 Y$ is an isomorphism in $\Alg_\Phi(\tA)$.
\end{itemize}
\item We say that an object is of \bit{type $K_A(M,0)$} in $\locpiAlgtTE$ if it is of type $K_{M \ltimes A}$, and we say that a morphism in $\locpiAlgtTE$ is of \bit{type $\vec{K}_A(M,0)$} if it admits an equivalence to the map $\const(A \ra M \ltimes A)$.
\end{enumerate}
We refer to objects of type $K_A$ and $K_A(M,n)$ collectively as \bit{algebraic Eilenberg--Mac Lane objects}, and to morphisms of type $\vec{K}_A(M,n)$ collectively as \bit{algebraic Eilenberg--Mac Lane morphisms}.  We will see that these all exist and are unique in Propositions \ref{moduli of alg EM objects} \and \ref{moduli of alg EM maps}; 
justified by this, we may simply write $K_A$ or $K_A(M,n)$ for convenience when referring to an algebraic Eilenberg--Mac Lane object of the indicated type.
\end{defn}

\begin{obs}
Suppose that $X \ra Y$ is morphism in $\locpiAlgtTE$ of type $\vec{K}_A(M,n)$ for some $n \geq 1$.  Then $P_0^\alg(Y)$ is of type $K_A$, and the composite
\[ X \ra Y \xra{\tau_0^\alg} P_0^\alg(Y) \]
with the canonical $0$-truncation map is an equivalence.  Fixing an equivalence $X \simeq A$ then allows us to consider
\[ K_A(M,n) \in \locpiAlgtTE_{A/\!/A} . \]
Of course, such consideration is immediate for $n=0$.
\end{obs}

\begin{obs}\label{alg EM spectrum}
For any $n \geq 0$, taking the pullback of a map of type $\vec{K}_A(M,n+1)$ with itself yields a fiber square
\[ \begin{tikzcd}
K_A(M,n) \arrow{r}{\tau_0^\alg} \arrow{d}[swap]{\tau_0^\alg} & A \arrow{d} \\
A \arrow{r} & K_A(M,n+1)
\end{tikzcd} \]
in $\locpiAlgtTE$.  Hence, the objects 
\[ \left\{ K_A(M,n) \in \locpiAlgtTE_{A/\!/A} \right\}_{n \geq 0} \]
assemble into an $\Omega$-spectrum object
\[ K_A M \in \Stab \left( \locpiAlgtTE_{A/\!/A} \right) . \]
\end{obs}

\begin{defn}\label{define AQcoh}
Let $A \in \Alg_\Phi(\tA)$, let $M \in \Mod_A^\Phi(\tA)$, and let $n \geq 0$.  Suppose that $k \ra A = \const(A)$ is a morphism in $\locpiAlgtTE$, and use this to consider $K_A(M,n) \in \locpiAlgtTE_{k/\!/A}$.  Then, choose any object $X \in \locpiAlgtTE_{k/\!/A}$.
\begin{enumerate}
\item We define the \bit{$n\th$ (Andr\'e--Quillen) cohomology group} of $X$ with coefficients in $M$ to be the abelian group
\[ H^n_{\tT_E}(X/k;M) = [ X , K_A(M,n) ]_{\locpiAlgtTE_{k/\!/A}} \in \Ab . \]
\item We define the \bit{$n\th$ (Andr\'e--Quillen) cohomology space} of $X$ with coefficients in $M$ to be the based space
\[ \ms{H}^n_{\tT_E}(X/k;M) = \hom_{\locpiAlgtTE_{k/\!/A}}(X,K_A(M,n)) \in \S_* . \]
\end{enumerate}
Thus, we have that
\[ H^n_{\tT_E}(X/k;M) = \pi_0 ( \ms{H}^n_{\tT_E}(X/k;M) ) , \]
and moreover it follows from \Cref{alg EM spectrum} that
\[ H^{n-i}_{\tT_E}(X/k;M) = \pi_i ( \ms{H}^n_{\tT_E}(X/k;M)) \]
for $0 \leq i \leq n$.  (In particular, cohomology groups are indeed abelian groups, and cohomology spaces are infinite loopspaces.)
\end{defn}

\begin{obs}\label{AQcoh lexseq}
In the setting of \Cref{define AQcoh}, there is an evident pullback square
\[ \begin{tikzcd}
\hom_{\locpiAlgtTE_{X/\!/A}}(A,K_A(M,n)) \arrow{r} \arrow{d} & \hom_{\locpiAlgtTE_{k/\!/A}}(A,K_A(M,n)) \arrow{d} \\
\{ X \ra A \ra K_A(M,n) \} \arrow[hook]{r} & \hom_{\locpiAlgtTE_{k/\!/A}}(X,K_A(M,n))
\end{tikzcd} \]
in $\S_*$, which is by definition a pullback square
\[ \begin{tikzcd}
\ms{H}^n_{\tT_E}(A/X;M) \arrow{r} \arrow{d} & \ms{H}^n_{\tT_E}(A/k;M) \arrow{d} \\
\{ 0 \} \arrow[hook]{r} & \ms{H}^n_{\tT_E}(X/k;M).
\end{tikzcd} \]
This gives rise to a long exact sequence
\[ \begin{tikzcd}[row sep=0.1cm]
0 \arrow{r} & H^0_{\tT_E}(A/X;M) \arrow{r} & H^0_{\tT_E}(A/k;M) \arrow{r} & H^0_{\tT_E}(X/k;M) \arrow{r}{\delta} & \cdots \\
\cdots \arrow{r}{\delta} & H^n_{\tT_E}(A/X;M) \arrow{r} & H^n_{\tT_E}(A/k;M) \arrow{r} & H^n_{\tT_E}(X/k;M) \arrow{r}{\delta} & H^{n+1}_{\tT_E}(A/X;M) \arrow{r} & \cdots
\end{tikzcd} \]
in $\Ab$; exactness at $H^0_{\tT_E}(A/X;M)$ follows from the fact that the space
\[ \hom_{\locpiAlgtTE_{k/\!/A}}(X,K_A(M,0)) \simeq \hom_{\Alg_\Phi(\tA)_{\pi_0 k /\!/A}}(\pi_0 X, M \ltimes A) \]
is discrete (and so in particular has vanishing $\pi_1$).  We refer to this as the \bit{transitivity sequence}.
\end{obs}

\begin{rem}
When $M \in \tA$ is an \textit{extended} comodule, these cohomology computations reduce to analogous ones in $\locpiAlgTE$ (see \cite[Proposition 2.4.7]{GH}).
\end{rem}

\subsection{Moduli spaces in algebra}

We will be interested in various moduli spaces of algebraic objects: ultimately, our obstruction theory will be based on homotopy groups in the $\infty$-category $\locpiAlgtTE$.

In order to be able to effectively control these homotopy groups, we need to make the following assumption.

\begin{ass}\label{assume BM excision}
We assume that $\locpiAlgtTE$ has \bit{Blakers--Massey excision}: for any pushout square
\[ \begin{tikzcd}
X \arrow{r}{\psi} \arrow{d}[swap]{\varphi} & Z \arrow{d}{\rho} \\
Y \arrow{r} & W
\end{tikzcd} \]
such that $\pi_{<m}(\fib(\varphi)) = \pi_{<n}(\fib(\psi)) = 0$, the map $\pi_k(\fib(\varphi)) \ra \pi_k(\fib(\rho))$ is an isomorphism for $k < m+n$ and is surjective for $k=m+n$.
\end{ass}

\begin{cor}[{\cite[Corollary 2.3.15]{GH}}]\label{BM lexseq}
Suppose that
\[ \begin{tikzcd}
X \arrow{r}{\psi} \arrow{d}[swap]{\varphi} & Z \arrow{d} \\
Y \arrow{r} & W
\end{tikzcd} \]
is a pushout square in $\locpiAlgtTE$ such that $\pi_{<m}(\fib(\varphi)) = \pi_{<n}(\fib(\psi)) = 0$.  Then there is an induced partial long exact sequence
\[ \begin{tikzcd}[row sep=0.1cm, ampersand replacement=\&]
\& \& \pi_{m+n}(Y) \oplus \pi_{m+n}(Z) \arrow{r} \& \pi_{m+n}(W) \arrow{r}{\delta} \& \pi_{m+n-1}(X) \arrow{r} \& \cdots \\
\cdots \arrow{r}{\delta} \& \pi_0(X) \arrow{r} \& \pi_0(Y) \oplus \pi_0(Z) \arrow{r} \& \pi_0(W) \arrow{r} \& 0
\end{tikzcd} \]
in $\tA$, which we refer to as the \bit{Blakers--Massey long exact sequence}.
\end{cor}

\begin{rem}
\Cref{assume BM excision} holds in examples of interest, e.g.\! when $\tT_E$ is the monad corresponding to an operad $E_\star(T) \in \Op(s\tA)$ for any $T \in \Op(s\Set)$ (see \cite[Theorem 2.3.13 and Remark 2.3.14]{GH}).
\end{rem}

Our moduli spaces will be related by the following natural construction.

\begin{constr}\label{alg difference constrn}
Let $X \xra{\varphi} Y$ be a map in $\locpiAlgtTE$, and write
\[ \pozalg(\varphi) = Y \underset{X}{\coprod} P_0^\alg X = \colim \left( \begin{tikzcd}
X \arrow{r}{\tau_0^\alg} \arrow{d}[swap]{\varphi} & P_0^\alg (X) \\ Y \end{tikzcd} \right) \]
for the indicated pushout.  For any $n \geq 0$ we obtain a commutative diagram
\[ \begin{tikzcd}
X \arrow{r}{\tau_0^\alg} \arrow{d}[swap]{\varphi} & P_0^\alg( X) \arrow{rd}{\delta_n(\varphi)} \arrow{d} \\
Y \arrow{r} & \pozalg(\varphi) \arrow{r}[swap]{\tau_{n+1}^\alg} & P_{n+1}^\alg(\pozalg(\varphi))
\end{tikzcd} \]
in $\locpiAlgtTE$, and we refer to the map $\delta_n(\varphi)$ as the \bit{$n\th$ difference construction} on the map $\varphi$.  This defines an augmented endofunctor on $\Fun([1],\locpiAlgtTE)$.  We will generally only apply this in the case that $n \geq 1$, and in the case that $\pi_{<n}(\varphi)$ is an isomorphism.
\end{constr}

\begin{lem}\label{alg difference constrn gives map of EM type}
Suppose that the map $X \xra{\varphi} Y$ in $\locpiAlgtTE$ is an isomorphism on $\pi_{<n}$ for some $n \geq 1$.  Write $A = \pi_0 X \cong \pi_0 Y \in \Alg_\Phi(\tA)$ and $M = \pi_n \fib(\varphi) \in \Mod_A^\Phi(\tA)$.  Then, the map
\[ P_0^\alg (X) \xra{\delta_n(\varphi)} P_{n+1}^\alg (\pozalg(\varphi)) \]
is of type $\vec{K}_A(M,n+1)$.
\end{lem}

\begin{proof}
This follows from \Cref{assume BM excision}.
\end{proof}

\begin{cor}[{\cite[Proposition 2.5.13]{GH}}]\label{nth alg difference constrn gives functorial k-invts}
Let $X \xra{\varphi} Y$ be a map in $\locpiAlgtTE$.  Suppose that $\pi_* \fib(\varphi)$ is concentrated in degree $n$.  The the square
\[ \begin{tikzcd}[column sep=1.5cm]
X \arrow{r}{\tau_0^\alg} \arrow{d}[swap]{\varphi} & P_0^\alg (X) \arrow{d}{\delta_n(\varphi)} \\
Y \arrow{r} & P_{n+1}^\alg ( \pozalg(\varphi))
\end{tikzcd} \]
is a pullback in $\locpiAlgtTE$. \qed
\end{cor}

\begin{obs}
In the setting of \Cref{nth alg difference constrn gives functorial k-invts}, if additionally $X$ (and hence $Y$) is $n$-truncated, then we can identify the map $X \ra Y$ as $\tau_{\leq n}^\alg X \ra \tau_{\leq (n-1)}^\alg X$, and from here \Cref{alg difference constrn gives map of EM type} allows us to identify the pullback square of \Cref{nth alg difference constrn gives functorial k-invts} as
\[ \begin{tikzcd}
P_n^\alg X \arrow{r} \arrow{d}[swap]{\tau_{n-1}^\alg} & K_A \arrow{d} \\
P_{n-1}^\alg X \arrow{r} & K_A( M , n+1 )
\end{tikzcd} \]
(in which the right vertical map is of type $\vec{K}_A(M,n+1)$).  This is a functorial construction of k-invariants in $\locpiAlgtTE$.
\end{obs}


\begin{notn}\label{fix a simplicial algebra k}
We fix an object $k \in \locpiAlgtTE$.  We will generally work in its undercategory $\locpiAlgtTE_{k/}$; in particular, we will generally have fixed a map $k \ra A = \const(A)$.  Everything will take place in this undercategory, so that e.g.\! a morphism in $\locpiAlgtTE_{k/}$ of type $\vec{K}_A(M,n)$ will be understood to mean a commutative triangle
\[ \begin{tikzcd}[column sep=0.5cm]
& k \arrow{ld} \arrow{rd} \\
K_A \arrow{rr} & & K_A(M,n)
\end{tikzcd} \]
in $\locpiAlgtTE$ in which the left vertical arrow identifies with the fixed map.
\end{notn}


\begin{notn}
Suppose that $Y \in \locpiAlgtTE_{k/}$ is $(n-1)$-truncated for some $n \geq 1$, write $A = \pi_0 Y \in \Alg_\Phi(\tA)_{k/}$, and suppose $M \in \Mod_A^\Phi(\tA)$.  We write
\[ \ms{M}_k(Y \oplus (M,n)) \subset \locpiAlgtTE_{k/} \]
for the moduli space of those objects $X$ such that
\begin{itemize}
\item $X$ is $n$-truncated,
\item there exists an equivalence $P_{n-1}^\alg X \xra{\sim} Y$, and
\item there exists an isomorphism $\pi_n X \cong M$ via the resulting equivalence $\Mod_{\pi_0 X}^\Phi(\tA) \simeq \Mod_A^\Phi(\tA)$.
\end{itemize}
\end{notn}

\begin{notn}
In our moduli spaces, we will use the symbol $\loopra$ to denote a restriction to morphisms which are isomorphisms on homotopy groups in those dimensions for which \textit{both} the source and the target have nonvanishing homotopy. 
\end{notn}

\begin{prop}[{\cite[Theorem 2.5.16]{GH}}]\label{adding M in degree n is classified by a map to KAMn+1}
Suppose that $Y \in \locpiAlgtTE_{k/}$ is $(n-1)$-truncated for some $n \geq 1$, write $A = \pi_0 Y \in \Alg_\Phi(\tA)_{k/}$, and suppose $M \in \Mod_A^\Phi(\tA)$.  Then the functor
\[ X \mapsto \left( P_{n-1}^\alg(X) \ra P_{n+1}^\alg(\pozalg((\tau_{n-1}^\alg)_X)) \xla{\delta_n((\tau_{n-1}^\alg)_X)} P_0^\alg(X) \right) \]
determines an equivalence
\[ \ms{M}_k(Y \oplus (M,n)) \xra{\sim} \ms{M}_k( Y \loopra K_{A}(M,n+1) \loopla K_{A} ) \]
in $\S$.
\end{prop}

\begin{proof}
An inverse is provided by the pullback functor.
\end{proof}

\begin{notn}\label{notation for moduli spaces of alg EM objects}
For any $A \in \Alg_\Phi(\tA)_{k/}$, we write
\[ \ms{M}_{A/k} \subset \locpiAlgtTE_{k/} \]
for the moduli space of objects of type $K_{A/k}$.  For any $M \in \Mod_A^\Phi(\tA)$ and any $n \geq 1$, we write
\[ \ms{M}_{A/k}(M,n) \subset \Fun([1],\locpiAlgtTE_{k/}) \]
for the moduli space of morphisms of type $\vec{K}_{A/k}(M,n)$.
\end{notn}

\begin{notn}
It will be of auxiliary use to write 
\[ \ms{M}_{A/k}(M,0) \]
for the moduli space of pairs of an object $X \in \locpiAlgtTE$ and an abelian ($\infty$-)group object $Y \in \Ab(\locpiAlgtTE_{/X})$ in its overcategory which are in the image of $(A,M)$ under the derived right adjoint
\[ \Alg_\Phi(\tA)_{/A} \xra{\const} \locpiAlgtTE_{/A} \]
of the Quillen adjunction of \Cref{pi0 and const are a qadjn on tTE and Phi algs}.
\end{notn}

\begin{prop}\label{equivce of moduli spaces of alg EM maps}
Let $A \in \Alg_\Phi(\tA)_{k/}$, let $M \in \Mod_A^\Phi(\tA)$, and let $n \geq 0$.  Then the functor
\[ (X \ra Y) \mapsto {\lim}^{\locpiAlgtTE_{X/}} \left( \begin{tikzcd} & X \arrow{d} \\ X \arrow{r} & Y \end{tikzcd} \right) \]
defines an equivalence
\[ \ms{M}_{A/k}(M,n+1) \xra{\sim} \ms{M}_{A/k}(M,n) \]
in $\S$.
\end{prop}

\begin{proof}
For $n \geq 1$, an inverse is provided by the functor
\[ (Z \ra W) \mapsto \delta_n(W \ra P_0^\alg (W)) . \]
For $n=0$, an inverse is provided by the functor taking the pair
\[ \left( W \in \locpiAlgtTE \ , \ Z \in \Ab(\locpiAlgtTE_{/W}) \right) , \]
say with structure map $Z \xra{\varphi} W$, to the map
\[ K_{\pi_0 W} \ra K_{\pi_0 W} ( \ker(\pi_0(\varphi)) , 1) \]
(which is evidently of type $\vec{K}_A(M,1)$).
\end{proof}

\begin{prop}[{\cite[Lemma 2.5.18]{GH}}]\label{all homs into alg EM give multiple coh grps}
Let $A \in \Alg_\Phi(\tA)_{k/}$, let $M \in \Mod_A^\Phi(\tA)$, let $X \in \locpiAlgtTE_{k/}$, and let $n \geq 0$.  Then there exists a natural isomorphism
\[ [ X , K_A(M,n) ]_{\locpiAlgtTE_{k/}} \cong \underset{\hom_{\Alg_\Phi(\tA)_{k/}}(\pi_0 X , A)}{\coprod} H^n_{\tT_E}(X/k;M) \]
in $\Ab$ (where the implicit structure map $X \ra A = \const(A)$ in $\Alg_{\tT_E}(s\tA)_{k/}$ necessary for defining the cohomology of $X$ varies over the indexing set). \qed
\end{prop}

\begin{notn}
Given an $\infty$-category $\D$ and objects $d_1,d_2 \in \D$, we write $\hom^{\simeq}_\D(d_1,d_2) \subset \hom_\D(d_1,d_2)$ for the subspace of equivalences.  For any other sort of decoration denoting a certain property of a morphism, we use corresponding exponent notation to denote the subspace of the hom-space corresponding to morphisms having this property.
\end{notn}

\begin{notn}
For any $A \in \Alg_\Phi(\tA)_{k/}$, we write $\Aut_k(A) = \Aut_{\Alg_\Phi(\tA)_{k/}}(A)$.  Moreover, for any $M \in \Mod_A^\Phi(\tA)$, we write $\Aut_k(A,M)$ for the group of pairs
\[ \left( \varphi \in \Aut_k(A) \ , \ \psi \in \hom_{\Mod_A^\Phi(\tA)}^{\cong}(M , \varphi^*(M)) \right) . \]
\end{notn}

\begin{prop}[{\cite[Proposition 2.5.19(1)]{GH}}]\label{moduli of alg EM objects}
For any $A \in \Alg_\Phi(\tA)_{k/}$, we have an equivalence $\ms{M}_{A/k} \simeq B \Aut_k(A)$ in $\S$.
\end{prop}

\begin{proof}
This is the assertion that the canonical map
\[ \Aut_{\Alg_\Phi(\tA)_{k/}}(A) \ra \Aut_{\locpiAlgtTE_{k/}}(\const(A)) \]
induced by the functor
\[ \Alg_\Phi(\tA) \xra{\const} \locpiAlgtTE \]
is an equivalence, which follows from \Cref{pi0 and const are a qadjn on tTE and Phi algs} since it implies that this functor is a full inclusion.
\end{proof}

\begin{prop}[{\cite[Proposition 2.5.19(2)]{GH}}]\label{moduli of alg EM maps}
Suppose that $A \in \Alg_\Phi(\tA)_{k/}$ and that $M \in \Mod_A^\Phi(\tA)$.  Then for any $n \geq 0$ we have an equivalence $\ms{M}_{A/k}(M,n) \simeq B \Aut_k(A,M)$.
\end{prop}

\begin{proof}
This follows from combining \Cref{equivce of moduli spaces of alg EM maps} with the essentially definitional equivalence $\ms{M}_{A/k}(M,0) \simeq B\Aut_k(A,M)$.
\end{proof}


\begin{notn}
Given an object $X \in \locpiAlgtTE_{k/}$, we write
\[ \ms{M}_k(X) \subset \locpiAlgtTE_{k/} \]
for the full subgroupoid generated by it.
\end{notn}

\begin{lem}[{\cite[Proposition 2.5.22]{GH}}]\label{pb square with fiber coh spaces indexed by isos}
For any $X \in \locpiAlgtTE_{k/}$, there exists a canonical pullback square
\[ \begin{tikzcd}
\underset{\hom^{\cong}_{\Alg_\Phi(\tA)}(\pi_0 X, A)}{\coprod} \ms{H}^n_{\tT_E}(X/k;M) \arrow{r} \arrow{d} & \ms{M}_k(X \loopra K_A(M,n) \loopla A ) \arrow{d} \\
\pt_\S \arrow{r}[swap]{(X,\id_{(A,M)})} & \ms{M}_k(X) \times B \Aut_k(A,M) 
\end{tikzcd} \]
in $\S$.
\end{lem}

\begin{proof}
This is immediate from the definitions.
\end{proof}

\begin{notn}
We write
\[ \what{\ms{H}}^n_{\tT_E}(A/k;M) = \left( \ms{H}^n_{\tT_E}(A/k;M) \right)_{\Aut_k(A,M)} \in \S_* \]
for the based space of coinvariants of the canonical action of $\Aut_k(A,M)$ on $\ms{H}^n_{\tT_E}(A/k;M) \in \S_*$.
\end{notn}

\begin{cor}\label{identify moduli of two alg EM maps as htpy quotient of coh space}
There exists a canonical pullback square
\[ \begin{tikzcd}
\ms{H}^n_{\tT_E}(A/k;M) \arrow{r} \arrow{d} & \ms{M}_k(A \loopra K_A(M,n) \loopla A) \arrow{d} \\
\pt_\S \arrow{r} & B \Aut_k(A,M)
\end{tikzcd} \]
in $\S$, whose induced action of $\Aut_k(A,M)$ on $\ms{H}^n_{\tT_E}(A/k;M)$ is the natural one, and which induces an equivalence
\[ \ms{M}_k(A \loopra K_A(M,n) \loopla A) \simeq \what{\ms{H}}^n_{\tT_E}(A/k;M) \]
in $\S$.
\end{cor}

\begin{proof}
First of all, applying \Cref{pb square with fiber coh spaces indexed by isos} in the case that $X = A$ yields a pullback square
\[ \begin{tikzcd}
\underset{\hom^{\cong}_{\Alg_\Phi(\tA)}(A, A)}{\coprod} \ms{H}^n_{\tT_E}(A/k;M) \arrow{r} \arrow{d} & \ms{M}_k(A \loopra K_A(M,n) \loopla A ) \arrow{d} \\
\pt_\S \arrow{r}[swap]{(X,\id_{(A,M)})} & \ms{M}_k(A) \times B \Aut_k(A,M)
\end{tikzcd} \]
in $\S$.  By \Cref{moduli of alg EM objects}, we have $\ms{M}_k(A) \simeq B \Aut_k(A) = \Aut_{\Alg_\Phi(\tA)_{k/}}(A)$, and the action on the fibers is clearly the canonical one and is hence free on its path components.  Thus, pulling back along the map
\[ B\Aut_k(A,M) \simeq \{ A \} \times B\Aut_k(A,M) \ra \ms{M}_k(A) \times B\Aut_k(A,M) \]
yields a pullback square
\[ \begin{tikzcd}
\ms{H}^n_{\tT_E}(A/k;M) \arrow{r} \arrow{d} & \ms{M}_k(A \loopra K_A(M,n) \loopla A) \arrow{d} \\
\pt_\S \arrow{r}[swap]{\id_{(A,M)}} & B \Aut_k(A,M)
\end{tikzcd} \]
in $\S$.  The claim now follows readily from \cite[Proposition 2.1]{MIC-gr}.
\end{proof}

\section{Homotopical topology}
\label{section.homotopical.topology}

\subsection{Postnikov towers in topology}

We now study the homotopy theory of the $\infty$-category $\Alg_T(s\C)$ of simplicial $T$-algebras; we will mostly work in its localization $\locresAlgT$, but we will ultimately be interested in deducing results about its further localization $\locEAlgT$ (recall \Cref{going from locpi to locE is a left localization}).

\begin{defn}
For any $n \geq 0$, an object $X \in \locresAlgT$ is called \bit{$n$-truncated} if $\pi_{>n,\eps}^\natural X = 0$ for all $\eps \in \GE^\delta$.  Such objects form a full subcategory $\locresAlgT^{\leq n} \subset \locresAlgT$, and as $n$ varies these subcategories are evidently nested as
\[ \locresAlgT \hookla \cdots \hookla \locresAlgT^{\leq 1} \hookla \locresAlgT^{\leq 0} . \]
By presentability considerations, 
these inclusions admit left adjoints, and we denote the corresponding left localization adjunctions by
\[ P_n^\top : \locresAlgT \adjarr \locresAlgT^{\leq n} : \forget_n^\top . \]
We therefore obtain a tower of functors
\[ \id_{\locresAlgT} \ra \cdots \ra P_1^\top \ra P_0^\top . \]
We refer to its value on an object of $\locresAlgT$ as its \bit{Postnikov tower}.  We write
\[ \id_{\locresAlgT} \xra{\tau_n^\top} P_n^\top \]
for the natural transformation (or its for composite with $\forget_m^\top$ for any $m \geq 0$), which we refer to as the \bit{$n$-truncation map}.
\end{defn}

\begin{obs}
By a colimit argument, if $X \in \locresAlgT$ is $n$-truncated then $E_{\leq n,\star}^\natural X = 0$ as well.
\end{obs}

\subsection{Topological Eilenberg--Mac Lane objects}

We now define certain objects of $\locresAlgT$ which will represent the various functors ``apply $E_\star^\lw$, then take cohomology''.


\begin{defn}
Let $A \in \Alg_\Phi(\tA)$, let $M \in \Mod_A^\Phi(\tA)$, and let $n \geq 1$.
\begin{enumerate}

\item We say that an object $X \in \locresAlgT$ is of \bit{type $B_A$} if there exists a universal map $E_\star^\lw X \ra K_A$ inducing natural equivalences
\[ \hom_\locresAlgT(Z,X) \xra{\sim} \hom_\locpiAlgtTE(E_\star^\lw Z , K_A) \]
for all $Z \in \locresAlgT$.

\item We say that an object $Y \in \locresAlgT$ is of \bit{type $B_A(M,n)$} if there exists a universal map $E_\star^\lw Y \ra K_A(M,n)$ inducing natural equivalences
\[ \hom_\locresAlgT(Z,X) \xra{\sim} \hom_\locpiAlgtTE(E_\star^\lw Z , K_A(M,n)) \]
for all $Z \in \locresAlgT$.

\item We say that a map $X \ra Y$ in $\locresAlgT$ is of \bit{type $\vec{B}_A(M,n)$} if $X$ is of type $B_A$, $Y$ is of type $B_A(M,n)$ and the map $\pi_0 E_\star^\lw X \ra \pi_0 E_\star^\lw Y$ is an isomorphism in $\Alg_\Phi(\tA)$.

\end{enumerate}
We refer to objects of type $B_A$ and $B_A(M,n)$ collectively as \bit{topological Eilenberg--Mac Lane objects}, and to morphisms of type $\vec{K}_A(M,n)$ collectively as \bit{topological Eilenberg--Mac Lane morphisms}.
\end{defn}

\begin{lem}\label{existence of top EM objects and maps}
For any $A \in \Alg_\Phi(\tA)$, and $M \in \Mod_A^\Phi(\tA)$, and any $n \geq 1$, there exist objects of type $B_A$ and $B_A(M,n)$, and there exists a morphism of type $\vec{B}_A(M,n)$.
\end{lem}

\begin{proof}
By the presentability of $\locresAlgT$ (which follows from \Cref{identify localization of sNres} and the derived monadic adjunction underlying the monadic Quillen adjunction $\free_T \adj \forget_T$), this follows from \Cref{infty-categorical consequence of homotopical adaptedness}.
\end{proof}

\begin{notn}
Justified by \Cref{existence of top EM objects and maps}, we may simply write $B_A$ or $B_A(M,n)$ for convenience when referring to a topological Eilenberg--Mac Lane object of the indicated type.
\end{notn}

\begin{obs}\label{natural and classical htpy of topological EM objects}
If $X \in \locresAlgT$ is an object of type $B_A$, it follows immediately that $\pi_{0,\star}^\natural X \cong \Yo^{\pi_0 E_\star^\lw}(A)$ in $\PSd(T(\GE))$ and that $\pi_{>0,*}^\natural X = 0$.  By the spiral exact sequence, it follows that
\[ \pi_i \pi_\star X \cong \left\{ \begin{array}{ll} \Yo^E(A) , & i = 0 \\ \Yo^E(\Omega A) , & i = 2 \\ 0 , & i \notin \{0,2 \} . \end{array} \right. \]
For convenience, we simply write $\pi_* \pi_\star X \cong \Yo^E(A) \times \Yo^E(\Omega A)[2]$.

Now, suppose that $X \ra Y$ is a map of type $\vec{B}_A(M,n)$.  It follows from the definition of an object of type $B_A(M,n)$ that $\pi_{0,S^\eps}^\natural Y \cong \Yo^E(A)$ in $\PSd(T(\GE))$ and that for $i \geq 1$,
\[ \pi_{i,\star}^\natural Y \cong \left\{ \begin{array}{ll} \Yo^E(M) , & i = n \\ 0 , & i \not= n
\end{array} \right. \]
in $\Mod_A^\Phi(\tA)$.  Then, note further that if $X \ra Y$ is a map of type $\vec{B}_A(M,n)$, then the composite $X \ra Y \ra P_0^\top (Y)$ is an equivalence; combining this with the spiral exact sequence yields that $\pi_* \pi_\star Y \cong \pi_* \pi_\star X \times \Yo^E(M)[n] \times \Yo^E(\Omega M)[n+2]$.
\end{obs}

\subsection{Moduli spaces in topology}

We begin by mimicking \Cref{alg difference constrn}.

\begin{constr}
Let $X \xra{\varphi} Y$ be a map in $\locresAlgT$, and write
\[ \poztop(\varphi) = Y \underset{X}{\coprod} P_0^\top X = \colim \left( \begin{tikzcd} X \arrow{r}{\tau_0^\top} \arrow{d}[swap]{\varphi} & P_0^\top(X) \\ Y \end{tikzcd} \right) \]
for the indicated pushout.  For any $n \geq 0$ we obtain a commutative diagram
\[ \begin{tikzcd}
X \arrow{r}{\tau_0^\top} \arrow{d}[swap]{\varphi} & P_0^\top (X) \arrow{d} \arrow{rd}{\delta_n(\varphi)} \\
Y \arrow{r} & \poztop(\varphi) \arrow{r}[swap]{\tau_{n+1}^\top} & P_{n+1}^\top(\poztop(\varphi))
\end{tikzcd} \]
in $\locresAlgT$, and we refer to the map $\delta_n(\varphi)$ as the \bit{$n\th$ difference construction} on the map $\varphi$.  This defines an augmented endofunctor on $\Fun([1],\locresAlgT)$.  We will generally only apply this in the case that $n \geq 1$, and in the case that $\pi_{<n,\star}^\natural(\varphi)$ is an isomorphism.
\end{constr}

We now employ our assumption that $T$ is homotopically adapted to $E$, which provides a fundamental link between our computations in homotopical topology and homotopical algebra.

\begin{prop}\label{homotopical adaptedness implies things for natural E-homology of simplicial T-algebras}
Let $X \xra{\varphi} Y$ be a map in $\locresAlgT$, let $n \geq 1$, and suppose that $E_{< n , \star}^\natural(\varphi)$ is an isomorphism and that $E_{n,\star}^\natural(\varphi)$ is surjective.  Write $A = E_{0,\star}^\natural X \cong E_{0,\star}^\natural Y$ in $\Alg_\Phi(\tA)$, and write $M = \fib(E_{n,\star}^\natural(\varphi)) \in \tA$.
\begin{enumerate}
\item We can canonically consider $M \in \Mod_A^\Phi(\tA)$.
\item The map $\delta_n(\varphi)$ becomes equivalent to a morphism of type $\vec{B}_A(M,n)$ under the localization functor $\leftloc_{E_\star^\lw} : \locresAlgT \ra \locEAlgT$.
\item\label{fiber concentrated in degree n+1 gives E-local pullback} If $\pi_{i,\star}^\natural(\fib(\varphi)) = 0$ for $i \not= n+1$, then the square
\[ \begin{tikzcd}
X \arrow{r}{\tau_0^\top} \arrow{d}[swap]{\varphi} & P_0^\top(X) \arrow{d}{\delta_n(\varphi)} \\
Y \arrow{r} & P_{n+1}^\top(\poztop(\varphi))
\end{tikzcd} \]
becomes a pullback under the localization functor $\leftloc_{E_\star^\lw} : \locresAlgT \ra \locEAlgT$.
\end{enumerate}
\end{prop}

\begin{proof}
It follows from \Cref{infty-categorical consequence of homotopical adaptedness} that the functor
\[ \locresAlgT \xra{E_\star^\lw} \locpiAlgtTE \]
preserves pushouts.  Thus, the square
\[ \begin{tikzcd}[column sep=1.5cm]
E_\star^\lw X \arrow{r}{E_\star^\lw(\tau_0^\top)} \arrow{d}[swap]{E_\star^\lw(\varphi)} & E_\star^\lw(P_0^\top(X)) \arrow{d} \\
E_\star^\lw Y \arrow{r} & E_\star^\lw(\poztop(\varphi))
\end{tikzcd} \]
is a pushout in $\locpiAlgtTE$.  From here, the proof is essentially identical to that of \cite[Proposition 3.2.9]{GH}.
\end{proof}

In order to work in a relative setting, we fix the following.

\begin{notn}\label{fix base object Y}
We assume we are given an object $Y \in \Alg_\oO(\C)$ equipped with an isomorphism $E_\star^\lw Y \cong k$ in $\Alg_\Phi(\tA)$ for some chosen object $k \in \Alg_\Phi(\tA)$ (specialized via the derived right adjoint $\Alg_\Phi(\tA) \xra{\const} \locpiAlgtTE$ from our previous assumption from \Cref{fix a simplicial algebra k} that $k \in \locpiAlgtTE$).  A map $k \ra A$ in $\Alg_\Phi(\tA)$ gives rise to a composite
\[ E_\star^\lw \const(Y) \xra{\cong} k \ra A \]
in $\Alg_\Phi(\tA)$, via which for any choice of topological Eilenberg--Mac Lane object $B_A$ we obtain a canonical map $\const(Y) \ra B_A$.  We will simply write $Y = \const(Y) \in \locresAlgT$, and we will work in $\locresAlgT_{Y/\!/B_A}$.
\end{notn}

\begin{obs}
Fix any morphism $B_A \ra B_A(M,n)$ in $\locresAlgT$ of type $\vec{B}_A(M,n)$.  From \Cref{natural and classical htpy of topological EM objects} \and \Cref{fix base object Y}, we obtain a sequence of composable morphisms
\[ Y \ra B_A \ra B_A(M,n) \ra B_A \]
(in which the composite of all but the first map is an equivalence).  For any $X \in \locresAlgT_{Y/\!/B_A}$ and as soon as $n \geq 2$, we immediately obtain equivalences
\[ \hom_{\locresAlgT_{Y/}}(X,B_A) \xra{\sim} \hom_{\Alg_\Phi(\tA)_{k/}}(\pi_0 E_\star^\lw X , A ) \]
and
\[ \hom_{\locresAlgT_{Y/\!/B_A}}(X,B_A(M,n)) \xra{\sim} \ms{H}^n_{\tT_E}(E_\star^\lw (X) / k; M ) \]
in $\S_*$.
\end{obs}

\begin{notn}
We write $\ms{M}_Y(A) \subset \locresAlgT_{Y/}$ for the moduli space of objects $Y \ra X$ such that $X$ is of type $B_A$ and moreover the map $E_0^\lw Y \ra E_0^\lw X$ is equivalent to the map $k \ra A$ in $\locpiAlgtTE$.  Moreover, we write $\ms{M}_{A/Y}(M,n) \subset \locresAlgT_{Y/}$ for the moduli space of morphisms $Z \ra W$ of type $\vec{B}_A(M,n)$ under $Y$ such that $(Y \ra Z) \in \ms{M}_Y(A)$.
\end{notn}

\begin{prop}\label{determine moduli spaces of topological EM objects and morphisms}
The functor
\[ X \mapsto P_0^\alg E_\star^\lw (X) \]
defines an equivalence
\[ \ms{M}_Y(A) \xra{\sim} \ms{M}_k(A) , \]
and the functor
\[ \varphi \mapsto \delta_{n-1}(E_\star(\varphi)) \]
defines an equivalence
\[ \ms{M}_{A/Y}(M,n) \xra{\sim} \ms{M}_{A/k}(M,n) \simeq B \Aut_k(A,M) . \]
\end{prop}

\begin{proof}
These assertions both follow immediately from the functors that topological Eilenberg--Mac Lane objects are defined to represent, just as in the proof of \cite[Proposition 3.2.17]{GH}.
\end{proof}

\section{Decomposition of moduli spaces}
\label{section.decomp.of.mod.spaces}

\subsection{Realizations and $n$-stages}

We finally come to our main theorems: these provide an inductive procedure for understanding our moduli space of ultimate interest, which we begin by introducing.

\begin{defn}
With respect to
\begin{itemizesmall}
\item our fixed base object $Y \in \Alg_\oO(\C)$,
\item our chosen morphism $k \ra A$ in $\Alg_\Phi(\tA)$, and
\item our chosen isomorphism $E_\star Y \cong k$ in $\Alg_\Phi(\tA)$,
\end{itemizesmall}
we define a \bit{realization} to be an object $(Y \xra{\varphi} X) \in \LEAlgO_{Y/}$ such that there exists an isomorphism $E_\star X \cong A$ in $\Alg_\Phi(\tA)_{k/}$.  We write 
\[ \ms{M}_{A/Y} \subset \LEAlgO_{Y/} \]
for the moduli space of realizations (and $E_\star$-equivalences between them).
\end{defn}

Before diving in, we provide a bit of big-picture intuition.

\begin{rem}\label{big-picture intuition for obstruction theory}
Given a simplicial $T$-algebra $Z$, a good way to control $E_\star |Z|$ is to control its spiral spectral sequence.  More to the point, the easiest way to ensure that $|Z|$ be a realization is to demand that $\Etwo = \pi_* E_\star^\lw Z \cong \pi_0 E_\star^\lw Z \cong A$, so that the spectral sequence collapses immediately.

However, it is not so straightforward to obtain such an object or understand its automorphisms: the $\Etwo$ page consists of \textit{natural} $E$-homology groups, but it is the \textit{classical} $E$-homology groups that are more closely connected to the actual homotopy theory of the $\infty$-category $\locresAlgT$.

Luckily, however, we have a tool that relates these two types of $E$-homology groups: the localized spiral exact sequence.  As it is one-third classical and two-thirds natural, it allows us to exert control over the classical $E$-homology groups by manipulating the natural $E$-homology groups.

Thus, our method will be to attempt to interpolate one stage at a time from
\begin{itemizesmall}
\item objects which are easy to understand (read: have controlled natural $E$-homology) but do not have the correct $\Etwo$ pages (read: have the wrong classical $E$-homology), towards
\item objects which are somewhat more difficult to understand (read: have more complicated natural $E$-homology) but have $\Etwo$ pages which are closer and closer to collapsing at $A$ (read: their classical $E$-homology is equivalent to $A$ itself (concentrated in degree $0$) in an increasingly large range).
\end{itemizesmall}
Of course, such interpolation will not always be possible, but in the course of our attempt we will discover the precise cohomological obstructions to their possibility.
\end{rem}

We now define certain objects of $\locresAlgT$ which, via geometric realization, provide approximations to realizations.

\begin{defn}\label{define n-stages}
For $0 \leq n \leq \infty$, we say that an object $Z \in \locresAlgT_{Y/}$ is an \bit{$n$-stage} if the following conditions hold:
\begin{enumerate}
\item\label{E0 is A} there exists an isomorphism $\pi_0 E_\star^\lw Z \cong A$ in $\Alg_\Phi(\tA)_{k/}$;
\item\label{no high-diml natural htpy} $\pi_{>n,\star}^\natural Z = 0$; and
\item\label{no low-diml classical E-homology} $\pi_i E_\star^\lw Z = 0$ for $1 \leq i \leq n+1$.
\end{enumerate}
We write
\[ \ms{M}_n(A/Y) \subset \locEAlgT_{Y/} \]
for the moduli space of $n$-stages (and $E_\star$-equivalences between them).
\end{defn}

\begin{obs}\label{natural and classical E-homology of an n-stage}
Suppose that $Z \in \ms{M}_n(A/Y)$.  By condition \ref{no low-diml classical E-homology}, the tail end of the localized spiral exact sequence degenerates into a sequence of isomorphisms.  By induction, this implies that $E_{i,\star}^\natural Z \cong \Omega^i A$ for all $i \leq n$: the base case of $i=0$ follows from condition \ref{E0 is A} and \Cref{iso of two types of 0th E-homology is as Phi-algs}.  Then, after a colimit argument, condition \ref{no high-diml natural htpy} implies that we have an isomorphism $\pi_{n+2} E_\star^\lw Z \xra{\cong} \Omega ( E_{n,\star}^\natural Z)$ and that $\pi_{>n+2} E_\star^\lw Z = 0$.  The table of \Cref{table of both E-hlgy gps for an n-stage} summarizes these computations.
\begin{figure}[h]
\[ \begin{array}{c||c|c|c|c|c|c|c|c|c|c}
i & 0 & 1 & 2 & \cdots & n-1 & n & n+1 & n+2 & n+3 & \cdots \\ \hline \hline
\pi_i E_\star^\lw Z & A & 0 & 0 & \cdots & 0 & 0 & 0 & \Omega^{n+1} A & 0 & \cdots \\ \hline
E_{i,\star}^\natural Z & A & \Omega A & \Omega^2 A & \cdots & \Omega^{n-1} A & \Omega^n A & 0 & 0 & 0 & \cdots
\end{array} \]
\caption{The classical and natural $E$-homology groups of an $n$-stage $Z \in \ms{M}_n(A/Y)$.}
\label{table of both E-hlgy gps for an n-stage}
\end{figure}
Moreover, the same argument shows that if $n= \infty$ then $E_{i,\star}^\natural Z \cong \Omega^i A$ for all $i \geq 0$ and that $\pi_* E_\star^\lw Z \cong \pi_0 E_\star^\lw Z \cong A$.
\end{obs}

We now provide the connection between realizations and $n$-stages.

\begin{thm}\label{infty-stages give realizations}
Geometric realization induces an equivalence
\[ \ms{M}_\infty(A/Y) \xra{\sim} \ms{M}_{A/Y} . \]
\end{thm}

\begin{proof}
The adjunction $|{-}| : \Alg_T(s\C) \adjarr \Alg_\oO(\C) : \const$ evidently descends (or perhaps rather restricts) to an adjunction $|{-}| : \locEAlgT \adjarr \LEAlgO : \const$ by the universal property of localization.  In turn, the spiral spectral sequence implies that (after taking undercategories of $Y$) this latter adjunction restricts to give the desired equivalence.
\end{proof}

\begin{rem}
Note that we do \textit{not} generally have a pullback square
\[ \begin{tikzcd}
\ms{M}_\infty(A/Y) \arrow{r} \arrow[hook]{d} & \ms{M}_{A/Y} \arrow[hook]{d} \\
\locEAlgT_{Y/} \arrow{r}[swap]{|{-}|} & \LEAlgO_{Y/} .
\end{tikzcd} \]
Rather, as alluded to in \Cref{big-picture intuition for obstruction theory}, an $\infty$-stage is exactly an object whose spiral spectral sequence has $\Etwo = \pi_* E_\star^\lw X \cong \pi_0 E_\star^\lw \cong A$, so that in particular it collapses immediately.
\end{rem}

\begin{thm}\label{topological truncations move between stages}
For any $0 \leq n \leq m \leq \infty$, the $n$-truncation functor
\[ \locresAlgT \xra{P_n^\top} \locresAlgT \]
induces a map
\[ \ms{M}_m(A/Y) \ra \ms{M}_n(A/Y) , \]
and these assemble to give an equivalence
\[ \ms{M}_\infty(A/Y) \xra{\sim} \lim \left( \cdots \xra{P_2^\top} \ms{M}_2(A/Y) \xra{P_1^\top} \ms{M}_1(A/Y) \xra{P_0^\top} \ms{M}_0(A/Y) \right) . \]
\end{thm}

\begin{proof}
First of all, it is immediate from the localized spiral exact sequence that the $n$-truncation of an $m$-stage is an $n$-stage.  From here, the asserted equivalence follows from an ($\infty$-categorical but otherwise) identical argument to that of \cite[4.6]{DKClass}.
\end{proof}

\begin{thm}\label{thm moduli space of 0-stages is algebraic}
The functor
\[ \locEAlgT \xra{\pi_0 E_\star^\lw} \Alg_\Phi(\tA) \]
induces an equivalence
\[ \ms{M}_0(A/Y) \xra{\sim} \ms{M}_{A/k} . \]
\end{thm}

\begin{proof}
Inspection of the definitions reveals an equivalence $\ms{M}_0(A/Y) \simeq \ms{M}_Y(A)$ with the moduli space of objects under $Y$ of type $B_A$, and from here the claim follows from \Cref{determine moduli spaces of topological EM objects and morphisms}.
\end{proof}

\subsection{Climbing the tower}

We now come to the essential result, which explains how to move up the tower of moduli spaces.

\begin{thm}\label{pullback square to climb tower}
For any $n \geq 1$, there is a natural pullback square
\[ \begin{tikzcd}
\ms{M}_n(A/Y) \arrow{r} \arrow{d}[swap]{P_{n-1}^\top} & B \Aut_k(A,\Omega^n A) \arrow{d} \\
\ms{M}_{n-1}(A/Y) \arrow{r} & \what{\ms{H}}^{n+2}_{\tT_E}(A/k;\Omega^n A)
\end{tikzcd} \]
in $\S$.
\end{thm}

In order to prove this, we will first develop an understanding of the object-by-object passage between $(n-1)$-stages and $n$-stages, and then we will analyze how this behaves in families.

\begin{obs}\label{topological EM objects are local}
Directly from the definitions, topological Eilenberg--Mac Lane objects are local with respect to the left localization adjunction $\leftloc_{E_\star^\lw} : \locresAlgT \adjarr \locEAlgT : \forget_{E_\star^\lw}$.  Nevertheless, we will often keep the localization functor in the notation for clarity.
\end{obs}

\begin{obs}\label{try to build n-stage from (n-1)-stage}
Suppose first that $Z \in \ms{M}_n(A/Y)$.  Then $P_{n-1}^\top(Z) \in \ms{M}_{n-1}(A/Y)$ by \Cref{topological truncations move between stages}, and moreover \Cref{homotopical adaptedness implies things for natural E-homology of simplicial T-algebras}\ref{fiber concentrated in degree n+1 gives E-local pullback} implies that we have a pullback square
\[ \begin{tikzcd}
\leftloc_{E_\star^\lw} (Z) \arrow{r} \arrow{d}[swap]{\leftloc_{E_\star^\lw}(\tau_{n-1}^\top)} & \leftloc_{E_\star^\lw}(B_A) \arrow{d} \\
\leftloc_{E_\star^\lw}(P_{n-1}^\top(Z)) \arrow{r} & \leftloc_{E_\star^\lw}(B_A(\Omega^n A , n+1))
\end{tikzcd} \]
in $\locEAlgT$.

Let us attempt to reverse this process.  Suppose that $W \in \ms{M}_{n-1}(A/Y)$, and suppose that we form a pullback
\[ \begin{tikzcd}
\leftloc_{E_\star^\lw}(\widetilde{W}) \arrow{r} \arrow{d} & \leftloc_{E_\star^\lw}(B_A) \arrow{d} \\
\leftloc_{E_\star^\lw}(W) \arrow{r}[swap]{\varphi} & \leftloc_{E_\star^\lw}(B_A(\Omega^n A,n+1))
\end{tikzcd} \]
in $\locEAlgT$.  Then, $\leftloc_{E_\star^\lw}(\widetilde{W}) \in \ms{M}_n(A/Y)$ if and only if the induced composite
\[ E_\star^\lw W \xra{E_\star(\varphi)} E_\star^\lw(B_A(\Omega^n A , n+1)) \ra K_A(\Omega^n A , n+1) \]
with the universal map is an equivalence in $\locpiAlgtTE$: this follows from the long exact sequence in classical $E$-homology induced by a pullback square.
\end{obs}

\begin{obs}\label{obstrns live in AQcoh}
We can interpret the conclusion of \Cref{try to build n-stage from (n-1)-stage} as follows.  By \Cref{natural and classical E-homology of an n-stage}, the object $E_\star^\lw W \in \locpiAlgtTE$ has homotopy concentrated in degrees $0$ and $n+1$ and moreover $P_n^\alg (E_\star^\lw W) \simeq A$.  By \Cref{adding M in degree n is classified by a map to KAMn+1}, this object therefore corresponds to a unique pullback square
\[ \begin{tikzcd}
E_\star^\lw W \arrow{r} \arrow{d} & K_A \arrow{d} \\
A \arrow{r}[swap]{\chi} & K_A(\Omega^n A,n+2)
\end{tikzcd} \]
in $\locpiAlgtTE$.

Recall from \Cref{alg EM spectrum} that we have a pullback square
\[ \begin{tikzcd}
K_A(\Omega^n A, n+1) \arrow{r} \arrow{d} & K_A \arrow{d} \\
K_A \arrow{r} & K_A(\Omega^n A,n+2)
\end{tikzcd} \]
in $\locpiAlgtTE$.  Now, we claim that there exists an equivalence $E_\star^\lw W \xra{\sim} K_A(\Omega^n A,n+1)$ in $\locpiAlgtTE$ if and only if $\chi$ represents the zero element $0 \in H^{n+2}_{\tT_E}(A/k;\Omega^n A)$.
\begin{itemize}
\item Indeed, if $[\chi] = 0$, then the existence of an equivalence is manifest.
\item Conversely, if such an equivalence exists, then by \Cref{adding M in degree n is classified by a map to KAMn+1} there exists an equivalence between these two pullback squares, implying that $[\chi]=0$.
\end{itemize}
Thus, the obstructions to a given $(n-1)$-stage lifting to an $n$-stage are given by elements of $H^{n+2}_{\tT_E}(A/k;\Omega^n A)$.  In particular, if this group vanishes then \textit{every} $(n-1)$-stage lifts to an $n$-stage.
\end{obs}

We now provide the key piece of input to the proof of \Cref{pullback square to climb tower}: in effect, we work with $\ms{M}_{n-1}(A/Y)$ one path component at a time.

\begin{notn}
For any $Z \in \ms{M}_{n-1}(A/Y)$, we write $\ms{M}_{n/Z}(A/Y) \subset \ms{M}_n(A/Y)$ for the subspace of those $n$-stages $W \in \ms{M}_n(A/Y)$ such that there exists an equivalence $P_{n-1}(W) \simeq Z$ in $\locEAlgT_{Y/}$.
\end{notn}

\begin{obs}\label{an (n-1)-stage lifts iff its E-homology is equivalent to an alg EM obj}
Note that the space $\ms{M}_{n/Z}(A/Y)$ may well be empty; indeed, by \Cref{obstrns live in AQcoh} it will be empty if and only if $\ms{M}_k(E_\star^\lw Z \loopra K_A(\Omega^n A,n+1))$ is empty.
\end{obs}

\begin{notn}
For any $Z \in \ms{M}_{n-1}(A/Y)$, we write $Z \xra{\uEeq} B_A(\Omega^n A,n)$ for a morphism in $\locresAlgT_{Y/}$ which classifies an equivalence $E_\star^\lw Z \xra{\sim} K_A(\Omega^n A,n)$ in $\locpiAlgtTE_{k/}$.
\end{notn}

\begin{lem}\label{pullback square extending a given (n-1)-stage to an n-stage}
Suppose that $Z \in \ms{M}_{n-1}(A/Y)$ for some $n \geq 1$.  Then there is a natural pullback square
\[ \begin{tikzcd}[column sep=1.5cm, row sep=1.5cm]
\ms{M}_{n/Z}(A/Y) \arrow{r} \arrow{d}[swap]{P_{n-1}^\top} & \ms{M}_k(E_\star^\lw Z \loopra K_A(\Omega^n A,n+1) \loopla K_A) \arrow{d} \\
\ms{M}_Y(Z) \arrow{r}[swap]{E_\star^\lw} & \ms{M}_k(E_\star^\lw Z)
\end{tikzcd} \]
in $\S$.
\end{lem}

\begin{proof}
The difference construction provides a map $\ms{M}_{n/Z}(A/Y) \ra \ms{M}_Y(Z \xra{\uEeq} B_A(\Omega^n A,n+1) \loopla B_A)$, which is an equivalence by \Cref{try to build n-stage from (n-1)-stage}.  Thus we obtain a commutative diagram
\[ \begin{tikzcd}[column sep=0.5cm, row sep=1.5cm]
\ms{M}_{n/Z}(A/Y) \arrow{r}{\sim} \arrow{d}[swap]{P_{n-1}^\top} & \ms{M}_Y(Z \xra{\uEeq} B_A(\Omega^n A,n+1) \loopla B_A) \arrow{d} \arrow{r} & \ms{M}_k(E_\star^\lw Z \loopra K_A(\Omega^n A,n+1) \loopla K_A) \arrow{d} \\
\ms{M}_Y(Z) \arrow{r}[swap]{\sim} & \ms{M}_Y(Z) \arrow{r}[swap]{E_\star^\lw} & \ms{M}_k(E_\star^\lw Z)
\end{tikzcd} \]
in $\S$, in which
\begin{itemize}
\item the right square is obtained by applying $E_\star^\lw$ and using the universal characterization of topological Eilenberg--Mac Lane objects,
\item the left square is tautologically a pullback, and
\item our goal is to show that the outer rectangle is a pullback;
\end{itemize}
thus, it suffices to show that the right square is a pullback.

In the right square, both downwards maps are obtained by forgetting certain data: a morphism of type $\vec{B}_A(\Omega^n A,n+1)$ on the left, and a morphism of type $\vec{K}_A(\Omega^n A,n+1)$ on the right.  Thus, it is convenient to use the equivalence $\ms{M}_{A/Y}(\Omega^n A,n+1) \xra{\sim} \ms{M}_{A/k}(\Omega^n A,n+1)$ of \Cref{determine moduli spaces of topological EM objects and morphisms} (between the moduli spaces of such Eilenberg--Mac Lane morphisms) to obtain a larger commutative square
\[ \begin{tikzcd}[column sep=1.5cm, row sep=1.5cm]
\ms{M}_Y(Z \xra{\uEeq} B_A(\Omega^n A,n+1) \loopla B_A) \arrow{r} \arrow{d} & \ms{M}_k(E_\star^\lw Z \loopra K_A(\Omega^n A,n+1) \loopla K_A) \arrow{d} \\
\ms{M}_Y(Z) \times \ms{M}_{A/Y}(\Omega^n A,n+1) \arrow{r} & \ms{M}_k(E_\star^\lw Z) \times \ms{M}_{A/k}(\Omega^n A,n+1)
\end{tikzcd} \]
which it then suffices to show is a pullback.

Now, observe that both spaces on the bottom row are connected (by definition and by Propositions \ref{moduli of alg EM maps} \and \ref{determine moduli spaces of topological EM objects and morphisms}).  So for any basepoint of $\ms{M}_Y(Z) \times \ms{M}_{A/Y}(\Omega^n A,n+1)$, it suffices to check that the induced map on fibers is an equivalence.  Unwinding the definitions, we see that this is the map
\[ \hom^{\uEeq}_{\locEAlgT}(Z,B_A(\Omega^n A,n+1)) \ra \hom^{\simeq}_{\locpiAlgtTE}(E_\star^\lw Z,K_A(\Omega^nA,n+1)) . \]
As $\locEAlgT \subset \locresAlgT$ is a full subcategory, we see that this is by definition an equivalence of subspaces of the equivalence
\[ \hom_{\locresAlgT}(Z,B_A(\Omega^nA,n+1)) \xra{\sim} \hom_{\locpiAlgtTE}(E_\star^\lw Z,K_A(\Omega^nA,n+1)) \]
characterizing the object $B_A(\Omega^nA,n+1) \in \locresAlgT$.
\end{proof}

We can now prove our main decomposition theorem.

\begin{proof}[Proof of \Cref{pullback square to climb tower}]
We begin with the commutative square
\[ \begin{tikzcd}
\ms{M}_k(K_A(\Omega^n A,n+1) \loopla K_A) \arrow{r}{\sim} \arrow{d} & \ms{M}_k(K_A(\Omega^n,n+2) \loopla K_A) \arrow{d} \\
\ms{M}_k(K_A \oplus (\Omega^n A,n+1)) \arrow{r}[swap]{\sim} & \ms{M}_k(K_A \loopra K_A(\Omega^n A,n+2) \loopla K_A)
\end{tikzcd} \]
in $\S$, in which
\begin{itemizesmall}
\item the upper horizontal map is (the inverse of) the equivalence of \Cref{equivce of moduli spaces of alg EM maps},
\item the left vertical map is forgetful,
\item the right vertical map repeats the given morphism,
\item the lower horizontal map is the equivalence of \Cref{adding M in degree n is classified by a map to KAMn+1}.
\end{itemizesmall}
This is tautologically a pullback square.

Now, suppose that $Z \in \ms{M}_{n-1}(A/Y)$.  We claim that there exists a pullback square
\[ \begin{tikzcd}
\ms{M}_{n/Z}(A/Y) \arrow{r} \arrow{d} & \ms{M}_k(K_A(\Omega^n A,n+2) \loopla K_A) \arrow{d} \\
\ms{M}_Y(Z) \arrow{r} & \ms{M}_k(K_A \loopra K_A(\Omega^n A,n+2) \loopla K_A)
\end{tikzcd} \]
in $\S$.  To see this, we separate the argument into two cases, depending on whether or not there exists an equivalence $E_\star^\lw Z \xra{\sim} K_A(\Omega^n A,n+1)$ in $\locpiAlgtTE$.
\begin{itemize}

\item Suppose that no such equivalence exists.  Then $\ms{M}_{n/Z}(A/Y)$ is empty by \Cref{an (n-1)-stage lifts iff its E-homology is equivalent to an alg EM obj}.  In this case, the subspace $\ms{M}_k(E_\star^\lw Z) \subset \ms{M}_k(K_A \oplus (\Omega^n A,n+1))$ is not in the image of the left vertical map of our original tautological pullback square.  These facts imply that the above square is indeed (equally tautologically) a pullback.

\item Suppose that such an equivalence exists.  In this case, we obtain an evident forgetful equivalence
\[ \ms{M}_k(E_\star^\lw Z \loopra K_A(\Omega^n A,n+1) \loopla K_A) \xra{\sim} \ms{M}_k(K_A(\Omega^nA,n+1) \loopla K_A) \]
in $\S$, which reduces the pullback square of \Cref{pullback square extending a given (n-1)-stage to an n-stage} to a pullback square
\[ \begin{tikzcd}
\ms{M}_{n/Z}(A/Y) \arrow{r} \arrow{d} & \ms{M}_k(K_A(\Omega^n A,n+1) \loopla K_A) \arrow{d} \\
\ms{M}_Y(Z) \arrow{r} & \ms{M}_k(K_A(\Omega^n A,n+1)) .
\end{tikzcd} \]
The right vertical arrow of this pullback square includes as a subobject of the left vertical arrow of our original tautological pullback square, yielding the claim.
\end{itemize}

Now, assembling this pullback square over all $Z \in \ms{M}_{n-1}(A/Y)$, we obtain a pullback square
\[ \begin{tikzcd}
\ms{M}_n(A/Y) \arrow{r} \arrow{d} & \ms{M}_k(K_A(\Omega^n,n+2) \loopla K_A) \arrow{d} \\
\ms{M}_{n-1}(A/Y) \arrow{r} & \ms{M}_k(K_A \loopra K_A(\Omega^nA,n+2) \loopla K_A) .
\end{tikzcd} \]
From here, the equivalence
\[ \ms{M}_k(K_A(\Omega^n A,n+2) \loopla K_A) = \ms{M}_{A/k}(\Omega^n A , n+2) \simeq B \Aut_k(A,\Omega^n A) \]
of \Cref{moduli of alg EM maps} and the equivalence
\[ \ms{M}_k(K_A \loopra K_A(\Omega^n A,n+2) \loopla K_A) \simeq \what{\ms{H}}^{n+2}_{\tT_E}(A/k;\Omega^n A) \]
of \Cref{identify moduli of two alg EM maps as htpy quotient of coh space} allow us to rewrite this as a pullback square
\[ \begin{tikzcd}
\ms{M}_n(A/Y) \arrow{r} \arrow{d} & B \Aut_k(A,\Omega^n A) \arrow{d} \\
\ms{M}_{n-1}(A/Y) \arrow{r} & \what{H}^{n+2}_{\tT_E}(A/k;\Omega^n A) ,
\end{tikzcd} \]
which completes the proof.
\end{proof}

\bibliographystyle{amsalpha}
\bibliography{GHOsT}{}

\end{document}